\documentclass[12pt]{amsart}
\usepackage{amscd,amssymb,amsthm,amsmath,amssymb,mathrsfs,enumerate,url,hyperref}
\usepackage{graphicx}
\usepackage[matrix,arrow,curve]{xy}
\usepackage[margin=2cm]{geometry}
\usepackage[normalem]{ulem}
\usepackage{verbatim}

\newtheorem{theorem}{Theorem}

\newtheorem*{theorem*}{Theorem}
\newtheorem{proposition}[theorem]{Proposition}
\newtheorem*{proposition*}{Proposition}
\newtheorem{lemma}[theorem]{Lemma}
\newtheorem*{lemma*}{Lemma}
\newtheorem{corollary}[theorem]{Corollary}
\newtheorem*{corollary*}{Corollary}

\newtheorem*{problem*}{Problem}

\theoremstyle{definition}

\newtheorem*{construction*}{Construction}
\newtheorem*{example*}{Example}

\newtheorem*{observation*}{Observation}

\theoremstyle{remark}
\newtheorem{remark}[theorem]{Remark}
\newtheorem*{remark*}{Remark}

\makeatletter\@addtoreset{equation}{section} \makeatother

\author{Ivan Cheltsov, Igor Krylov, Sione Ma'u}

\thanks{Throughout this paper, all varieties are assumed to be complex and projective.}

\title{$G$-birationally rigid cubic threefolds}

\begin{document}

\begin{abstract}
We classify pairs $(X,G)$ consisting of a (possibly singular) cubic threefold $X\subset\mathbb{P}^4$ and a finite subgroup $G\subset\mathrm{Aut}(X)$ such that $X$ is $G$-birationally rigid, i.e., $X$ is a $G$-Mori fiber space (over a point), and $X$ is not $G$-birational to any $G$-Mori fibre space that is not $G$-biregular to $X$.
\end{abstract}

\address{ \emph{Ivan Cheltsov}\newline
\textnormal{University of Edinburgh, Edinburgh, Scotland}
\newline
\textnormal{\texttt{i.cheltsov@ed.ac.uk}}}

\address{\emph{Igor Krylov}
\newline
\textnormal{Institute for Basic Science, Pohang, Korea}
\newline
\textnormal{\texttt{krylov.igor.o@gmail.com}}}

\address{\emph{Sione Ma'u}
\newline
\textnormal{University of Auckland, Auckland, New Zealand}
\newline
\textnormal{\texttt{s.mau@auckland.ac.nz}}}

\maketitle

\tableofcontents

\section{Introduction}
\label{section:introduction}

The notion of birational rigidity originated in the~work of Iskovskikh and Manin \cite{IskovskikhManin}, where they implicitly proved that every smooth hypersurface of degree $4$ in $\mathbb{P}^4$ is birationally superrigid and, in particular, non-rational. Recall that a Fano variety $X$ is \emph{birationally rigid} if $X$ is a Mori fiber space over a point (that is, $X$ has terminal singularities and the~class group $\mathrm{Cl}(X)$ has rank~$1$), and $X$ is the~only Mori fiber space in its birational equivalence class. If, in addition, 
$$
\mathrm{Bir}(X)=\mathrm{Aut}(X),
$$
then $X$ is said to be \emph{birationally superrigid}. Birational rigidity is an obstruction to rationality.

Problems concerning the~irrationality of Fano varieties defined over non-algebraically closed fields lead naturally to an equivariant generalization of the~notion of birational rigidity. To formulate this notion, we fix a Fano variety $X$ that has at most terminal singularities, and we fix a finite subgroup $G\subset \mathrm{Aut}(X)$. Then the~Fano variety $X$ is said to be $G$-birationally rigid if the~following two conditions are satisfied:
\begin{enumerate}
\item $X$ is a~$G$-Mori fiber space over a~point. i.e., $X$ has terminal singularities and $\mathrm{rk}\,\mathrm{Cl}(X)^G=1$,
\item if $Y$ is a $G$-Mori fiber space that is $G$-birational to $X$, then $Y$ is $G$-biregular to $X$.
\end{enumerate}
If, in addition, we have 
$$
\mathrm{Bir}^G(X)=\mathrm{Aut}^G(X),
$$
then $X$ is said to be $G$-birationally superrigid. The~first result on $G$-birational rigidity of Fano varieties has been implicitly proved by Segre and Manin in  \cite{Segre,Manin1968}, where they studied irrationality of smooth cubic surfaces defined over an algebraically non-closed field.

\begin{theorem}[{Segre--Manin Theorem}]
\label{theorem:Manin-Segre}
Let $X$ be a~smooth cubic surface in $\mathbb{P}^3$, and let $G$ be a finite subgroup in $\mathrm{Aut}(X)$ such that $\mathrm{rk}\,\mathrm{Cl}(X)^G=1$. Then $X$ is $G$-birationally rigid.
\end{theorem}

\begin{remark}
\label{remark:del-Pezzo}
Cubic surfaces are del Pezzo surfaces of degree $3$. Segre--Manin Theorem has a natural generalization for del Pezzo surfaces of degree $\leqslant 2$. For the~classification of all $G$-birationally rigid del Pezzo surfaces, see \cite{Cheltsov2008,DolgachevIskovskikh2009,Cheltsov2014,Sakovich,Pinardin2024,Yasinsky2025,PSY,CheltsovTschinkelZhang2026}.
\end{remark}

In this paper, we generalize Segre--Manin theorem for three-dimensional cubics.

\begin{theorem}[Main Theorem]
\label{theorem:main}
Let $X$ be a smooth cubic threefold in $\mathbb{P}^4$, and let $G$ be a finite subgroup in $\mathrm{Aut}(X)$. Then the~following conditions are equivalent:
\begin{itemize}
\item[$(\mathrm{a})$] $\mathbb{P}^4$ contains no $G$-invariant planes, and $G\not\simeq C_5\rtimes C_4, C_3\times(C_5\rtimes C_4), C_{11}\rtimes C_5$;
\item[$(\mathrm{b})$] $X$ is $G$-birationally superrigid;
\item[$(\mathrm{c})$] $X$ is $G$-birationally rigid.
\end{itemize}
\end{theorem}

In this theorem, we consider $G$ as a subgroup of $\mathrm{PGL}_5(\mathbb{C}$), since the~$G$-action lifts to $\mathbb{P}^4$.

All possibilities for the~group $G$ in Theorem~\ref{theorem:main} are known \cite{WeiYu2020}. Namely, if $G$ is a finite group that faithfully  acts on a smooth cubic threefold, it follows from \cite{WeiYu2020} that $G$ is isomorphic to a~subgroup of one of the~following six finite groups:
\begin{enumerate}
\item $C_4^3\rtimes\mathfrak{S}_5$, which is the~automorphism group of the~Fermat cubic threefold
\begin{equation}
\label{equation:Fermat}\tag{$\clubsuit$}
\big\{x_0^3+x_1^3+x_2^3+x_3^3+x_4^3=0\big\}\subset\mathbb{P}^4;
\end{equation}
\item $\mathrm{PSL}_2(\mathbf{F}_{11})$, which is the~automorphism group of the~Klein cubic threefold
\begin{equation}
\label{equation:Klein}\tag{$\spadesuit$}
\big\{x_0x_1^2+x_1x_2^2+x_2x_3^2+x_3x_4^2+x_4x_0^2=0\big\}\subset\mathbb{P}^4;
\end{equation}
\item $((C_3^2\rtimes C_3)\rtimes C_4)\times\mathfrak{S}_3$, which is the~finite group with GAP ID [648,541];
\item $C_3\times\mathfrak{S}_5$, which is the~automorphism group of the~cubic threefold
\begin{equation}
\label{equation:360-119}\tag{$\maltese$}
\big\{x_0+x_1+x_2+x_3+x_4=0,x_0^3+x_1^3+x_2^3+x_3^3+x_4^3+x_5^3=0\big\}\subset\mathbb{P}^5;
\end{equation}
\item $C_{24}$;
\item $C_{16}$.
\end{enumerate}
Using this classification and the proof of Theorem~\ref{theorem:main}, we obtain the following classification.

\begin{corollary}
\label{corollary:main}    
Let $X$ be a smooth cubic threefold in $\mathbb{P}^4$, and let $G$ be a finite subgroup in $\mathrm{Aut}(X)$. Then $X$ is $G$-birationally rigid if and only if one of the following cases holds:
\begin{enumerate}
\item[$(\mathrm{i})$] $X$ is the Fermat cubic \eqref{equation:Fermat}, and $G$ is a subgroup listed in Figure~\ref{figure:26} such that \mbox{$G\not\simeq C_5\rtimes C_4$}, and $\mathbb{P}^4$ contains no $G$-invariant planes (generators of a subgroup conjugated to $G$ can be found using Magma code provided in Appendix~\ref{subsection:plane-free-groups});
\item[$(\mathrm{ii})$] $X$ is the Klein cubic \eqref{equation:Klein}, and $G\simeq\mathrm{PSL}_2(\mathbf{F}_{11}),\mathfrak{A}_5$;
\item[$(\mathrm{iii})$] $X$ is the cubic \eqref{equation:360-119}, and $G\simeq C_3\times\mathfrak{S}_5, C_3\times\mathfrak{A}_5, \mathfrak{S}_5, \mathfrak{A}_5$;
\item[$(\mathrm{iv})$] $X$ is the cubic
$$
\Big\{\sum_{i=0}^{4} x_i^3+a\Big(\sum_{i=0}^{4} x_i\Big)\Big(\sum_{i=0}^{4} x_i^2\Big)=0\Big\}\subset\mathbb{P}^4
$$
for some $a\in \mathbb{C}$, and $G\simeq \mathfrak{S}_5,\mathfrak{A}_5$ acting on $\mathbb{P}^4$ by permuting coordinates;
\item[$(\mathrm{v})$] $X$ is the cubic
\begin{multline*}
\quad \quad \quad \quad \Big\{x_0^3+x_1^3+x_2^3+x_3^3+x_4^3+a(x_0x_1x_3+\zeta_3^2x_0x_1x_2+\zeta_3x_0x_1x_4+x_2x_3x_4+\\
+x_1x_2x_4+\zeta_3^2x_1x_3x_4+x_0x_3x_4+\zeta_3x_1x_2x_3+x_3x_2x_0+x_0x_2x_4)=0\Big\}\subset\mathbb{P}^4\quad \quad 
\end{multline*}
for some $a\in \mathbb{C}$, and $G\simeq \mathfrak{A}_5$ (generators of $G$ can be found in Section~\ref{section:A5});
\item[$(\mathrm{vi})$] $X$ is the cubic
$$
\Big\{x_0^3+x_1^3+x_2^3+x_3^3+x_4^3+ax_3(x_0x_2+x_1x_4)=0\Big\}\subset\mathbb{P}^4
$$
for some $a\in \mathbb{C}$, and $G\simeq C_3^2\rtimes C_4, \mathfrak{S}_3^2, \mathfrak{S}_3\wr C_2$ such that $\mathbb{P}^4$ does not contain $G$-invariant planes (generators of $G$ can be found in Section~\ref{section:36-9-36-10}).  
\end{enumerate}
\end{corollary}

\begin{remark}
\label{remark:Costya}    
If $X$ is the~Klein cubic threefold \eqref{equation:Klein} and $G\simeq\mathrm{PSL}_2(\mathbf{F}_{11})$, then $X$ is $G$-birationally superrigid by \cite[Theorem~A.5]{CheltsovShramov2014}. But we will not use this result in the proof of Theorem~\ref{theorem:main}.
\end{remark}

As a by-product of the proof of Theorem~\ref{theorem:main} (Main Theorem), we obtain the~following result, which completes the investigations started in \cite{CheltsovShramov2014,Avilov2016b,Avilov2018,CheltsovTschinkelZhang2025}. 

\begin{theorem}
\label{theorem:singular}
Let $X$ be a singular cubic threefold in $\mathbb{P}^4$ that has at most terminal singularities, and let $G$ be a finite subgroup in $\mathrm{Aut}(X)$ such that $\mathrm{rk}\,\mathrm{Cl}(X)^G=1$. Then $X$ is $G$-birationally rigid if and only if $X$ is $G$-birationally superrigid, which happens if and only if one of the~following holds:
\begin{itemize}
\item[$(\mathrm{v})$] $X$ is the~(ten nodal) Segre cubic threefold
\begin{equation}
\label{equation:cubic-Segre}\tag{$\diamondsuit$}
\big\{x_0+x_1+x_2+x_3+x_4+x_5=0,x_0^3+x_1^3+x_2^3+x_3^3+x_4^3+x_5^3=0\big\}\subset\mathbb{P}^5,
\end{equation}
and $G\simeq\mathfrak{S}_6,\mathfrak{A}_6$, or $G\simeq\mathfrak{S}_5,\mathfrak{A}_5$ and $G$ leaves invariant a hyperplane in $\mathbb{P}^4$;

\item[$(\mathrm{vi})$] $X$ is the~(nine-nodal) cubic threefold
\begin{equation}\tag{$\heartsuit$}
\label{equation:cubic-9-nodes}
\big\{x_0+x_1+x_2=x_3+x_4+x_5, x_0x_1x_2+x_3x_4x_5=0\big\}\subset\mathbb{P}^5,
\end{equation}
and $G\simeq\mathfrak{S}_3\wr C_2, C_3^2\rtimes C_4$, or $G\simeq\mathfrak{S}_3^2$ and $G$ acts transitively on $\mathrm{Sing}(X)$.
\end{itemize}
\end{theorem}

\begin{proof}
If $X$ is the~cubic \eqref{equation:cubic-9-nodes}, and $G\simeq C_3^2\rtimes C_4$, or $G\simeq\mathfrak{S}_3^2$ and $G$ acts transitively on $\mathrm{Sing}(X)$, then the cubic threefold $X$ is $G$-birationally superrigid by Theorem~\ref{theorem:36-9}, which is proved in Section~\ref{section:36-9-36-10}. The remaining assertions are proved in \cite{CheltsovShramov2014,Avilov2016b,Avilov2018,CheltsovTschinkelZhang2025}.
\end{proof}

It would be natural to generalize both Theorems~\ref{theorem:main} and \ref{theorem:singular} for del Pezzo threefolds. Let us recall their classification. Let $X$ be a Fano threefold with terminal Gorenstein singularities such that 
$$
-K_X\sim 2H,
$$
where $H$ is a Cartier divisor on $X$. Suppose that $\mathrm{rk}\,\mathrm{Cl}(X)^G=1$ for a finite subgroup $G\subset\mathrm{Aut}(X)$. Set $d=H^3$. Then $d\in\{1,2,3,4,5,6,8\}$, the threefold $X$ is a del Pezzo threefold of degree $d$, and it follows from \cite{Prokhorov2013a,Prokhorov2013b,Prokhorov2015,KuznetsovProkhorov2023} that one of the following cases holds:
\begin{itemize}
\item[$(\mathrm{I})$] $d=1$ and $X$ is a sextic hypersurface in $\mathbb{P}(1,1,1,2,3)$;
\item[$(\mathrm{II})$] $d=2$ and $X$ is a quartic hypersurface in $\mathbb{P}(1,1,1,1,2)$; 
\item[$(\mathrm{III})$] $d=3$ and $X$ is a cubic hypersurface in $\mathbb{P}^4$;
\item[$(\mathrm{IV})$] $d=4$ and $X$ is a complete intersection of two quadrics in $\mathbb{P}^5$;
\item[$(\mathrm{V})$] $d=5$ and $X$ is a smooth intersection of the Grassmannian
$\mathrm{Gr}(2,5)\hookrightarrow\mathbb{P}^9$  with a linear subspace of dimension $6$ ($X$ is unique up to an isomorphism, and $\mathrm{Aut}(X)\simeq\mathrm{PGL}_2(\mathbb{C})$);
\item[$(\mathrm{VI})$] $d=6$ and $X=\mathbb{P}^1\times\mathbb{P}^1\times \mathbb{P}^1$;
\item[$(\mathrm{VI}^\prime)$] $d=6$ and $X=\{x_0y_0+x_1y_1+x_2y_2=0\}\subset\mathbb{P}^2_{x_0,x_1,x_2}\times\mathbb{P}^2_{y_0,y_1,y_2}$; 
\item[$(\mathrm{VIII})$] $d=8$ and $X=\mathbb{P}^3$ (this case is usually excluded).
\end{itemize}
It seems feasible to find all possibilities for $X$ and $G$ such that $X$ is $G$-birationally rigid for $d\ne 3$. For $X=\mathbb{P}^3$, this is done in \cite{CheltsovShramov2012,CheltsovShramov2014,CheltsovShramov2019}. For $\mathbb{P}^1\times\mathbb{P}^1\times \mathbb{P}^1$, this is done in~\cite{CheltsovDuboulozKishimoto2023,Stor}. If $d=5$, then $X$ is $G$-birationally rigid if and only if $G\simeq\mathfrak{A}_5$ \cite{CheltsovShramov2016}. 
For $d=4$, the problem is solved in \cite{Avilov2016a,RobertAurore}. For $d\leqslant 2$, partial results are obtained in \cite{CheltsovShramovV5,Avilov2019,Avilov2023,Cheltsov2023,AbbanCheltsovParkShramov2025}.

The structure of this paper is simple. In Section~\ref{section:proof}, we prove Theorem~\ref{theorem:main}  modulo $5$ results, which are proved in Sections~\ref{section:20-3}, \ref{section:55-1}, \ref{section:405-15}, \ref{section:A5},~\ref{section:36-9-36-10}. In~Appendix~\ref{section:Magma}, we present Magma codes used in the~proof.

\medskip
\noindent
\textbf{Notation.}
Throughout the~paper, we denote by $C_n$ the~cyclic group of order $n$, and by  $\mathfrak{D}_n$ the~dihedral group of order~$2n\geqslant 6$. Similarly, we denote by $\mathfrak{S}_n$ the~symmetric group of degree $n$, and we denote by $\mathfrak{A}_n$ its alternating subgroup. We denote by $\zeta_n$ a primitive $n$-th root of unity.

\medskip
\noindent
\textbf{AI disclosure.}
The first version of this paper was written entirely by the authors, and it was submitted to Arxiv in April 2026. Then, in November 2026, we used ChatGPT~6.1 Astra to check the~manuscript. It identified numerous typos, found one gap in the original proof of Lemma~\ref{lemma:A5-irreducible-curves-degree-10}, suggested an alternative (much simpler) proof of this lemma and also strengthened its assertion, which allowed to shorten the whole proof of Theorem~\ref{theorem:A5}. The authors take full responsibility for the~second version of the paper.

\medskip
\noindent
\textbf{Acknowledgements.}
Most of the~work presented in this paper was carried out during a visit by the~first two authors to the~University of Auckland (Tamaki Makaurau, Aotearoa) in March~2026. We thank the~university for providing an excellent environment for conducting research. 

We are very grateful to Paolo Cascini and Zhijia Zhang for useful discussions. We would like to thank OpenAI for developing ChatGPT~6.1 Astra, which helped us improve this paper.

The~first author was supported by Simons Collaboration grant \emph{Moduli of varieties}.
The~second author was supported by \emph{IBS-R003-D1} grant.

\begin{center}
\begin{figure}
\caption{Plane free groups acting on Fermat cubic}
  \label{figure:26}
\includegraphics[height=\textheight]{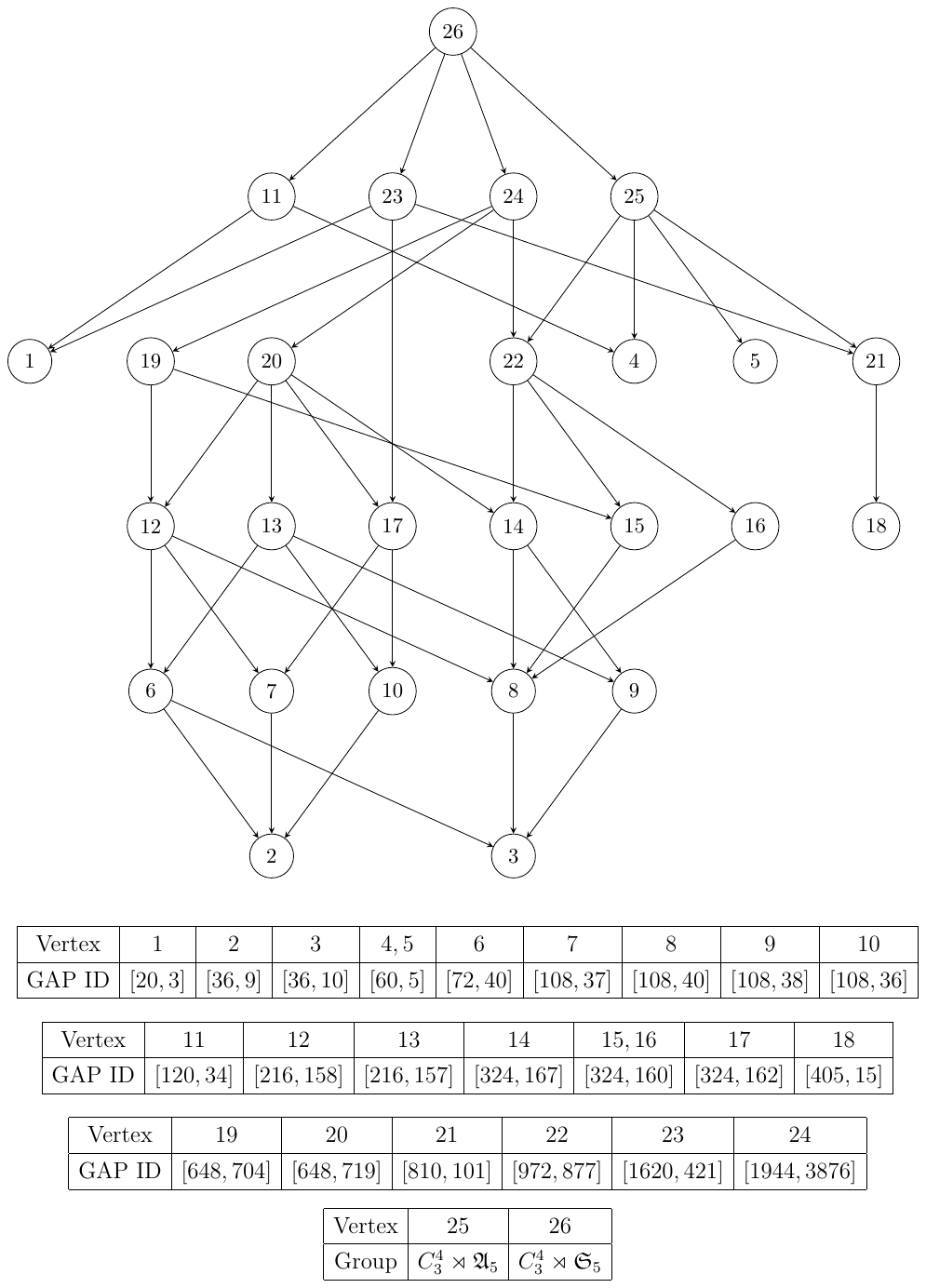}
\end{figure}
\end{center}

\section{The proof of Main Theorem}
\label{section:proof}

Let $X$ be a smooth cubic threefold in $\mathbb{P}^4$, and let $G$ be a subgroup in $\mathrm{Aut}(X)$. Then $G$ is finite, and its action on $X$ lifts to $\mathbb{P}^4$, so we consider $G$ as a subgroup in $\mathrm{PGL}_5(\mathbb{C})$. To prove Theorem~\ref{theorem:main}, we must show that the~following three conditions are equivalent:
\begin{itemize}
\item[$(\mathrm{a})$] $\mathbb{P}^4$ does not contain $G$-invariant planes, $G\not\simeq C_5\rtimes C_4, C_3\times(C_5\rtimes C_4), C_{11}\rtimes C_5$;
\item[$(\mathrm{b})$] $X$ is $G$-birationally superrigid;
\item[$(\mathrm{c})$] $X$ is $G$-birationally rigid.
\end{itemize}
The implication $(\mathrm{b})\Rightarrow (\mathrm{c})$ follows from the~definitions. If $G\simeq C_5\rtimes C_4$ or $G\simeq C_3\times(C_5\rtimes C_4)$, then the~cubic threefold $X$ is not $G$-birationally rigid by Theorem~\ref{theorem:20-3}, which will be proved in Section~\ref{section:20-3}. Similarly, if $G\simeq C_{11}\rtimes C_5$, then the~cubic threefold $X$ is not $G$-birationally rigid by Theorem~\ref{theorem:Klein-cubic}, which will be proved in Section~\ref{section:55-1}. Therefore, the~implication $(\mathrm{c})\Rightarrow (\mathrm{a})$ follows from

\begin{lemma}
\label{lemma:Pn-pencil}
Suppose that $\mathbb{P}^4$ contains a $G$-invariant plane. Then~$X$ is not $G$-birationally rigid.
\end{lemma}

\begin{proof}
Let $\mathcal{P}$ be the~pencil in $|\mathcal{O}_X(1)|$ consisting of all surfaces that are cut out by hyperplanes that pass through $\Pi$.
Then $\mathcal{P}$ is mobile, and all its members are irreducible cubic surfaces.
Let $\psi\colon X\dasharrow\mathbb{P}^1$ be the~map given by the pencil~$\mathcal{P}$. Then $\psi$ is $G$-equivariant and dominant.
Moreover, equivariantly resolving the~indeterminacy of this map, we get the~following $G$-equivariant commutative diagram:
$$
\xymatrix{
& Y\ar@{->}[ld]_\pi\ar@{->}[rd]^{\phi} & \\
X\ar@{-->}[rr]^\psi  && \mathbb{P}^1}
$$
where $Y$ is a smooth threefold, $\pi$ is a birational morphism, and $\phi$ is a surjective morphism, whose general fiber is birational to a cubic surface. Now, applying relative $G$-Minimal Model Program to~$Y$ over~$\mathbb{P}^1$, we obtain a $G$-birational map from $Y$ to a $G$-Mori fiber space with a positive-dimensional base, which implies that~$X$ is not $G$-birationally rigid.
\end{proof}

Thus, to complete the~proof of Theorem~\ref{theorem:main}, it is enough to prove the~implication $(\mathrm{a})\Rightarrow (\mathrm{b})$. Hence, we suppose that $\mathbb{P}^4$ does not contain $G$-invariant planes, $G\not\simeq C_5\rtimes C_4$, $G\not\simeq C_3\times(C_5\rtimes C_4)$, and $G\not\simeq C_{11}\rtimes C_5$. Let us show that $X$ is $G$-birationally superrigid.

By Duncan's lemma \cite[Lemma 2.4]{AbbanCheltsovKishimotoMangolte2025}, the~subgroup $G$ lifts isomorphically to $\mathrm{GL}_5(\mathbb{C})$, so the~action of the~group $G$ on $\mathbb{P}^4$ is given by its faithful five-dimensional representation $\mathbb{V}_5$. By~assumption, $\mathbb{V}_5$ does not have three-dimensional subrepresentations. Therefore, twisting $\mathbb{V}_5$ by an~appropriate $1$-dimensional representation of $G$, we may assume that either $\mathbb{V}_5$ is irreducible, or
$$
\mathbb{V}_5=\mathbb{I}\oplus\mathbb{V}_4,
$$
where $\mathbb{V}_4$ is an irreducible faithful $4$-dimensional representation of the~group $G$, and $\mathbb{I}$ is the~trivial representation. In particular, we see that $G$ is neither cyclic nor dihedral. Hence, it follows from the~classification \cite{WeiYu2020} that $G$ is isomorphic to a subgroup of one of the~following four groups:
\begin{enumerate}
\item $C_3^4\rtimes\mathfrak{S}_5$ --- the~automorphism group of the~Fermat cubic \eqref{equation:Fermat};
\item $\mathrm{PSL}_2(\mathbf{F}_{11})$ --- the~automorphism group of the~Klein cubic \eqref{equation:Klein};
\item $((C_3^2\rtimes C_3)\rtimes C_4)\times\mathfrak{S}_3$ --- the~group with GAP ID [648,541];
\item $C_3\times\mathfrak{S}_5$.
\end{enumerate}

\begin{lemma}
\label{lemma:Fermat}
Let $H$ be a subgroup of the~group $C_4^3\rtimes\mathfrak{S}_5$ such that $H$ has a faithful irreducible representation of dimension $4$ or $5$. Then $H$ is isomorphic to one of the~groups listed in Figure~\ref{figure:26}.
\end{lemma}

\begin{proof}
Recall that $C_4^3\rtimes\mathfrak{S}_5$ is the~automorphism group of the~Fermat cubic \eqref{equation:Fermat}.
Using the~Magma code presented in Appendix~\ref{subsection:plane-free-groups}, we find all subgroups of the~automorphism group of this cubic (up to conjugation) that do not leave invariant planes in $\mathbb{P}^4$. The~graph of these subgroups is presented in Figure~\ref{figure:26} (where edges correspond to inclusions). Up to isomorphism, this graph contains $24$ groups, and all of them have a faithful irreducible representation of dimension $4$ or $5$.

Now, using the~Magma code presented in Appendix~\ref{subsection:qualifying-subgroups}, we obtain the~list of all subgroups of the~group $C_4^3\rtimes\mathfrak{S}_5$ that have an irreducible faithful representation of dimension $4$ or~$5$, which contains the~same subgroups that are listed in Figure~\ref{figure:26}.
\end{proof}

\begin{lemma}
\label{lemma:648-541}
Let $H$ be a subgroup of the~group $((C_3^2\rtimes C_3)\rtimes C_4)\times\mathfrak{S}_3$ (GAP ID [648, 541]) such that $H$ has a faithful irreducible representation of dimension $4$ or $5$. Then one of the~following cases holds:
\begin{itemize}
\item $H\simeq\mathfrak{S}_3^2$ and its GAP ID is [36,10];
\item $H\simeq C_3\times\mathfrak{S}_3^2$ and its GAP ID is [108,38].
\end{itemize}
\end{lemma}

\begin{proof}
As in the~proof of Lemma~\ref{lemma:Fermat}, we use the~Magma code presented in Appendix~\ref{subsection:weird-threefold-qualifying-subgroups} to obtain the~list of all subgroups of the~group $((C_3^2\rtimes C_3)\rtimes C_4)\times\mathfrak{S}_3$ that have an irreducible representation of dimension $4$ or~$5$. This list consists of $8$ subgroups (up to isomorphism). Then, using GAP, we check that only two of them have a faithful irreducible representation of dimension $4$ or $5$. These are the~groups $\mathfrak{S}_3^2$ and $C_3\times\mathfrak{S}_3^2$. For instance, the~following GAP code verifies that the~group $((C_3^2\rtimes C_3)\rtimes C_4)\times\mathfrak{S}_3$ does not have faithful $4$-dimensional irreducible representations.
\begin{verbatim}
    G := SmallGroup(648,541);
    tbl := CharacterTable(G);
    irr := Irr(tbl);
    deg4 := Filtered(irr, chi -> Degree(chi) = 4);
    Print("Number of irreducible representations: ", Length(deg4), "\n");
    faithful := [];
    for chi in deg4 do
        ker := KernelOfCharacter(chi);
        if Size(ker) = 1 then
            Add(faithful, chi);
        fi;
    od;
    if Length(faithful) > 0 then
        Print("Faithful representation exists.\n");
        Print("Number of representations: ", Length(faithful),"\n");
    else
        Print("No faithful representation found.\n");
    fi;
\end{verbatim}
Hence, we conclude that $H\simeq\mathfrak{S}_3^2$ or $H\simeq C_3\times\mathfrak{S}_3^2$, as required.
\end{proof}

If $G$ is isomorphic to a subgroup of the~group $\mathrm{PSL}_2(\mathbf{F}_{11})$, then  
$G\simeq\mathrm{PSL}_2(\mathbf{F}_{11})$ or $G\simeq\mathfrak{A}_5$, because $G\not\simeq C_{11}\rtimes C_5$ by assumption,
and all other subgroups of the~group $\mathrm{PSL}_2(\mathbf{F}_{11})$ do not have irreducible representations of dimension $4$ or $5$. Similarly, up to isomorphism,  $C_3\times\mathfrak{S}_5$ contains three subgroups that have a faithful irreducible representation of dimension $4$ or $5$ that are not listed in Figure~\ref{figure:26}. These are the~groups $C_3\times\mathfrak{S}_5$, $C_3\times\mathfrak{A}_5$, $C_3\times(C_5\rtimes C_4)$.
Since $G\not\simeq C_3\times(C_5\rtimes C_4)$, we see that one of the~following cases holds:
\begin{enumerate}
\item $G\simeq\mathrm{PSL}_2(\mathbf{F}_{11})$,
\item $G\simeq C_3\times\mathfrak{S}_5$,
\item $G\simeq C_3\times\mathfrak{A}_5$,
\item $G$ is isomorphic to one of the~groups listed in Figure~\ref{figure:26} except $C_5\rtimes C_4$.
\end{enumerate}
Now, using Lemmas~\ref{lemma:Fermat} and \ref{lemma:648-541} (and looking at the~graph presented in Figure~\ref{figure:26}), we conclude that the~group $G$ contains a subgroup $H$ such that the~subgroup $H$ is isomorphic to one of the~following four \textit{minimal} groups:
\begin{center}
$C_3^4\rtimes C_5$ (GAP ID is [405,15]), $\mathfrak{A}_5$, $C_3^2\rtimes C_4$ (GAP ID is [36,9]), $\mathfrak{S}_3^2$ (GAP ID is [36,10]).
\end{center}
In Sections~\ref{section:405-15}, \ref{section:A5}, \ref{section:36-9-36-10} we will show that $X$ is $G$-birationally superrigid if $G$ is one of these four \textit{minimal} groups. This gives the~following result.

\begin{theorem}[{cf. \cite{Kollar2009}}]
\label{theorem:minimal}
Suppose that the~restriction of the~$G$-representation $\mathbb{V}_5$ to $H$ is
\begin{itemize}
\item either an irreducible $5$-dimensional representation of the~group $H$,
\item or a sum of a $1$-dimensional representation and an irreducible $4$-dimensional representation.
\end{itemize}
Then $X$ is $H$-birationally superrigid, and $X$ is $G$-birationally superrigid.
\end{theorem}

\begin{proof}
If $H\simeq C_3^4\rtimes C_5$, then $X$ is $H$-birationally superrigid by Theorem~\ref{theorem:405-15}. Similarly, if $H\simeq \mathfrak{A}_5$, then $X$ is $H$-birationally superrigid by Theorem~\ref{theorem:A5}. If $H\simeq C_3^2\rtimes C_4,\mathfrak{S}_3^2$, then $X$ is $H$-birationally superrigid by Theorem~\ref{theorem:36-9}. Using \cite[Corollary 3.3.3]{CheltsovShramov2016} and \cite[Corollary~1.4.3]{Cheltsov2005}, we see that the~following two conditions are equivalent:
\begin{enumerate}
\item $X$ is not $G$-birationally superrigid;
\item there is a non-empty $G$-invariant mobile linear subsystem $\mathcal{M}\subset |\mathcal{O}_X(n)|$, for some  $n\in\mathbb{Z}_{>0}$, such that the~singularities of the~log pair $(X,\frac{2}{n}\mathcal{M})$ are not canonical.
\end{enumerate}
Thus, if $X$ is not $G$-birationally superrigid, then $X$ is not $H$-birationally superrigid. Hence, we conclude that $X$ is  $G$-birationally superrigid.
\end{proof}

In the~proof of Theorem~\ref{theorem:minimal}, we used Theorems~\ref{theorem:405-15}, \ref{theorem:A5}, \ref{theorem:36-9}. We will prove these theorems later in Section~\ref{section:405-15}, \ref{section:A5}, \ref{section:36-9-36-10}, respectively. Their proofs are independent of each other.

Now, we are ready to conclude the~proof of Theorem~\ref{theorem:main}. Suppose that $X$ is not $G$-birationally superrigid.  Then, by Theorem~\ref{theorem:minimal}, the~restriction of the~$G$-representation $\mathbb{V}_5$ to $H$ is a faithful $5$-dimensional representation of the~subgroup $H$ that  has a $3$-dimensional subrepresentation $\mathbb{V}_3$. This subrepresentation $\mathbb{V}_3$ corresponds to a $H$-invariant plane
$$
\Pi\subset\mathbb{P}^4.
$$
But the~groups $C_3^4\rtimes C_5$ and $C_3^2\rtimes C_4$ do not have a faithful $5$-dimensional representation that have a $3$-dimensional subrepresentation.
This shows that $H\not\simeq C_3^4\rtimes C_5$ and $H\not\simeq C_3^2\rtimes C_4$.

If $H\simeq \mathfrak{A}_5$, then $\mathbb{V}_3$ must be an irreducible representation of the~group $H$, which implies that the~plane $\Pi$ does not contain $H$-invariant cubic curves, which is a contradiction, because $\Pi\not\subset X$, so $X\vert_{\Pi}$ is an $H$-invariant cubic curve. This shows that $H\not\simeq \mathfrak{A}_5$.

Hence, we have $H\simeq \mathfrak{S}_3^2$. Applying previous arguments to other subgroups of $G$, we conclude that $G$ does not contain subgroups isomorphic to $C_3^4\rtimes C_5$, $\mathfrak{A}_5$, $C_3^2\rtimes C_4$. Now, looking at the~graph presented in Figure~\ref{figure:26}, we conclude that the~GAP ID of the~group $G$ is one of the~following:
\begin{center}
[108,38], [108,40], [324,160], [324,167], [972,877].
\end{center}
These groups do not have irreducible $5$-dimensional representations. Thus, we have
$$
\mathbb{V}_5=\mathbb{I}\oplus\mathbb{V}_4,
$$
where $\mathbb{V}_4$ is an irreducible faithful $4$-dimensional representation of the~group $G$ such that its restriction to the~subgroup $H$ is a faithful non-irreducible representation of the~group $H\simeq \mathfrak{S}_3^2$. This is purely a group-theoretic condition. Let us verify it using Magma.

If the~GAP ID of $G$ is [108,38], then $G$ has a unique subgroup isomorphic to $\mathfrak{S}_3^2$, which is normal, so $H$ must be this subgroup. This can be checked using the following Magma code:
\begin{verbatim}
    G:=SmallGroup(108,38);
    S:=Subgroups(G);
    for s in S do
        H:=s`subgroup;
        print(IdentifyGroup(H));
        if (IsNormal(G,H)) then 
        print ("Normal\n" );
        end if;
    end for;
\end{verbatim}
In this case, $G$ has two faithful $4$-dimensional representations. Moreover, restricting them to $H$, we obtain irreducible representations (this also follows from the~Clifford theory), which implies that the~GAP ID of $G$ is not [108,38].
Here, we used the following Magma code:
\begin{verbatim}
    G:=SmallGroup(108,38);
    K:=CyclotomicField(108);
    L:=IrreducibleModules(G,K);
    //print(L);
    R:=L[27];
    print(R);
    print(IdentifyGroup(Kernel(R)));
    S:=Subgroups(G);
    H:=S[48]`subgroup;
    print(IdentifyGroup(H));
    W:=Restriction(R, H);
    print(W);
    IsIrreducible(W);
\end{verbatim}

If  the~GAP ID of $G$ is [108,40], then $G$ has three subgroups isomorphic to $\mathfrak{S}_3^2$ (up to conjugation), and $G$ has two (non-isomorphic) faithful $4$-dimensional representations. Restricting each of these representations to each of these subgroups, we obtain irreducible representations, which shows that the~GAP ID of $G$ is not [108,40].

If the~GAP ID of $G$ is [324,160], it has a unique subgroup isomorphic to $\mathfrak{S}_3^2$ (up to conjugation), and $G$ has $6$ (non-isomorphic) $4$-dimensional representations, which are all faithful, and their restrictions to $H$ are irreducible. Hence, the~GAP ID of $G$ is not [324,160].

Similarly, if  the~GAP ID of $G$ is [324,167], then $G$ contains three subgroups isomorphic to $\mathfrak{S}_3^2$ (up to conjugation), and $G$ has four non-isomorphic faithful irreducible 4-dimensional representations. Restricting each of these four irreducible representations to each of these three subgroups, we obtain an irreducible representation, which shows that the~GAP ID of $G$ is not [324,167].

Finally, if the~GAP ID of the~group $G$ is [972,877], then $G$ contains a unique subgroup isomorphic to $\mathfrak{S}_3^2$ (up to conjugation), and $G$ has $12$ (non-isomorphic) faithful $4$-dimensional representations. Restricting each of these representations to this subgroup, we obtain an irreducible representation. Thus, the~GAP ID of $G$ is not [972,877].

The obtained contradiction completes the~proof of Theorem~\ref{theorem:main} (Main Theorem).

\section{Groups $C_5\rtimes C_4$ and $C_3\times(C_5\rtimes C_4)$}
\label{section:20-3}

Let $X$ be a smooth cubic threefold in $\mathbb{P}^4$. Suppose that $\mathrm{Aut}(X)$ contains a subgroup such that either $G\simeq C_5\rtimes C_4$ or $G\simeq C_3\times(C_5\rtimes C_4)$. The~goal of this section is to prove the~following result.

\begin{theorem}
\label{theorem:20-3}
The cubic threefold $X$ is not $G$-birationally rigid.
\end{theorem}

\begin{proof}
The $G$-action on $X$ lifts to $\mathbb{P}^4$. Moreover, it follows from \cite[Lemma 2.4]{AbbanCheltsovKishimotoMangolte2025} that the~action of the~group $G$ on $\mathbb{P}^4$ is given by its faithful $5$-dimensional representation. Looking at representations of the~group $G$, we see that this representation splits as a sum of a $1$-dimensional representation and an irreducible $4$-dimensional representation.

Let $S$ be the $G$-invariant hyperplane section of $X$, and let $\overline{G}\subset\mathrm{Aut}(S)$ be the image of $G$ via the  natural homomorphism $G\to \mathrm{Aut}(S)$. Then $\overline{G}\simeq C_5\rtimes C_4$, and $S$ has isolated singularities, since $X$ is smooth. In fact, the~surface $S$ is smooth. Indeed, if $S$ is singular, then it has exactly four singular points, because $S$ does not have $G$-orbits of length less than $4$, and there are no cubic surfaces with isolated singularities that have more than four singular points. But cubic surface with four singular points is unique, and it does not admit a faithful $\overline{G}$-action. Hence, $S$ is smooth.

Since $S$ is smooth, it follows from \cite{DolgachevIskovskikh2009} that $S$ is isomorphic to the~Clebsch cubic surface, so we have $\mathrm{Aut}(S)\simeq\mathfrak{S}_5$.
Moreover, by \cite{Iskovskikh1996,Wolter},  we have the~following $G$-Sarkisov link:
$$
\xymatrix{
&S\ar[dl]_{\pi}\ar[dr]^{\phi}&\\
\overline{S}&&\mathbb{P}^1\times\mathbb{P}^1}
$$
where $\overline{S}$ is a smooth quintic del Pezzo surface, $\pi$ and $\phi$ are blowups of $2$ and $5$ points, respectively. Let $L_1$ and $L_2$ be the~$\pi$-exceptional curves, and let $\ell_1$, $\ell_2$, $\ell_3$, $\ell_4$, $\ell_5$ be the~$\phi$-exceptional curves. Then $L_1$, $L_2$, $\ell_1$, $\ell_2$, $\ell_3$, $\ell_4$, $\ell_5$ are lines, and it follows from \cite{CheltsovKrylovMartinezShinder2024,SashaYura2025} that this $G$-Sarkisov link can be extended to the~following three-dimensional $G$-Sarkisov link:
$$
\xymatrix{
&\widetilde{X}\ar[dl]_{\alpha}\ar[dr]^{\beta}\ar@{-->}[rr]^{\chi}&&\widehat{X}\ar[dl]_{\gamma}\ar[dr]^{\delta}&\\
X&&V&&Y}
$$
where $\alpha$ is the~blow up of the~lines $L_1$ and $L_2$, $V$ is a divisor of degree $(2,2)$ in $\mathbb{P}^2\times\mathbb{P}^2$ that has five isolated ordinary double points, $\beta$ is a small birational morphism that contracts strict transforms of the~lines $\ell_1$, $\ell_2$, $\ell_3$, $\ell_4$, $\ell_5$ to the~five singular points of $V$, $\chi$ is a composition of Atiyah flops of the~curves contracted by $\beta$, $Y$ is a Fano threefold with $\mathrm{Pic}(Y)=\mathbb{Z}$ and $(-K_Y)^3=14$ that has one isolated ordinary double point, $\gamma$ is a small birational morphism, and $\delta$ is a birational morphism that contracts the~strict transform of the~surface $S$ to the~singular point of the~threefold $Y$. 

By construction, $Y$ is a $G$-Mori fibre space (over a point) and $Y\not\simeq X$. Thus, we conclude that the cubic threefold $X$ is not $G$-birationally rigid.
\end{proof}

\section{Group $C_{11}\rtimes C_5$}
\label{section:55-1}

Let $X$ be a smooth cubic threefold in $\mathbb{P}^4$ such that $\mathrm{Aut}(X)$ contains a subgroup $G\simeq C_{11}\rtimes C_5$. Then the~$G$-action on $X$ lifts to $\mathbb{P}^4$, and it follows from \cite[Lemma 2.4]{AbbanCheltsovKishimotoMangolte2025} that the~action of the~group $G$ on the~projective space $\mathbb{P}^4$ is given by its faithful $5$-dimensional representation. This representation must be irreducible, and it is isomorphic to one of the~two distinct irreducible $5$-dimensional representations of the~group $G$. Since these $G$-representations differ by an outer automorphism of the~group $G$, we see that $G$ is uniquely determined as a subgroup of $\mathrm{PGL}_{5}(\mathbb{C})$ up to a conjugation. Hence, changing coordinates on $\mathbb{P}^4$ if necessary, we may assume that $G$ is generated by the~following two projective transformations:
\begin{align*}
[x_0:x_1:x_2:x_3:x_4]&\mapsto [x_1:x_2:x_3:x_4:x_0],\\
[x_0:x_1:x_2:x_3:x_4]&\mapsto [\zeta_{11}x_0:\zeta_{11}^5x_1:\zeta_{11}^3x_2:\zeta_{11}^4x_3:\zeta_{11}^9x_4].
\end{align*}
Then explicit computations show that $G$ leaves invariant $5$ smooth cubic threefolds:
$$
X_a=\big\{x_0x_1^2+\zeta_5^ax_1x_2^2+\zeta_5^{2a}x_2x_3^2+\zeta_5^{3a}x_3x_4^2+\zeta_5^{4a}x_4x_0^2=0\},
$$
where $a\in\{0,1,2,3,4\}$. Since all these hypersurfaces are $G$-biregular, we may assume that $X=X_0$, which also follows from \cite{Roulleau}. Moreover, it follows from \cite{Adler} that $\mathrm{Aut}(X)\simeq\mathrm{PSL}_2(\mathbf{F}_{11})$, and the cubic $X$ is known as the~Klein cubic threefold.

\begin{theorem}
\label{theorem:Klein-cubic}
The cubic threefold $X$ is not $G$-birationally rigid.
\end{theorem}

To prove this theorem, let $\chi\colon X\dasharrow\mathbb{P}^9$ be the~rational map given by
$$
[x_0:x_1:x_2:x_3:x_4]\mapsto [f_0:f_1:f_2:f_3:f_4:f_5:f_6:f_7:f_8:f_9],
$$
where
\begin{multline*}
f_0=x_0x_1(x_0x_1x_2x_3x_4),
f_1=x_1x_2(x_0x_1x_2x_3x_4),
f_2=x_2x_3(x_0x_1x_2x_3x_4),
f_3=x_3x_4(x_0x_1x_2x_3x_4),\\
f_4=x_4x_0(x_0x_1x_2x_3x_4),
f_5=x_3(x_1x_2^2+x_3x_4^2)x_2x_3x_4,
f_6=x_4(x_0^2x_4+x_2x_3^2)x_3x_4x_0,\\
f_7=x_0(x_0x_1^2+x_3x_4^2)x_4x_0x_1,
f_8=x_1(x_0^2x_4+x_1x_2^2)x_0x_1x_2,
f_9=x_2(x_0x_1^2+x_2x_3^2)x_1x_2x_3,
\end{multline*}
and let $Y$ be the~subscheme in $\mathbb{P}^9$ that is given by the~following equations:
\begin{multline*}
\quad \quad \quad \quad \quad y_0y_9+y_1y_3+y_2y_8=0, y_2y_7-y_0^2-y_3y_4=0,y_0y_6+y_1y_4+y_3y_7=0,\\
y_0y_4+y_1^2-y_3y_8=0, y_1y_2-y_0y_5+y_3^2=0, y_1y_5+y_2y_4+y_3y_9=0,\\
y_0y_3+y_2y_6+y_4y_5=0, y_0y_2+y_1y_7+y_4y_8=0, y_0y_1+y_2^2-y_4y_9=0,\\
y_0y_1-y_3y_6-y_5y_7=0, y_1y_6-y_2y_3-y_4^2=0, y_4y_7-y_1y_2+y_6y_8=0,\\
y_0y_4-y_2y_5-y_6y_9=0, y_2y_3-y_0y_8-y_7y_9=0, y_3y_4-y_1y_9-y_5y_8=0.\quad \quad \quad \quad \quad
\end{multline*}
Then, using Magma, one can check that $Y$ is an irreducible smooth three-dimensional subvariety.
Moreover, $Y$ is $G$-invariant, $\chi$ is $G$-equivariant, and $\chi(X)\subset Y$. 

\begin{lemma}
\label{lemma:55-1-Gr-2-6}
The threefold $Y$ is an intersection of the~Grassmannian $\mathrm{Gr}(2,6)\subset \mathbb{P}^{14}$ with a linear subspace of dimension nine, where $\mathrm{Gr}(2,6)\hookrightarrow\mathbb{P}^{14}$ is the~Pl\"ucker embedding.
\end{lemma}

\begin{proof}
We may assume that $\mathrm{Gr}(2,6)$ is given in $\mathbb{P}^{14}$ by the~following system of equations:
\begin{multline*}
\quad \quad p_{01}p_{23}-p_{02}p_{13}+p_{03}p_{12}=0, p_{01}p_{24}-p_{02}p_{14}+p_{04}p_{12}=0, p_{01}p_{25}-p_{02}p_{15}+p_{05}p_{12}=0,\\
p_{01}p_{34}-p_{03}p_{14}+p_{04}p_{13}=0, p_{01}p_{35}-p_{03}p_{15}+p_{05}p_{13}=0, p_{01}p_{45}-p_{04}p_{15}+p_{05}p_{14}=0,\\
p_{02}p_{34}-p_{03}p_{24}+p_{04}p_{23}=0, p_{02}p_{35}-p_{03}p_{25}+p_{05}p_{23}=0, p_{02}p_{45}-p_{04}p_{25}+p_{05}p_{24}=0,\\
p_{03}p_{45}-p_{04}p_{35}+p_{05}p_{34}=0, p_{12}p_{34}-p_{13}p_{24}+p_{14}p_{23}=0, p_{12}p_{35}-p_{13}p_{25}+p_{15}p_{23}=0,\\
p_{12}p_{45}-p_{14}p_{25}+p_{15}p_{24}=0, p_{13}p_{45}-p_{14}p_{35}+p_{15}p_{34}=0, p_{23}p_{45}-p_{24}p_{35}+p_{25}p_{34}=0,
\quad \quad
\end{multline*}
where $p_{ij}$ are Pl\"ucker coordinates on $\mathbb{P}^{14}$. Let us rename these coordinates as follows: $t_0=p_{01}$, $t_1=p_{12}$, $t_2=p_{23}$, $t_3=p_{34}$, $t_4=p_{04}$, $t_5=p_{13}$, $t_6=p_{24}$, $t_7=p_{03}$, $t_8=p_{14}$, $t_9=p_{02}$,
$t_{10}=p_{45}$, $t_{11}=p_{05}$, $t_{12}=p_{15}$, $t_{13}=p_{25}$, $t_{14}=p_{35}$. Set
$$
\Lambda=\big\{t_{10}=t_5, t_{11}=t_6, t_{12}=t_7, t_{13}=t_8, t_{14}=t_9\big\}\subset\mathbb{P}^{14},
$$
and identify $\Lambda=\mathbb{P}^9$ with coordinates $t_{0}$, $t_1$, $t_2$, $t_3$, $t_4$, $t_5$, $t_6$, $t_7$, $t_8$, $t_9$. Finally, we change coordinates on $\Lambda$ as follows:  $y_0=-t_8$, $y_1=t_9$,  $y_2=t_5$,
$y_3=-t_6$, $y_4=t_7$, $y_5=-t_4$, $y_6=t_0$, $y_7=t_1$, $y_8=t_2$, $y_9=-t_3$. Then $\Lambda\cap \mathrm{Gr}(2,6)=Y$.
\end{proof}

Now, it follows from \cite{Gushel,IskovskikhProkhorov} that $Y$ is an anticanonically embedded smooth Fano threefold such that $\mathrm{Pic}(Y)=\mathbb{Z}[-K_Y]$ and $(-K_Y)^3=14$. Moreover, one can show that the~$G$-equivariant rational map $X\dasharrow Y$ induced by $\chi$ is birational. Hence, the cubic $X$ is not $G$-birationally rigid, which proves Theorem~\ref{theorem:Klein-cubic}. We~expect that $X$ is $G$-birationally solid \cite{CheltsovDuboulozKishimoto2023,CheltsovSarikyan2023,Pinardin2024,PinardinZhang2025}.

Let us explain how did we find $\chi$. Consider $G$-invariant singular quartic threefold:
$$
V_4=\{x_0^2x_1x_3+x_0x_2^2x_3+x_0x_2x_4^2+x_1^2x_2x_4+x_1x_3^2x_4=0\}\subset\mathbb{P}^4.
$$
Observe that the~singular locus of $V_4$ consists of $5$ terminal Gorenstein singular points, which form the~$G$-orbit of the~point $[1:0:0:0:0]$. Note also that the~threefolds $X$ and $V_4$ are $G$-birational. Indeed, let $\psi\colon X\dasharrow V_4$ be the~rational map given by
$$
[x_0:x_1:x_2:x_3:x_4]\mapsto [x_0x_1:x_1x_2:x_2x_3:x_3x_4:x_0x_4],
$$
and let $C$ be the~$G$-invariant curve in $X$ that consists of the~following $5$ lines:
\begin{align*}
\ell_1&=\{x_0=0, x_1=0, x_3=0\},\\
\ell_2&=\{x_1=0, x_2=0, x_4=0\},\\
\ell_3&=\{x_2=0, x_3=0, x_0=0\},\\
\ell_4&=\{x_3=0, x_4=0, x_1=0\},\\
\ell_5&=\{x_4=0, x_0=0, x_2=0\}.
\end{align*}
Then $\psi$ is birational, its indeterminacy locus consists of the~curve $C$, and we have the~following $G$-equivariant commutative diagram:
\begin{equation}
\label{equation:55-1}
\xymatrix{
&\widetilde{X}\ar[dl]_{\alpha}\ar[dr]^{\beta}&\\
X\ar@{-->}[rr]_{\psi}&&V_4}
\end{equation}
where $\alpha$ is the~blow up of the~curve $C$, the~divisor $-K_{\widetilde{X}}$ is nef and big, and $\beta$ is a small birational morphism, which is given by the~linear system $|-K_{\widetilde{X}}|$. It follows from \eqref{equation:55-1} that $\mathrm{rk}\,\mathrm{Cl}(V_4)^G=2$, which implies that $V_4$ is not a $G$-Mori fiber space.

\begin{remark}
\label{remark:55-1-lines}
The lines $\ell_1$, $\ell_2$, $\ell_3$, $\ell_4$, $\ell_5$ form a pentagon configuration, so $C$ has $5$ singular points, which form the~$G$-orbit of the~point $[1:0:0:0:0]$. Thus, the~threefold $\widetilde{X}$ has $5$ isolated ordinary double points. Hence, $C$ is a (reducible) curve of degree $5$ and arithmetic genus $1$.
\end{remark}

Arguing as in \cite{CutroneMarshburn2013,ArapCutroneMarshburn2017,CutroneLimarziMarshburn2019,CutroneMarshburn2025,Kaloghiros2012,BlancLamy}, we see that  \eqref{equation:55-1} can be uniquely extended to a $G$-Sarkisov link that starts at $X$. Let us describe this $G$-Sarkisov link.

\begin{remark}
\label{remark:55-1-link}
If $C$ were a smooth elliptic curve, the~$G$-Sarkisov link induced by \eqref{equation:55-1} would be a~usual Sarkisov link that starts at $X$ and ends up at a smooth Fano threefold of Picard rank $1$ and anticanonical degree $14$. This Sarkisov link is classically known \cite{Fano,Iskovskikh1980,Puts}, and it appeared in \cite{Beauville2002,Kuznetsov2004,SashaYuraCostya,AtanasDima,AtanasLaurent,IskovskikhProkhorov,YuraZhijia}. Another Sarkisov link that starts at a smooth cubic threefold and ends at a smooth Fano threefold of Picard rank $1$ and anticanonical degree $14$ is described in \cite{Takeuchi1989,Tregub}.
\end{remark}

Thus, to construct the~required $G$-Sarkisov link, we can follow the~proof of \cite[Theorem~1.4]{Iskovskikh1980}. Namely, let $\mathcal{M}$ be the~linear system on $X$ that consists of the~following surfaces:
$$
\big\{a_0f_0+a_1f_1+a_2f_2+a_3f_3+a_4f_4+a_5f_5+a_6f_6+a_7f_7+a_8f_8+a_9f_9=0\big\}\cap X,
$$
where $[a_0:a_1:a_2:a_3:a_4:a_5:a_6:a_7:a_8:a_9]\in\mathbb{P}^9$. Then $\mathcal{M}$ is $G$-invariant, and the~constructed rational map $\chi$ is given by the~linear system $\mathcal{M}$. Note that the~base locus of this linear system consists of the~curve $C$. Moreover, we have
$$
\mathrm{mult}_{C}\big(\mathcal{M}\big)=4.
$$
Now, we set $S_i=X\cap\{x_i=0\}$ for $i\in\{0,1,2,3,4\}$. Then each $S_i$ is a singular cubic surface that has a Du Val singular point of type $\mathbb{D}_5$, each $\psi(S_i)$ is a plane in $V_4$, and each $\chi(S_i)$ is a line in $Y$. Set $L_1=\chi(S_0)$, $L_2=\chi(S_1)$, $L_3=\chi(S_2)$, $L_4=\chi(S_3)$, $L_5=\chi(S_4)$, and set
$$
Z=L_1+L_2+L_3+L_4+L_5.
$$
Then $L_1$, $L_2$, $L_3$, $L_4$, $L_5$ span a $4$-dimensional linear subspace $\Pi\subset\mathbb{P}^9$ such that
$$
Y\big\vert_{\Pi}=Z,
$$
and  $L_1$, $L_2$, $L_3$, $L_4$, $L_5$ form a pentagon configuration. So, the~arithmetic genus of $Z$ is $1$.

Let $\delta\colon\widehat{X}\to Y$ be the~blow up of the~curve $Z$. Then $\widehat{X}$ has $5$ isolated ordinary double points, the~divisor $-K_{\widehat{X}}$ is nef and big, and we have the~following $G$-equivariant commutative diagram:
\begin{equation}
\label{equation:55-1-half-link}
\xymatrix{
&\widehat{X}\ar[dl]_{\gamma}\ar[dr]^{\delta}&\\
V_4&&Y\ar@{-->}[ll]}
\end{equation}
where $\gamma$ is a small birational morphism given by $|-K_{\widehat{X}}|$, and $Y\dasharrow V_4$ is the~rational map induced by the~linear projection from $\Pi$. Now, combining \eqref{equation:55-1} and \eqref{equation:55-1-half-link}, we obtain the~following $G$-equivariant commutative diagram:
\begin{equation}
\label{equation:55-1-link}
\xymatrix{
\widetilde{X}\ar[d]_{\alpha}\ar[rr]^{\beta}&& V_4&&\widehat{X}\ar[ll]_{\gamma}\ar[d]^{\delta}\\
X\ar@{-->}[rrrr]\ar@{-->}[rru]_{\psi}&&&&Y\ar@{-->}[ull]}
\end{equation}
where $X\dasharrow Y$ is the~birational map induced by $\chi$. Note that  \eqref{equation:55-1-link} is a $G$-Sarkisov link, and $\delta$-exceptional surfaces are the~strict transforms of the~cubic surfaces $S_0$, $S_1$, $S_2$, $S_3$, $S_4$.

\section{Group $C_3^4\rtimes C_5$}
\label{section:405-15}

Let $X$ be a smooth cubic threefold in $\mathbb{P}^4$. Suppose that $\mathrm{Aut}(X)$ has a subgroup $G\simeq C_3^4\rtimes C_5$. Since the~$G$-action lifts to $\mathbb{P}^4$, we can consider $G$ as a subgroup in $\mathrm{PGL}_5(\mathbb{C})$. Moreover, it follows from Duncan's lemma \cite[Lemma 2.4]{AbbanCheltsovKishimotoMangolte2025} that the group $G$ lifts isomorphically to $\mathrm{GL}_5(\mathbb{C})$. Looking at representations of the group $G$, we see that the~embedding $G\hookrightarrow\mathrm{GL}_5(\mathbb{C})$ is given~by an~irreducible representation. 

The~group $G$ has $16$ irreducible $5$-dimensional representations, but all of them define the same subgroup in $\mathrm{PGL}_5(\mathbb{C})$ up to conjugation. Thus, we may assume that $G$ is generated by 
\begin{align*}
[x_0:x_1:x_2:x_3:x_4]&\mapsto [x_0:\zeta_3x_1:x_2:x_3:\zeta_3^2x_4],\\
[x_0:x_1:x_2:x_3:x_4]&\mapsto [x_0:\zeta_3x_1:x_2:\zeta_3^2x_3:x_4],\\
[x_0:x_1:x_2:x_3:x_4]&\mapsto [x_0:\zeta_3x_1:\zeta_3^2x_2:x_3:x_4],\\
[x_0:x_1:x_2:x_3:x_4]&\mapsto [\zeta_3x_0:\zeta_3^2x_1:x_2:x_3:x_4],\\
[x_0:x_1:x_2:x_3:x_4]&\mapsto [x_1:x_2:x_3:x_4:x_0].
\end{align*}
Then $X=\{x_0^3+\zeta_5^ax_1^3+\zeta_5^{2a}x_2^3+\zeta_5^{3a}x_3^3+\zeta_5^{4a}x_4^3=0\big\}\subset\mathbb{P}^4$ for some $a\in\{0,1,2,3,4\}$. Since all these cubic threefolds are $G$-biregular, we may assume that $a=0$, so $X$ is the~Fermat cubic \eqref{equation:Fermat}. 

\begin{theorem}
\label{theorem:405-15}
The cubic threefold $X$ is $G$-birationally superrigid.
\end{theorem}

Before proving this result, we state the following ten facts. Excepting $(1)$, all these facts can be checked by direct computations.
\begin{enumerate}
\item The group $G$ cannot act faithfully on a smooth curve  of genus $\leqslant 15$ \cite{Breuer,Jen,lmfdb}.
\item Every proper non-trivial subgroup of $G$ is isomorphic to $C_3$, $C_5$, $C_3^2$, $C_3^3$ or $C_3^4$.
\item The group $G$ contains a unique subgroup of index $5$, which is isomorphic to $C_3^4$.
\item If $H$ is the~subgroup of $G$ of index $5$, then $X$ does not contain $H$-invariant lines.
\item The cubic $X$ does not contain $G$-orbits of length $<15$;
\item The cubic $X$ contains only two $G$-orbits of length $15$, which can be described as follows:
\begin{align*}
\Sigma_{15}&=\mathrm{Orb}_G\big([0:0:0:1:-1]\big),\\
\Sigma_{15}^\prime&=\mathrm{Orb}_G\big([0:0:1:0:-1]\big).
\end{align*}
\item As a $G$-representation, $H^0(X,\mathcal{O}_{X}(-K_X))$ splits as a sum of three pairwise non-isomorphic $5$-dimensional irreducible representations, which correspond to the~$4$-dimensional linear subsystems in $|\mathcal{O}_{X}(2)|$ that can be described as follows:
\begin{itemize}
\item $\mathcal{H}$ is the~linear system that is cut out by quadrics
$$
a_0x_0x_2+a_1x_1x_3+a_2x_2x_4+a_3x_0x_3+a_4x_1x_4=0,
$$
\item $\mathcal{H}^\prime$ is the~linear system that is cut out by quadrics
$$
a_0x_0x_1+a_1x_1x_2+a_2x_2x_3+a_3x_3x_4+a_4x_0x_4=0,
$$
\item $\mathcal{H}^{\prime\prime}$ is the~linear system that is cut out by quadrics
$$
a_0x_0^2+a_1x_1^2+a_2x_2^2+a_3x_3^2+a_4x_4^2=0,
$$
\end{itemize}
where $[a_0:a_1:a_2:a_3:a_4]\in\mathbb{P}^4$.
\item The $G$-orbit $\Sigma_{15}$ is the~base locus of the linear system $\mathcal{H}$.
\item The $G$-orbit $\Sigma_{15}^\prime$ is the~base locus of the linear system $\mathcal{H}^\prime$.
\item The linear system $\mathcal{H}^{\prime\prime}$ is base point free.
\end{enumerate}

\begin{proof}[Proof of Theorem~\ref{theorem:405-15}]
Suppose $X$ is not $G$-birationally superrigid. Let us seek for a contradiction. 

It follows from \cite[Corollary 3.3.3]{CheltsovShramov2016} that there exists a non-empty $G$-invariant mobile linear subsystem $\mathcal{M}\subset |\mathcal{O}_X(n)|$, for some  $n\in\mathbb{Z}_{>0}$, such that the~log pair $(X,\frac{2}{n}\mathcal{M})$ is not canonical. Then, by \cite[Corollary~2.3]{CheltsovShramov2014}, the~log pair $(X,\frac{4}{n}\mathcal{M})$ is not log canonical. 

Choose a positive rational number $\mu<\frac{4}{n}$ such that the~singularities of  $(X,\mu\mathcal{M})$ are strictly log canonical. Let $C$ be an irreducible subvariety in $X$ that is a minimal center of log canonical singularities of the~log pair $(X,\mu\mathcal{M})$, and let $Z$ be the $G$-irreducible subvariety in $X$ whose irreducible component is $C$. Then every irreducible component of $Z$ is a minimal center of log canonical singularities of the log pair $(X,\mu\mathcal{M})$, which implies that these irreducible components are disjoint \cite[Proposition 1.5]{Kawamata-1}. Moreover, we know that $Z$ is not a surface, since $\mathcal{M}$ is mobile. Therefore, either $Z$ is a $G$-irreducible curve or $Z$ is the~$G$-orbit of a point. 

 Now, we are going to use \cite[Lemma 2.8]{cheltsov2011exceptional}, \cite[Lemma 2.4.10]{CheltsovShramov2016} or \cite[Lemma A.28]{CalabiBook}, which are equivariant versions of \emph{Tie Breaking} trick used in the~proof of \cite[Theorem 1.10]{Kawamata-1}. Namely, using this \textit{equivariant trick} together with \cite[Theorem 4.8]{Kollar1997}, we can replace the~rational number $\mu$ and the~mobile linear system $\mathcal{M}$ by a positive rational number $\lambda<2$ and an effective $G$-invariant $\mathbb{Q}$-divisor $D\sim_{\mathbb{Q}}-K_{X}$ such that the log pair $(X,\lambda D)$ is log canonical, $C$ is a minimal center of the~non-Kawamata log terminal singularities of this pair, and $\mathrm{Nklt}(X,\lambda D)=Z$,
where 
$$
\mathrm{Nklt}\big(X,\lambda D\big)=\big\{P\in X\ \text{such that $(X,\lambda D)$ is not Kawamata log terminal along $P$}\big\}.
$$

Let $\mathcal{I}_Z$ be the~ideal sheaf of the~subvariety $Z\subset X$. Then it follows from the Nadel vanishing theorem \cite{Nadel} in the~form of \cite[Theorem 9.4.8]{Lazarsfeld2004} that
$$
h^1\big(X,\mathcal{O}_{X}(-K_X)\otimes\mathcal{I}_Z\big)=0.
$$
Hence, we have the~following exact sequence of $G$-representations:
$$
0\longrightarrow H^0\big(X,\mathcal{I}_Z\otimes\mathcal{O}_{X}(-K_X)\big)\longrightarrow H^0\big(X,\mathcal{O}_{X}(-K_X)\big)\longrightarrow H^0\big(Z,\mathcal{O}_{Z}(-K_X\vert_{Z})\big)\longrightarrow 0,
$$
which gives
\begin{equation}
\label{equation:405-15}
h^0\big(Z,\mathcal{O}_{Z}(-K_X\vert_{Z})\big)=15-h^0\big(X,\mathcal{I}_Z\otimes\mathcal{O}_{X}(-K_X)\big),
\end{equation}
because $h^0(X,\mathcal{O}_{X}(-K_X))=15$.  Thus, if $Z$ is a $G$-orbit of a point, \eqref{equation:405-15} gives
$$
|Z|=h^0(Z,\mathcal{O}_{Z})=15-h^0\big(X,\mathcal{I}_Z\otimes\mathcal{O}_{X}(-K_X)\big)\leqslant 15,
$$
so $Z$ is a $G$-orbit of length $15$ and $h^0(X,\mathcal{I}_Z\otimes\mathcal{O}_{X}(-K_X))=0$, which implies that $Z$ is not contained in any surface in $|-K_X|$.
On the~other hand, the~only $G$-orbits of length $15$ in $X$ are $\Sigma_{15}$ and~$\Sigma_{15}^\prime$, and each of them is contained in many surfaces in $|-K_X|$, since $\Sigma_{15}$ is the~base locus of the~linear system $\mathcal{H}$, and $\Sigma_{15}^\prime$ is the~base locus of the~linear system $\mathcal{H}^\prime$.
Thus, $Z$ is not a $G$-orbit of a point.

We see that $Z$ is a curve. Then it follows from \cite{Kawamata-2} that $C$ is smooth, so $Z$ is also smooth. Moreover, we have
$$
h^0\big(X,\mathcal{I}_Z\otimes\mathcal{O}_{X}(-K_X)\big)=0,
$$
because every non-empty $G$-invariant linear subsystem in $|-K_X|$ contains at least one linear subsystem among $\mathcal{H}$, $\mathcal{H}^\prime$, $\mathcal{H}^{\prime\prime}$, and the base loci of $\mathcal{H}$, $\mathcal{H}^\prime$, $\mathcal{H}^{\prime\prime}$ contain no curves.
Thus, \eqref{equation:405-15} gives
\begin{equation}
\label{equation:405-15-easy}
h^0\big(Z,\mathcal{O}_{Z}(-K_X\vert_{Z})\big)=15.
\end{equation}

Let $r$ be the~number of irreducible components of the curve $Z$. Using the fact $(1)$, we get $r>1$. Let $d$ be the~degree of the~curve~$C$, let $g$ be the~genus of the curve $C$, and let $A=\epsilon (-K_X)$ for some sufficiently small positive rational number $\epsilon$. Then, it follows from the~Kawamata subadjunction theorem \cite[Theorem 1]{Kawamata-2} that
$$
\big(K_{X}+\lambda D+A\big)\big\vert_{C}\sim_{\mathbb{Q}} K_C+\Delta_C
$$
for some effective $\mathbb{Q}$-divisor $\Delta_C$ on the~curve $C$. This gives
$$
(\lambda-1+\epsilon)2d\geqslant 2g-2.
$$
Since this inequality holds for any $\epsilon>0$, we get $2d>2g-2$, because $\lambda<2$. Then $g\leqslant d$, which implies that the~divisor $-K_X\vert_{C}$ is not special. Using \eqref{equation:405-15-easy} and the~Riemann--Roch formula, we get
$$
15=h^0\big(Z,\mathcal{O}_{Z}(-K_X\vert_{Z})\big)=r(2d-g+1).
$$
On the~other hand, going through the~subgroups of $G$, we see that $r\in\{1,5,15,45,81,135,405\}$.

If $g=0$, then $15=r(2d+1)$, which implies that $r=5$ and $d=1$, so $C$ is a line, and its stabilizer in $G$ is the~unique subgroup of index $5$. As we mentioned earlier, $X$ does not contain such lines. Hence,  $g\ne 0$. Then $d\geqslant 3$ and
$$
15=r(2d-g+1)\geqslant r(d+1)\geqslant 4r,
$$
which implies that $r=1$, so $C=Z$ and $G$ acts faithfully on $C$. Now, using $g\leqslant d$, we get
$$
15=2d-g+1\geqslant d+1,
$$
so $d\leqslant 14$. Then $g\leqslant d\leqslant 14$. But $G$ cannot act faithfully on a smooth curve  of genus $\leqslant 15$. 
\end{proof}

\section{Group $\mathfrak{A}_5$}
\label{section:A5}

Let $X$ be a smooth cubic hypersurface in $\mathbb{P}^4$ such that $\mathrm{Aut}(X)$ contains a subgroup $G\simeq\mathfrak{A}_5$. The goal of this section is to prove the following result.

\begin{theorem}
\label{theorem:A5}
The cubic $X$ is $G$-birationally superrigid.
\end{theorem}

First, we recall the following result from \cite{CheltsovShramov2016}.

\begin{lemma}
\label{lemma:A5-curves}
Let $C$ be a smooth irreducible curve that admits a non-trivial action of the group $G$. Then the length of a $G$-orbit in $C$ is $12$, $20$, $30$ or $60$. Moreover, if $g\leqslant 16$, then 
$$
g\in\big\{0,4,5,6,9,10,11,13,15,16\big\},
$$
and the possible numbers of $G$-orbits in $C$ consisting of $12$, $20$, $30$ points are contained in the table below, where $a_k$ is the~number of $G$-orbits in $C$.
\begin{center}\renewcommand\arraystretch{1.1}
\begin{tabular}{|c||c|c|c|c|c|c|c|c|c|c|c|}
\hline
$g$ &      $0$ & $4$ & $5$ & $6$ & $9$ & $10$ & $11$ & $13$ & $15$ & $16$ & $16$\\
\hline
$a_{12}$ & $1$ & $2$ & $1$ & $0$ & $2$ & $1$  & $0$  & $3$  & $1$  & $0$  & $0$ \\
\hline
$a_{20}$ & $1$ & $0$ & $2$ & $1$ & $1$ & $0$  & $2$  & $0$  & $1$  & $3$  & $0$ \\
\hline
$a_{30}$ & $1$ & $1$ & $0$ & $3$ & $0$ & $3$  & $2$  & $0$  & $2$  & $1$  & $5$\\
\hline
\end{tabular}
\end{center}
\end{lemma}

\begin{proof}
This is \cite[Lemma 5.1.4]{CheltsovShramov2016} and \cite[Lemma 5.1.5]{CheltsovShramov2016}, see also \cite{lmfdb,Jen,Breuer}.
\end{proof}

Since the~$G$-action on the cubic $X$ lifts to $\mathbb{P}^4$, we can also consider $G$ as a subgroup in $\mathrm{PGL}_5(\mathbb{C})$. Moreover, the subgroup $G\subset\mathrm{PGL}_5(\mathbb{C})$ lifts isomorphically to $\mathrm{GL}_5(\mathbb{C})$ by \cite[Lemma 2.4]{AbbanCheltsovKishimotoMangolte2025}. Note that the~embedding $G\hookrightarrow\mathrm{GL}_5(\mathbb{C})$ is given by a faithful $5$-dimensional $G$-representation $\mathbb{V}_5$.
Then either $\mathbb{V}_5$ is the unique irreducible $5$-dimensional representation of the group $G$, or
$$
\mathbb{V}_5\simeq \mathbb{I}\oplus\mathbb{V}_4,
$$
where $\mathbb{I}$ is the trivial representation, and $\mathbb{V}_4$ is the unique (standard) $4$-dimensional irreducible representation. If $\mathbb{V}_5$ is reducible, we say that $G$ is a \textit{standard} subgroup of the~group $\mathrm{PGL}_5(\mathbb{C})$.
If~$\mathbb{V}_5$ is irreducible, we say that $G$ is a \textit{non-standard} subgroup of the group~$\mathrm{PGL}_5(\mathbb{C})$. Set
$$
M_1=\begin{pmatrix}
0 & 0 & 0 & 1 & 0 \\
1 & 0 & 0 & 0 & 0 \\
0 & 0 & 1 & 0 & 0 \\
0 & 1 & 0 & 0 & 0 \\
0 & 0 & 0 & 0 & 1
\end{pmatrix}, M_2=\begin{pmatrix}
0 & 0 & 0 & 1 & 0 \\
1 & 0 & 0 & 0 & 0 \\
0 & 0 & \zeta_3& 0 & 0 \\
0 & 1 & 0 & 0 & 0 \\
0 & 0 & 0 & 0 & \zeta_3^2
\end{pmatrix},
M_3=\begin{pmatrix}
0 & 0 & 0 & 0 & 1 \\
0 & 1 & 0 & 0 & 0 \\
0 & 0 & 0 & 1 & 0 \\
0 & 0 & 1 & 0 & 0 \\
1 & 0 & 0 & 0 & 0
\end{pmatrix}. 
$$
Choosing appropriate coordinates on $\mathbb{P}^4$, we may assume that one of the following cases holds:
\begin{enumerate}
\item $G=\langle M_1,M_3\rangle$, and $G$ is a standard subgroup of $\mathrm{PGL}_5(\mathbb{C})$;
\item $G=\langle M_2,M_3\rangle$, and $G$ is a non-standard subgroup of $\mathrm{PGL}_5(\mathbb{C})$.
\end{enumerate}

\begin{lemma}
\label{lemma:A5-standard-equation}
Suppose that $G=\langle M_1,M_3\rangle$. Then $X$ can be given by
\begin{equation}
\label{equation:A5-standard}
x_0^3+x_1^3+x_2^3+x_3^3+x_4^3+a(x_0+x_1+x_2+x_3+x_4)(x_0^2+x_1^2+x_2^2+x_3^2+x_4^2)=0
\end{equation}
for some $a\in\mathbb{C}$ such that $a\ne -\frac{1}{5}$, $80a^3+144a^2+135a+27\ne 0$ and $40a^3+72a^2+45a+9\ne 0$.
\end{lemma}

\begin{proof}
It is well-known that the cubic threefold $X$ is given by
$$
\sum_{i=0}^4x_i^3+a\Big(\sum_{i=0}^4x_i\Big)\Big(\sum_{i=0}^4x_i^2\Big)+b\Big(\sum_{i=0}^4x_i\Big)^3=0
$$
for some $a,b\in\mathbb{C}$, where $(a,b)\ne (-\frac{3}{5},\frac{2}{25})$, since otherwise $X$ would be singular. Take some $s\in\mathbb{C}^*$. Then, applying the projective transformation
$$
\begin{pmatrix}
4+s&s-1&s-1&s-1&s-1\\
s-1&4+s&s-1&s-1&s-1\\
s-1&s-1&4+s&s-1&s-1\\
s-1&s-1&s-1&4+s&s-1\\
s-1&s-1&s-1&s-1&4+s
\end{pmatrix},
$$
we transform the equation of $X$ into
$$
\sum_{i=0}^4x_i^3+\frac{(5a+3)s-3}{5}\Big(\sum_{i=0}^4x_i\Big)\Big(\sum_{i=0}^4x_i^2\Big)+\frac{2+(5a+25b+1)s^3-(5a+3)s}{25}\Big(\sum_{i=0}^4x_i\Big)^3=0.
$$
Note that $2+(5a+25b+1)s^3-(5a+3)s$ is a non-constant polynomial in $s$, because $(a,b)\ne (-\frac{3}{5},\frac{2}{25})$. Therefore,  we can choose $s\in\mathbb{C}^*$ such that $2+(5a+25b+1)s^3-(5a+3)s=0$, which implies that we may assume that $b=0$, so the cubic $X$ is given by \eqref{equation:A5-standard} for some $a\in\mathbb{C}$ as claimed.

If $a=-\frac{1}{5}$, then $X$ has a node at $[1:1:1:1:1]$. Similarly, if  $80a^3+144a^2+135a+27=0$, then $X$ has five nodes. Finally, if $40a^3+72a^2+45a+9=0$, then $X$ has ten nodes.
\end{proof}

The proof of the following result is very similar to the proof of \cite[Proposition 5.1]{PinardinZhang2025}.

\begin{proposition}
\label{proposition:A5-standard}
Suppose that $G=\langle M_1,M_3\rangle$. Then $X$ is $G$-birationally superrigid.
\end{proposition}

\begin{proof}
By Lemma~\ref{lemma:A5-standard-equation}, we may assume that $X$ is given by \eqref{equation:A5-standard}.
Let's describe small $G$-orbits in~$X$. First, we observe that the cubic threefold $X$ does not contain $G$-orbits of length $6$,
which implies that all $G$-orbits in $X$ of length $<12$ have length $5$ or $10$.

Observe that all points in $\mathbb{P}^4$ with stabilizer isomorphic to $\mathfrak{A}_4$ are in the $G$-orbit of the points of the form $[s:s:s:s:t]$.
Now, to describe $G$-orbits in $X$ of length $5$, take  $[s:t]\in\mathbb{P}^1$ such that
$$
(16a+4)s^3+4as^2t+4ast^2+(a+1)t^3=0,
$$
then $[s:s:s:s:t]\in X$. Note that there are exactly three possibilities for $[s:s:s:s:t]$, because the discriminant of the polynomial $(16a+4)s^3+4as^2t+4ast^2+(a+1)t^3$ is not zero, since $a\ne -\frac{1}{5}$ and $80a^3+144a^2+135a+27\ne 0$ by Lemma~\ref{lemma:A5-standard-equation}.
Hence, $X$ has three $G$-orbits of length $5$.

To describe $G$-orbits of length $10$ in $X$, we let $O=[0:0:0:1:-1]$, let 
$$
\Sigma_{10}=\mathrm{Orb}_G(O),
$$
and let $G_O$ be the stabilizer in $G$ of the point $O$. Then $\Sigma_{10}\subset X$, $|\Sigma_{10}|=10$, $G_O\simeq\mathfrak{S}_3$, and the~group $G_O$ fixes $3$ extra points in $X$, which can be described as follows. Take $[s:t]\in\mathbb{P}^1$~such~that
$$
(3+9a)s^3+6ats^2+6at^2s+(4a+2)t^3=0.
$$
Then $[s:s:s:t:t]\in X$, and it is fixed by $G_O$. This implies that the $G$-orbit of the point $[s:s:s:t:t]$ has length $10$, and every $G$-orbit of length $10$ in $X$ different from $\Sigma_{10}$ can be obtained in this way. Note that there are exactly three possibilities for $[s:s:s:t:t]$, because the~discriminant of the polynomial $(3+9a)s^3+6ats^2+6at^2s+(4a+2)t^3$ is not zero, since $a\ne -\frac{1}{5}$ and $40a^3+72a^2+45a+9\ne 0$ by Lemma~\ref{lemma:A5-standard-equation}. Thus, $X$ has four $G$-orbits of length $10$.

We have described all $G$-orbits in $X$ of length $<12$. Now, we set
$$
S=X\cap\big\{x_0+x_1+x_2+x_3+x_4=0\big\}.
$$
Then $S$ is a smooth cubic surface, the group $G$ acts faithfully on $S$, and $S$ is the so-called Clebsch cubic surface. Observe that $\Sigma_{10}\subset S$, and $\Sigma_{10}$ is the unique $G$-orbit in $S$ of length $<12$, which also follows from \cite[Lemma 6.3.12]{CheltsovShramov2016}. Moreover, it follows from \cite[Lemma 5.8]{PinardinZhang2025} that $\alpha_G(S)=2$. Furthermore, $S$ does not contain $G$-invariant curves of degree $<6$ by \cite[Theorem 6.3.18]{CheltsovShramov2016}.

We claim that $X$ contains no $G$-invariant curves of degree $<12$ that are not contained in~$S$. Indeed, if $Z$ is a $G$-irreducible curve  in $X$ of degree $d<12$ and $Z\not\subset S$, then $S\cap Z$ is a $G$-invariant subset consisting of at most $S\cdot Z=d$ points, which must be a union of $G$-orbits. Thus, we have
$$
S\cap Z=\Sigma_{10},
$$
so $d=S\cdot Z=10$, which implies that $X$ does not contain $G$-invariant curves of degree $<6$. Moreover, since 
$$
|S\cap Z|=|\Sigma_{10}|=10=S\cdot Z,
$$
the surface $S$ intersects $Z$ transversally at every point of $\Sigma_{10}$, and, therefore $Z$ must be smooth at every point of $\Sigma_{10}$. This implies that $Z$ is reducible, since otherwise $G$ would act faithfully on $Z$, and the stabilizer of any point in $\Sigma_{10}\subset Z\setminus\mathrm{Sing}(Z)$ is cyclic, which is not the case. Then
\begin{enumerate}
\item either $Z$ is a union of $10$ lines;
\item or $Z$ is a union of $5$ smooth conics.
\end{enumerate}
Let $C$ be an irreducible component of $Z$, and let $G_C$ be its stabilizer in $G$. If $Z$ consists of $10$ lines, then $G_C\simeq\mathfrak{S}_3$, so we may assume that $G_C=G_O$. Then, since $S\cap C$ is a  $G_O$-fixed point, we see that $S\cap C=[0:0:0:1:-1]$, and $G_O$ fixes another point in $C$, which is not contained in $S$. Hence, $C$ contains $[s:s:s:t:t]$ for $[s:t]\in\mathbb{P}^1$ such that $(3+9a)s^3+6ats^2+6at^2s+(4a+2)t^3=0$,
which gives $80a^3+144a^2+135a+27=0$, which is impossible by Lemma~\ref{lemma:A5-standard-equation}. Hence, we conclude that $C$ is a smooth conic. Let $\Pi_C$ be the plane in $\mathbb{P}^4$ that contains $C$. Then
$$
\Pi_C\cap X=C+L,
$$
for a line $L$. Now, applying $G$ to the line $L$, we obtain a $G$-irreducible curve in $X$ of degree $\leqslant 5$, which does not exist. So, $X$ contains no $G$-irreducible curves of degree $<12$ not contained in $S$.

Now, we suppose that $X$ is not $G$-birationally superrigid.
Then, by \cite[Corollary 3.3.3]{CheltsovShramov2016}, there exists a non-empty $G$-invariant mobile linear subsystem
$\mathcal{M}\subset |nS|$, for a positive integer~$n$, such that $(X,\frac{2}{n}\mathcal{M})$ is not canonical.
Let 
$$
\Sigma=\big\{P\in X\ \text{such that $(X,\frac{2}{n}\mathcal{M})$ is not canonical at~$P$}\big\}.
$$
We claim that $\Sigma$ is a finite subset.

Indeed, suppose that $\Sigma$ is not finite subset. Then $\Sigma$ contains a $G$-irreducible curve $Z$, and it follows from \cite[Exercise 6.18]{KollarSmithCorti} that $\mathrm{mult}_Z\big(\mathcal{M}\big)>\frac{n}{2}$. Let $M$ and $M^\prime$ be general surfaces in $\mathcal{M}$. Then
$\mathrm{mult}_Z\big(M\big)=\mathrm{mult}_Z\big(M^\prime\big)=\mathrm{mult}_Z\big(\mathcal{M}\big)>\frac{n}{2}$. Then, in particular, we have $M\cdot M^\prime=mZ+\Delta$, where $m$ is a positive integer such that $m\geqslant\mathrm{mult}^2_Z(\mathcal{M})$, and $\Delta$ is an effective one-cycle on $X$ whose support does not contain $Z$. Then
$$
3n^2=S\cdot M\cdot M^\prime=m\mathrm{deg}(Z)+S\cdot\Delta\geqslant\mathrm{mult}^2_Z\big(\mathcal{M}\big)\mathrm{deg}(Z)>\frac{n^2}{4}\mathrm{deg}(Z),
$$
which gives $\mathrm{deg}(Z)<12$. Thus, we conclude that $Z\subset S$ and $\mathrm{deg}(Z)\geqslant 6$. Then 
$$
n(-K_S)\sim M\big\vert_{S}=m_SZ+\Delta_S
$$
for some $m_S\geqslant \mathrm{mult}_Z(\mathcal{M})>\frac{n}{2}$ and some effective divisor $\Delta_S$ on the~surface $S$ whose support does not contain the curve $Z$. Then 
$$
3n=-K_S\cdot (m_SZ+\Delta_S)=m_S\mathrm{deg}(Z)-K_S\cdot\Delta_S\geqslant m_S\mathrm{deg}(Z)\geqslant \mathrm{mult}_Z(\mathcal{M})\mathrm{deg}(Z)>\frac{n\mathrm{deg}(Z)}{2},
$$
so $\mathrm{deg}(Z)<6$, which is impossible, since $S$ contains no $G$-invariant curves of degree $<6$.

Therefore, we conclude that $\Sigma$ is a finite subset. Then $\Sigma\cap S=\varnothing$. Indeed, if $\Sigma\cap S\ne\varnothing$, then the log pair $(S,\frac{2}{n}\mathcal{M}\vert_{S})$ is not log canonical by the~inversion of adjunction \cite[Theorem 5.50]{KollarMori}, so~$\alpha_G(S)<2$, which is impossible, since $\alpha_G(S)=2$ by \cite[Lemma 5.8]{PinardinZhang2025}.

Arguing as in the proof of \cite[Lemma 4.6]{Avilov2018}, we see that $\Sigma$ contains no $G$-orbits of length~$5$. Indeed, suppose that $\Sigma$ contains a $G$-orbit $\Sigma_5$ that has length $5$. Since $\Sigma_5$ is not contained in $S$, the points of $\Sigma_5$ are in linearly general position. Further, it follows from \cite{Pukhlikov1998,Corti2000} that 
$$
\big(M\cdot M^\prime)_P>n^2,
$$
for every point $P\in\Sigma$, where as above, $M$ and $M^\prime$ are general surfaces in $\mathcal{M}$.  Let $\mathcal{P}$ be the linear subsystem in $|\mathcal{O}_{X}(3)|$ consisting of all surfaces that are singular at every point of the $G$-orbit $\Sigma_5$. Then $\mathcal{P}$ is non-empty, mobile, and its base locus consists of the $G$-orbit $\Sigma_5$, because $X$ does not contain lines in $\mathbb{P}^4$ that pass through pair of points of $\Sigma_5$. Hence, if $\mathscr{S}$ is a general surface in $\mathcal{P}$, then
$$
9n^2=\mathscr{S}\cdot M\cdot M^\prime\geqslant\sum_{P\in\Sigma_5}\mathrm{mult}_P(\mathscr{S})\big(M\cdot M^\prime)_P>2n^2|\Sigma_5|=10n^2,
$$
which is absurd. Therefore, we conclude that $\Sigma$ does not contain $G$-orbits of length $5$.

Now, we fix a point $P\in\Sigma$, and let $\Sigma_P$ be the $G$-orbit of the point $P$. Then $\Sigma_P\not\subset S$ and $|\Sigma_P|>5$. Choose $\mu\in\mathbb{Q}_{>0}$ such that the~log pair  $(X,\mu\mathcal{M})$ is strictly log canonical at~$P$. Then 
$$
\mu<\frac{3}{n}
$$ 
by \cite[Remark 3.6]{CheltsovSarikyanZhuang2024}. Recall from \cite{CheltsovShramov2016} that
$$
\mathrm{Nklt}\big(X,\mu\mathcal{M}\big)=\big\{P\in X\ \text{such that $(X,\mu\mathcal{M})$ is not Kawamata log terminal at $P$}\big\}.
$$
We claim that the locus $\mathrm{Nklt}(X,\mu\mathcal{M})$ does not contain curves that are not disjoint from $\Sigma_P$. Indeed, let $M$ and $M^\prime$ be general surfaces in $\mathcal{M}$.
If $Z$ is a $G$-irreducible curve in $\mathrm{Nklt}(X,\mu\mathcal{M})$, then it follows from \cite[Lemma 1.8]{Corti2000} that
$$
\big(M\cdot M^\prime\big)_Z\geqslant\frac{4}{\mu^2}>\frac{4n^2}{9}
$$
which gives $\mathrm{deg}(Z)\leqslant 6$, so $Z\subset S$, and, therefore $Z\cap\Sigma_P=\varnothing$. In fact, using \cite[Theorem~1.2]{DemaillyPham2014}, we get $\mathrm{deg}(Z)<6$, which would imply that $\mathrm{Nklt}(X,\mu\mathcal{M})$ consists of finitely many points.

We see that the points of the $G$-orbit $\Sigma_P$ are isolated components of the locus $\mathrm{Nklt}(X,\mu\mathcal{M})$. Now, using the~Nadel vanishing theorem as in the~proof of Theorem~\ref{theorem:405-15}, we see that
$$
5<\big|\Sigma_P\big|\leqslant h^0\big(X,\mathcal{O}_{X}(1)\big)=5,
$$
which is absurd. The~obtained contradiction completes the~proof of Proposition~\ref{proposition:A5-standard}.
\end{proof}

By Proposition~\ref{proposition:A5-standard}, to prove Theorem~\ref{theorem:A5}, we may assume that $G=\langle M_2,M_3\rangle$. Then, arguing as in the proof of Lemma~\ref{lemma:A5-standard-equation}, we see that $X$ is given by
\begin{multline*}
\quad \quad \quad \quad x_0^3+x_1^3+x_2^3+x_3^3+x_4^3+b(x_0x_1x_3+\zeta_3^2x_0x_1x_2+\zeta_3x_0x_1x_4+x_2x_3x_4+\\
+x_1x_2x_4+\zeta_3^2x_1x_3x_4+x_0x_3x_4+\zeta_3x_1x_2x_3+x_3x_2x_0+x_0x_2x_4)=0\quad \quad \quad \quad
\end{multline*}
for some $b\in\mathbb{C}$. Furthermore, we have $4b^2-3b+9\ne 0$, and 
$$
b\not\in\{-1,1,-3\},
$$
because the cubic threefold $X$ would be singular otherwise. In the remaining part of this section, we will prove that $X$ is $G$-birationally superrigid arguing as in the proof of \cite[Proposition~7.1]{PinardinZhang2025}.

As in the proof of Proposition~\ref{proposition:A5-standard}, we start with the classification of orbits of small length in $X$. First, we present some $G$-orbits in $X$. Let
\begin{align*}
\Sigma_{12}&=\mathrm{Orb}_G\big([\zeta_{15}:\zeta_{15}^9:1:\zeta_{15}^3:\zeta_{15}^7]\big),\\
\Sigma_{12}^\prime&=\mathrm{Orb}_G\big([\zeta_{15}^4:1:\zeta_{15}^{12}:\zeta_{15}^3:\zeta_{15}]\big).
\end{align*}
Then $\Sigma_{12}$ and $\Sigma_{12}^\prime$ are $G$-orbits of length $12$, which are contained in $X$. Let
$$
\Sigma_{15}=\mathrm{Orb}_G\big([-1:0:1:-1:1]\big),
$$
and let $\mathcal{L}_5$ be the $G$-irreducible curve  in $\mathbb{P}^4$ whose irreducible component is the line
$$
\big\{x_0=x_2, x_0=x_3, x_0=x_4\big\}\subset\mathbb{P}^4.
$$
Then $\Sigma_{15}\subset X$, $|\Sigma_{15}|=15$, $\mathcal{L}_{5}$ consists of $5$ disjoint lines, and $\mathcal{L}_{5}\cap X$ is a $G$-orbit of length $15$
$$
\mathrm{Orb}_G\big(\big[1:t:1:1:1]\big),
$$
where $t\in\mathbb{C}$ such that $t^3+4b+4=0$. Let $\mathcal{L}_{20}$ be the $G$-irreducible curve in $\mathbb{P}^4$ whose irreducible component is the line $\{x_0=\zeta_3x_1, x_1=x_2, x_4=0\}\subset\mathbb{P}^4$. Note that $\mathcal{L}_{20}$ consists of $20$ distinct lines, and $\mathcal{L}_{20}\cap X$ is a union of three $G$-orbits 
$$
\mathrm{Orb}_G\big(\big[\zeta_3:1:1:t:0]\big),
$$
where $t\in\mathbb{C}$ such that $t^3+3b\zeta_3t+b+3=0$. 

\begin{lemma}
\label{lemma:A5-orbits}
Let $\Sigma$ be a $G$-orbit in $X$ of length $|\Sigma|\leqslant 20$. Then $|\Sigma|\in\{12,15,20\}$ and
$$
\Sigma\subseteq\Sigma_{12}\cup \Sigma_{12}^\prime\cup \Sigma_{15}\cup\big(\mathcal{L}_{5}\cap X\big)\cup\big(\mathcal{L}_{20}\cap X\big).
$$
\end{lemma}

\begin{proof}
Let $P$ be a point in $\Sigma$, and let $G_P$ be its stabilizer in $G$. Then $G_P$ is a proper subgroup, because $G$ does not fix points in the cubic threefold $X$. Thus, $|\Sigma|=|G/G_P|\in\{5,6,10,12,15,20\}$, and $G_P$ is isomorphic to one of the following groups:
\begin{center}
$\mathfrak{A}_4$, $\mathfrak{D}_5$, $\mathfrak{S}_3$, $C_5$, $C_2^2$, $C_3$.
\end{center}
Note also that $G$ contains exactly one subgroup isomorphic to $G_P$ up to conjugation.

If $G_P\simeq \mathfrak{A}_4$, then $G_P$ fixes two points in $\mathbb{P}^4$, but none of them is contained in $X$, so that $|\Sigma|\ne 5$. Similarly, if $G_P\simeq \mathfrak{D}_5$ or $G_P\simeq \mathfrak{S}_3$, then $G_P$ fixes one point in $\mathbb{P}^4$, but it is not contained in $X$, which implies that $|\Sigma|\ne 6$ and $|\Sigma|\ne 10$ either.

If $|\Sigma|=12$, then $G_P\simeq C_5$ fixes $5$ points in $\mathbb{P}^4$, but only $4$ of them are contained in $X$, and these points are contained in $\Sigma_{12}\cup\Sigma_{12}^\prime$, which implies that $\Sigma\subset\Sigma_{12}\cup\Sigma_{12}^\prime$,

If $|\Sigma|=15$, then $G_P\simeq C_2^2$ fixes three points in $\Sigma_{15}$, and $G_P$ pointwise fixes an irreducible component of the curve $\mathcal{L}_5$, which implies that $\Sigma\subset\Sigma_{15}\cup(\mathcal{L}_{5}\cap X)$.

Finally, if $|\Sigma|=20$, then $G_P\simeq C_3$ pointwise fixes two irreducible components of the curve $\mathcal{L}_{20}$, which implies that $\Sigma\subset\mathcal{L}_{20}$.
\end{proof}

Now, we describe some reducible $G$-irreducible curves in $X$. Let $\mathcal{L}_6$ be the $G$-irreducible curve in $\mathbb{P}^4$ such that its irreducible component is the line
$$
\left\{\aligned
&x_0-(\zeta_{15}^7+\zeta_{15}^{13})x_1+x_4=0,\\
&\zeta_{15}^5x_0+x_3+(\zeta_{15}+\zeta_{15}^6-\zeta_{15}^{14})x_4=0,\\
&(\zeta_{15}+\zeta_{15}^6-\zeta_{15}^{14})x_0+x_2+\zeta_{15}^5x_4=0.
\endaligned
\right.
$$
and let $\mathcal{L}_6^\prime$ be the $G$-irreducible curve in $\mathbb{P}^4$  such that its irreducible component is the line
$$
\left\{\aligned
&(\zeta_{15}^7+\zeta_{15}^4-\zeta_{15}^3+\zeta_{15}^2+\zeta_{15}-1)x_0+x_1+x_3=0,\\
&(\zeta_{15}^7-\zeta_{15}^3+\zeta_{15}^2)x_1+x_2+x_3=0,\\
&x_0-(\zeta_{15}^4+\zeta_{15})x_1+x_4=0.
\endaligned
\right.
$$
Then $\mathcal{L}_6\cup\mathcal{L}_6^\prime\subset X$. Moreover, each curve $\mathcal{L}_6$ and $\mathcal{L}_6^\prime$ consists of $6$ disjoint lines.

\begin{lemma}
\label{lemma:A5-L6}
Let $Z$ be one of the curves $\mathcal{L}_6$ or $\mathcal{L}_6^\prime$, and let $\mathcal{I}_Z$ be its ideal sheaf in $\mathbb{P}^4$. Then 
$$
h^0\big(\mathbb{P}^4,\mathcal{I}_Z\otimes\mathcal{O}_{\mathbb{P}^4}(3)\big)=11,
$$
and $\mathcal{I}_Z\otimes\mathcal{O}_{\mathbb{P}^4}(3)$ is generated by global sections.
\end{lemma}

\begin{proof}
First, we found a basis of $H^0(\mathbb{P}^4,\mathcal{I}_Z\otimes\mathcal{O}_{\mathbb{P}^4}(3))$ using Maple, and then we used Magma to verify that the polynomials of this basis determine $Z$ as a subscheme in $\mathbb{P}^4$. 
\end{proof}

For every $t\in\mathbb{C}$, let $\mathcal{L}_{10}^t$ be the $G$-irreducible curve in $\mathbb{P}^4$ whose irreducible component is the line
\begin{equation}
\label{equation:A5-10-lines}
\left\{
\begin{aligned}
&x_0=t(\zeta_3x_3+x_4),\\
&x_1=t(x_3+x_4),\\
&x_2=t(x_3+\zeta_3x_4).
\end{aligned}
\right.
\end{equation}
Then $\mathcal{L}_{10}^t$ consists of $10$ lines, and direct substitution gives
$$
\mathcal{L}_{10}^t\subset X
\iff (b+3)t^3+3b\zeta_3t^2+1=0.
$$
Since $X$ is smooth, this polynomial has degree $3$ and three distinct roots. Thus, there are exactly three possibilities for $\mathcal{L}_{10}^t$ such that $\mathcal{L}_{10}^t\subset X$. Moreover, the following holds.
\begin{itemize}
\item If $(b+3)t^3+3b\zeta_3t^2+1=0$ and $t\ne\pm\frac{\zeta_3+2}{3}$, then $\mathcal{L}_{10}^t$ is a union of $10$ disjoint lines.
\item If $t=\pm\frac{\zeta_3+2}{3}$, then $\mathcal{L}_{10}^t$ is a nodal curve of arithmetic genus $6$, and $\mathrm{Sing}(\mathcal{L}_{10}^t)=\Sigma_{15}$.
\end{itemize}

\begin{lemma}
\label{lemma:A5-L10}
Let $Z=\mathcal{L}^t_{10}$ for $t=\pm\frac{\zeta_3+2}{3}$, and let $\mathcal{I}_Z$ be the ideal sheaf of the curve $Z$ in $\mathbb{P}^4$. Then 
$$
h^0\big(\mathbb{P}^4,\mathcal{I}_Z\otimes\mathcal{O}_{\mathbb{P}^4}(3)\big)=10,
$$
and $\mathcal{I}_Z\otimes\mathcal{O}_{\mathbb{P}^4}(3)$ is generated by global sections.
\end{lemma}

\begin{proof}
Explicit computations. For instance, if $t=\frac{\zeta_3+2}{3}$, a basis of  $H^0(\mathbb{P}^4,\mathcal{I}_Z\otimes\mathcal{O}_{\mathbb{P}^4}(3))$ consists of
\begin{multline*}
f_1=-\zeta_3x_0^3+(-1-2\zeta_3)x_0^2x_1+(-\zeta_3-2)x_2x_0^2+\zeta_3x_0x_1^2-\\
4\zeta_3x_0x_1x_3+(2\zeta_3+2)x_4x_0x_1+(1+\zeta_3)x_0x_2^2+4x_3x_2x_0+\\
(-2\zeta_3-2)x_4x_0x_2-2x_3^2x_0+(-1-2\zeta_3)x_4^2x_0+(-1-2\zeta_3)x_1^3+\zeta_3x_1^2x_2+\\
2\zeta_3x_1^2x_4+(\zeta_3-2)x_2^2x_1+(-4\zeta_3-2)x_4x_1x_2+(-1-\zeta_3)x_4^2x_1+\\
(-1-\zeta_3)x_2^3-2x_2^2x_4+2x_3^2x_2-x_4^2x_2-2x_3^2x_4+(-2\zeta_3-2)x_4^3,
\end{multline*}
\begin{multline*}
f_2=(-1-2\zeta_3)x_0^3-3\zeta_3x_0^2x_1+(-3\zeta_3-3)x_2x_0^2+6x_0^2x_4+(-6\zeta_3+6)x_3x_0x_1+\\
(\zeta_3-1)x_0x_2^2+(-6\zeta_3-6)x_4x_0x_2+(-4\zeta_3-2)x_3^2x_0-6\zeta_3x_0x_3x_4-3\zeta_3x_0x_4^2+\\
(-1-2\zeta_3)x_1^3+(3\zeta_3-3)x_1^2x_2+(3\zeta_3-3)x_4x_1^2+(\zeta_3-4)x_2^2x_1+(-12\zeta_3-6)x_4x_1x_2+\\
(-4\zeta_3-5)x_3^2x_1-3x_4^2x_1+(-4\zeta_3-5)x_2^3+(-2\zeta_3+2)x_2^2x_3+(-2\zeta_3-4)x_4x_2^2+\\
(2\zeta_3+4)x_3^2x_2+3\zeta_3x_2x_4^2+(-\zeta_3-2)x_4x_3^2+(-5\zeta_3-4)x_4^3,
\end{multline*}
\begin{multline*}
f_3=(\zeta_3-1)x_0^3+6x_0^2x_2+12x_0^2x_3+(-6\zeta_3-3)x_1^2x_0+(12\zeta_3+6)x_0x_1x_3+\\
(-6\zeta_3-12)x_4x_0x_1+(\zeta_3+5)x_0x_2^2+(-6\zeta_3-6)x_4x_0x_2+(-4\zeta_3-2)x_3^2x_0+\\
(-6\zeta_3-18)x_4x_0x_3+6x_4^2x_0+(10\zeta_3+5)x_1^3+(-6\zeta_3-3)x_4x_1^2+(\zeta_3+2)x_2^2x_1+\\
(12\zeta_3+6)x_4x_1x_2+(5\zeta_3+4)x_3^2x_1+6x_4^2x_1+(2\zeta_3+4)x_2^3+(4\zeta_3+2)x_2^2x_3+\\
(-2\zeta_3+2)x_4x_2^2+(2\zeta_3+4)x_3^2x_2+12x_4^2x_2+(-\zeta_3+7)x_4x_3^2+(\zeta_3-1)x_4^3,
\end{multline*}
\begin{multline*}
f_4=-6\zeta_3x_0x_1x_4-6x_0x_2x_4+12x_1x_2x_3+(3\zeta_3+6)x_1^2x_2+(-6-3\zeta_3)x_2x_0^2+\\
(-3\zeta_3-3)x_4^2x_1+(2\zeta_3-2)x_4x_2^2+(-2\zeta_3-4)x_3^2x_2+(-5\zeta_3-4)x_4x_3^2+(1-\zeta_3)x_2^2x_1+\\
(-1-2\zeta_3)x_3^2x_1+(2\zeta_3-2)x_2^2x_3-3x_0^2x_1-3x_4^2x_0+3x_4^2x_2+6\zeta_3x_1x_2x_4+(-6\zeta_3-6)x_0x_1x_3+\\
6\zeta_3x_0x_3x_4+(-4\zeta_3+1)x_0x_2^2+(4\zeta_3-4)x_3^2x_0+(-3\zeta_3+3)x_4x_1^2+(5\zeta_3-2)x_4^3+\\
(-\zeta_3-2)x_1^3+(\zeta_3-1)x_2^3+(2\zeta_3-2)x_0^3+(-6\zeta_3-3)x_1^2x_0,
\end{multline*}
\begin{multline*}
f_5=12x_1x_3x_4+(-4\zeta_3-2)x_3^2x_0+(-7\zeta_3-5)x_3^2x_1+3\zeta_3x_1^2x_2+(6\zeta_3+9)x_4^2x_0+\\
(3\zeta_3-6)x_4x_1^2+(-3+6\zeta_3)x_4^2x_2+(2\zeta_3+7)x_4x_3^2+(8\zeta_3+10)x_3^2x_2+(-6\zeta_3-3)x_0^2x_1+\\
(\zeta_3-4)x_0x_2^2+(-1-2\zeta_3)x_2^2x_1+(-2\zeta_3-4)x_2^2x_3+(3\zeta_3+6)x_2x_0^2+(3+3\zeta_3)x_4^2x_1+\\
(4\zeta_3+2)x_4x_2^2+6x_1^2x_0+(\zeta_3-1)x_1^3+(1-\zeta_3)x_2^3+(-1-2\zeta_3)x_4^3+(\zeta_3+5)x_0^3+\\
(6\zeta_3+12)x_3x_0x_1+(6\zeta_3+12)x_4x_0x_2+(-12\zeta_3-6)x_4x_0x_3+(-6\zeta_3-6)x_4x_1x_2,
\end{multline*}
\begin{multline*}
f_6=(-5\zeta_3-4)x_0^3+3x_0^2x_1-3\zeta_3x_0^2x_2+3\zeta_3x_0x_1^2+(6\zeta_3+6)x_4x_0x_1+\\
(\zeta_3-1)x_0x_2^2+(-6\zeta_3-6)x_4x_0x_2+(2\zeta_3-2)x_3^2x_0-3x_4^2x_0+(-1-2\zeta_3)x_1^3-\\
3\zeta_3x_1^2x_2+(\zeta_3+2)x_2^2x_1+6x_1x_2x_4+(2\zeta_3-2)x_3^2x_1+(-9\zeta_3-3)x_4^2x_1+\\
(1-\zeta_3)x_2^3+(4\zeta_3+8)x_2^2x_3+(-2\zeta_3-4)x_4x_2^2+(2\zeta_3-2)x_3^2x_2+\\
12x_2x_3x_4+3x_4^2x_2+(2\zeta_3+4)x_4x_3^2+(-2\zeta_3+2)x_4^3,
\end{multline*}
\begin{multline*}
f_7=(4\zeta_3+3)x_0^3+6\zeta_3x_0^2x_2-3\zeta_3x_0x_1^2+6\zeta_3x_0x_1x_3+(-6\zeta_3-6)x_4x_0x_1+\\
(-1-2\zeta_3)x_0x_2^2+(12\zeta_3+6)x_4x_0x_2+(2\zeta_3+4)x_3^2x_0+6x_0x_3x_4+6\zeta_3x_0x_4^2+\\
\zeta_3x_1^3-3\zeta_3x_1^2x_4+(\zeta_3+5)x_2^2x_1+6\zeta_3x_1x_2x_4+(1-\zeta_3)x_3^2x_1+\\
6\zeta_3x_1x_4^2+(2\zeta_3+2)x_2^3+(-2\zeta_3-4)x_2^2x_3+(4\zeta_3+2)x_4x_2^2+\\
(2\zeta_3-2)x_3^2x_2+4x_3^3+(2\zeta_3+1)x_4x_3^2+(4\zeta_3+3)x_4^3,
\end{multline*}
\begin{multline*}
f_8=-x_0^3+(-\zeta_3-2)x_0^2x_1+(1-\zeta_3)x_2x_0^2-2x_1^2x_0+4x_0x_1x_2+(-2\zeta_3-2)x_3x_0x_1+\\
(\zeta_3+3)x_0x_2^2+(-2\zeta_3-2)x_4x_0x_2-2x_3^2x_0+(2\zeta_3-4)x_4x_0x_3-\zeta_3x_0x_4^2+\\
(2\zeta_3+1)x_1^3+(1+\zeta_3)x_1^2x_2+(\zeta_3+3)x_4x_1^2-\zeta_3x_1x_2^2+(4\zeta_3+2)x_4x_1x_2+\\
(2\zeta_3+1)x_3^2x_1+(-2\zeta_3+1)x_4^2x_1+(2\zeta_3+1)x_2^3+(2\zeta_3+2)x_3x_2^2-\\
2\zeta_3x_2x_3^2+(-\zeta_3+2)x_4^2x_2-\zeta_3x_3^2x_4+\zeta_3x_4^3,
\end{multline*}
\begin{multline*}
f_9=(\zeta_3+2)x_0^3+(-1-\zeta_3)x_1^3+(-2\zeta_3-2)x_2^3+(-1-2\zeta_3)x_3^2x_1+\\
2\zeta_3x_2^2x_4-2x_1^2x_2-x_2^2x_1-3\zeta_3x_4^3+(\zeta_3-3)x_4x_1^2+(\zeta_3-2)x_0x_2^2+\\
(2\zeta_3+4)x_2x_0^2+(4\zeta_3+2)x_4^2x_1+(2\zeta_3-2)x_4^2x_2+(\zeta_3+2)x_4x_3^2+\\
(-2\zeta_3-2)x_2^2x_3+(4\zeta_3+2)x_3^2x_2+(-2\zeta_3+2)x_0^2x_1+(\zeta_3+3)x_1^2x_0+\\
(-2\zeta_3+2)x_3^2x_0-6\zeta_3x_0x_3x_4+(-6\zeta_3-6)x_4x_1x_2+(6\zeta_3+6)x_0x_1x_3+\\
(4\zeta_3+2)x_4x_0x_1+(2\zeta_3+4)x_4x_0x_2+4x_1^2x_3,
\end{multline*}
\begin{multline*}
f_{10}=-\zeta_3x_0^3+(2\zeta_3+4)x_1x_0^2+(\zeta_3-1)x_1^2x_0+(2\zeta_3-2)x_0x_1x_3+\\
(4\zeta_3+2)x_4x_0x_1+(-\zeta_3+2)x_0x_2^2+(-2\zeta_3-4)x_4x_0x_2+(2\zeta_3+2)x_3^2x_0+\\
(2\zeta_3+4)x_4x_0x_3-2x_4^2x_0+(1+\zeta_3)x_1^3+(-2\zeta_3+2)x_1^2x_2+(1-\zeta_3)x_4x_1^2+\\
x_2^2x_1+(6\zeta_3+6)x_4x_1x_2+x_3^2x_1-2\zeta_3x_1x_4^2+(2\zeta_3+2)x_3x_2^2-\\
2\zeta_3x_2^2x_4-2x_3^2x_2-2\zeta_3x_2x_4^2+\zeta_3x_3^2x_4+4x_3x_4^2+(\zeta_3+2)x_4^3.
\end{multline*}
Using Magma one can check that the subscheme defined by these polynomials is indeed $Z$. 
\end{proof}

Now, we are ready to prove the following result. 

\begin{lemma}
\label{lemma:A5-reducible-curves}
Let $Z$ be a $G$-invariant $G$-irreducible reducible curve of degree $<12$. Then
\begin{itemize}
\item either $Z=\mathcal{L}_6$ or $Z=\mathcal{L}_6^\prime$,
\item or $Z=\mathcal{L}^t_{10}$ for some $t\in\mathbb{C}$ such that $(b+3)t^3+3b\zeta_3t^2+1=0$.
\end{itemize}
\end{lemma}

\begin{proof}
Let $C$ be an irreducible component of the curve $Z$, and let $G_C$ be its stabilizer in $G\simeq\mathfrak{A}_5$.
Then one of the following cases holds:
\begin{itemize}
\item $Z$ is a union of $5$ lines, and $G_C\simeq \mathfrak{A}_4$,
\item $Z$ is a union of $5$ conics, and $G_C\simeq \mathfrak{A}_4$,
\item $Z$ is a union of $6$ lines, and $G_C\simeq \mathfrak{D}_5$,
\item $Z$ is a union of $10$ lines, and $G_C\simeq \mathfrak{S}_3$.
\end{itemize}
Let us deal with these cases one by one.

Suppose that $Z$ is a union of $5$ lines. Then the restriction of the representation $\mathbb{V}_5$ to $G_C\simeq \mathfrak{A}_4$ splits as a sum of an irreducible three-dimensional representation, and two non-isomorphic one-dimensional representations. Hence, there is a unique choice for $C$, so $C$ must be an irreducible component of the curve $\mathcal{L}_5$, which is not contained in $X$, which is a contradiction.

Suppose that $Z$ is a union of $5$ conics. Let $\Pi$ be the plane in $\mathbb{P}^4$ that contains $C$. Then
$$
X\cap \Pi=C+L
$$
for some line $L$, which must be an irreducible component of a $G$-irreducible curve consisting of five lines, because $G$ does not leave any line in $\mathbb{P}^4$ invariant. However, we already proved that $X$ does not contain $G$-irreducible curve consisting of five lines, so $Z$ is not a union of $5$ conics.

Suppose that $Z$ is a union of $6$ lines. Then the restriction of the representation $\mathbb{V}_5$ to $G_C\simeq \mathfrak{D}_5$ splits as a sum of two non-isomorphic irreducible two-dimensional representations, and trivial representation. This implies that $C$ is an irreducible component of one of the curves $\mathcal{L}_6$ or $\mathcal{L}_6^\prime$, and, therefore either $Z=\mathcal{L}_6$ or $Z=\mathcal{L}_6^\prime$ as claimed.

Suppose that $Z$ is a union of $10$ lines. The restriction of the representation $\mathbb{V}_5$ to~\mbox{$G_C\simeq \mathfrak{S}_3$} splits as a sum of two isomorphic irreducible two-dimensional representations, and trivial representation, so the line $C$ corresponds to a two-dimensional irreducible representation. Then $C$ is an irreducible component of the curve $\mathcal{L}^t_{10}$ for some $t\in\mathbb{C}$. Indeed, since $G$ contains a unique subgroup isomorphic to $\mathfrak{S}_3$ up to conjugation, we may assume that $G_C$ is generated by the following transformations:
\begin{align*}
T_1\colon [x_0:x_1:x_2:x_3:x_4]&\mapsto[x_2:x_1:x_0:x_4:x_3],\\
T_2\colon [x_0:x_1:x_2:x_3:x_4]&\mapsto[\zeta_3^2x_1:\zeta_3x_2:x_0:\zeta_3x_3:\zeta_3^2x_4].
\end{align*}
But $C$ has a unique $G_C$-orbit of length $2$, and the points of this orbit must be fixed by $\langle T_2\rangle\simeq C_3$. Computing the eigenvalues and eigenvectors of the matrix that corresponds to $T_2$, we see that this $G_C$-orbit of length $2$ consists of the points $[\zeta_3t:t:t:1:0]$ and $[t:t:\zeta_3t:0:1]$ for some $t\in\mathbb{C}$, which implies that $C$ is given by \eqref{equation:A5-10-lines}. Therefore, $C$ is  component of the curve $\mathcal{L}^t_{10}$ for some $t\in\mathbb{C}$. Now, using $C\subset X$, we get $(b+3)t^3+3b\zeta_3t^2+1=0$.
\end{proof}

Now, let us introduce $G$-irreducible surfaces in $X$ that have small degree. To do this, we set
\begin{multline*}
\quad \quad \quad \quad \quad \quad 
f_2=x_1x_2+x_1x_3+x_1x_4+\zeta_3^2x_2x_3+x_2x_4+\\
+\zeta_3x_3x_4+x_0x_1+\zeta_3x_0x_2+x_0x_3+\zeta_3^2x_0x_4,
\quad \quad \quad \quad \quad \quad
\end{multline*}
\begin{multline*}
\quad \quad \quad \quad \quad \quad f_3=x_0x_1x_3+\zeta_3^2x_0x_1x_2+\zeta_3x_0x_1x_4+x_2x_3x_4+x_1x_2x_4+\\
+\zeta_3^2x_1x_3x_4+x_0x_3x_4+\zeta_3x_1x_2x_3+x_3x_2x_0+x_0x_2x_4.\quad \quad \quad \quad \quad \quad
\end{multline*}
Let $Q=\{f_2=0\}\cap X$ and $S=\{f_3=0\}\cap X$. Then $Q$ and $S$ are $G$-invariant surfaces.

\begin{lemma}
\label{lemma:A5-Q-S-orbits}
The following assertions hold:
\begin{itemize}
\item[($\mathrm{i})$] $Q$ contains $\Sigma_{12}$ and $\Sigma_{12}^\prime$;
\item[($\mathrm{ii})$] $Q$ does not contain $G$-orbits of length $15$;
\item[($\mathrm{iii})$] $Q$ contains all $G$-orbits in $X$ of length $20$;
\item[($\mathrm{iv})$] $Q$ contains neither $\mathcal{L}_6$ nor $\mathcal{L}_6^\prime$;
\item[($\mathrm{v})$] $Q$ contains $\mathcal{L}_{10}^t$ if and only if $(b+3)t^3+3b\zeta_3t^2+1=0$, $11b^2+174b+207=0$ 
and 
$$
t=\frac{(11b + 150)\zeta_3}{63};
$$
\item[($\mathrm{vi})$] $S$ contains $\Sigma_{12}$ and $\Sigma_{12}^\prime$;
\item[($\mathrm{vii})$] $S$ contains $\Sigma_{15}$, but $S\cap(\mathcal{L}_5\cap X)=\varnothing$;
\item[($\mathrm{viii})$] $S$ does not contain $G$-orbits of length $20$;
\item[($\mathrm{ix})$] $S$  contains $\mathcal{L}_6$ and $\mathcal{L}_6^\prime$; 
\item[($\mathrm{x})$] $S$ does not contain $\mathcal{L}_{10}^t$ for any $t\in\mathbb{C}$ such that $(b+3)t^3+3b\zeta_3t^2+1=0$.
\end{itemize}
\end{lemma}

\begin{proof}
Explicit computations.
\end{proof}

Set $\mathcal{C}_{18}=Q\cap S$. Then $\mathcal{C}_{18}$ is a $G$-invariant curve of degree $18$.

\begin{lemma}
\label{lemma:A5-Q-S-smooth}
Surfaces $Q$ and $S$ are irreducible,
$\mathcal{C}_{18}$ is irreducible, $Q$ is smooth, 
\begin{equation}
\label{equation:S-C-singularities}
\mathrm{Sing}(\mathcal{C}_{18})=\mathrm{Sing}(S)=\Sigma_{12}\cup\Sigma_{12}^\prime,
\end{equation}
and $S$ has isolated ordinary double points (nodes) at $\Sigma_{12}\cup\Sigma_{12}^\prime$.
\end{lemma}

\begin{proof}
Note that $S$ and $\mathcal{C}_{18}$ do not depend on $b$, so we can use Magma to verify \eqref{equation:S-C-singularities}, and to check that $S$ has an isolated ordinary double singularity (a node) at every point of $\Sigma_{12}\cup\Sigma_{12}^\prime$.

Explicit computations show that $Q$ is smooth at every point of $\Sigma_{12}\cup\Sigma_{12}^\prime$, so $Q$ has isolated singularities, because $\mathcal{C}_{18}=Q\cap S$ is smooth away from $\Sigma_{12}\cup\Sigma_{12}^\prime$. Hence, $Q$ is normal and $K_Q\sim 0$, so it follows from \cite{Umezu} that $Q$ can have at most two non-Du Val singular points, which also follows from \cite[Theorem 6.9]{Shokurov92} and \cite{Fujino}.
Thus, by Lemma~\ref{lemma:A5-orbits}, $Q$ has  Du Val singularities.
Let $\widetilde{Q}$ be the minimal resolution of singularities of the surface $Q$. Then $\widetilde{Q}$ is a smooth K3 surface, so that $\mathrm{rk}\,\mathrm{Pic}(\widetilde{Q})\leqslant 20$, which gives $|\mathrm{Sing}(Q)|<20$. Thus, the surface $Q$ is smooth by Lemma~\ref{lemma:A5-orbits}, because $Q$ does not contain $\Sigma_{15}$. 
\end{proof}

Now, we study irreducible $G$-invariant curves in surfaces $Q$ and $S$.

\begin{lemma}
\label{lemma:A5-irreducible-curves-in-Q}
Let $C$ be an irreducible $G$-invariant curve contained in $Q$, and let $d$ be its degree. Suppose that $d<12$. Then $C$ is smooth, and one of the following two cases holds:
\begin{itemize}
\item either $d=8$ and $C\simeq\mathbb{P}^1$,
\item or $d=10$, and $C$ is a curve of genus $6$.
\end{itemize}
\end{lemma}

\begin{proof}
Let $\overline{C}$ be the normalization of the curve $C$, and let $g$ be the genus of the smooth curve $\overline{C}$. Since $\mathbb{V}_5$ is irreducible representation, $G$ acts faithfully on the curve $C$, and its action lifts to $\overline{C}$. Moreover, if $g\leqslant 15$, it follows from Lemma~\ref{lemma:A5-curves} that
$$
g\in\{0,4,5,6,9,10,11,13,15\},
$$
and the length of a $G$-orbit in $\overline{C}$ is $12$, $20$, $30$ or $60$.

Since $\mathcal{C}_{18}=Q\cap S$ is an irreducible curve of degree $18$, we see that $C\ne \mathcal{C}_{18}$. Thus, it follows from Lemma~\ref{lemma:A5-orbits} and the description of the $G$-orbits in $Q$ and $S$ that
$$
33\geqslant 3d=\mathcal{C}_{18}\cdot C=12\times 2\times a+30\times b,
$$
for some non-negative integers $a$ and $b$, because $\mathcal{C}_{18}$ is singular at $\Sigma_{12}\cup\Sigma_{12}^\prime$,
and $\mathcal{C}_{18}$ does not contain $G$-orbits of length $20$. Hence, we see that either $(d,a,b)=(8,1,0)$ or $(d,a,b)=(10,0,1)$.
Moreover, since $\Sigma_{12}$ or $\Sigma_{12}^\prime$ are the only $G$-orbits of length $12$ in $X$ by Lemma~\ref{lemma:A5-orbits}, we conclude that one of the following two cases holds:
\begin{enumerate}
\item $d=8$, the curve $C$ contains one orbit among $\Sigma_{12}$ or $\Sigma_{12}^\prime$, and $C$ is smooth at this orbit;
\item $d=10$, and the curve $C$ contains neither $\Sigma_{12}$ nor $\Sigma_{12}^\prime$.
\end{enumerate}

Recall from Lemma~\ref{lemma:A5-Q-S-smooth} that $Q$ is a smooth K3 surface.
Let $H_Q$ be its general hyperplane section. Then it follows from the Hodge index theorem that
$$
6C^2-d^2=\mathrm{det}\begin{pmatrix}
C^2 &d \\
d & 6
\end{pmatrix}=\mathrm{det}\begin{pmatrix}
C^2 &H_{Q}\cdot C \\
H_{Q}\cdot C & H_{Q}^2
\end{pmatrix}\leqslant 0.
$$
Therefore, it follows from the adjunction formula that
$$
2g-2\leqslant 2\mathrm{p}_a(C)-2=C^2\leqslant\Big\lfloor\frac{d^2}{6}\Big\rfloor=
\left\{\aligned
&10\ \text{if $d=8$},\\
&16\ \text{if $d=10$},
\endaligned
\right.
$$
where $\mathrm{p}_a(C)$ is the arithmetic genus of the curve $C$.
Therefore, we have
$$
g\leqslant \mathrm{p}_a(C)\leqslant \left\{\aligned
&6\ \text{if $d=8$},\\
&9\ \text{if $d=10$}.
\endaligned
\right.
$$
Hence, if $d=8$, then $g\in\{0,4,5,6\}$. Similarly, if $d=10$, then $g\in\{0,4,5,6,9\}$

We claim that $C$ is smooth. Indeed, if $C$ is singular, $\mathrm{Sing}(C)$ is a union of orbits $\Sigma_1,\ldots\Sigma_r$, so
$$
0\leqslant g=\mathrm{p}_a(C)-\sum_{i=1}^rn_i|\Sigma_i|\leqslant 9-\sum_{i=1}^rn_i|\Sigma_i|
$$
for some positive integers $n_1,\ldots,n_r$, which contradicts Lemma~\ref{lemma:A5-orbits}. Thus, $C$ is a smooth curve.

If $d=10$, then $C$ does not have $G$-orbits of length $12$, so it follows from Lemma~\ref{lemma:A5-curves} that $g=6$. Similarly, if $d=8$, the curve $C$ has one $G$-orbit of length $12$, so $g\ne 6$ and $g\ne 4$ by Lemma~\ref{lemma:A5-curves}. Thus, one of the following holds:
\begin{itemize}
\item $d=8$ and $g=0$,
\item $d=8$ and $g=5$,
\item $d=10$ and $g=6$.
\end{itemize}

Suppose that $d=8$ and $g=5$. To complete the proof, we must show that this case is impossible. In this case, we have $C^2=8$ on the surface $Q$, which gives $(2H_{Q}-C)^2=0$. Thus, it follows from the~Riemann--Roch formula and Serre duality that
$$
h^0\big(Q,\mathcal{O}_{Q}(2H_Q-C)\big)\geqslant 2.
$$
In particular, the~linear system $|2H_Q-C|$ is not empty. On the other hand, this linear system does not have base curves, because curves in $|2H_Q-C|$ have degree $4$. Thus, since $(2H_Q-C)^2=0$, the linear system $|2H_Q-C|$ is basepoint free. Moreover, general curve in $|2H_Q-C|$ is irreducible, since otherwise $Q$ would be covered by rational curves, which is impossible, since $Q$ is not uniruled. Hence, $h^0(Q,\mathcal{O}_{Q}(2H_Q-C))=2$ \cite{SD}.
We have the following embedding of $G$-representations:
$$
H^0(Q,\mathcal{O}_{Q}(2H_Q-C))\hookrightarrow H^0\big(Q,\mathcal{O}_{Q}(2H_Q)\big)\simeq H^0\big(X,\mathcal{O}_X(2)\big)/I(Q).
$$
On the other hand, GAP computations show that $H^0(X,\mathcal{O}_X(2))$ does not have two-dimensional subrepresentations. The obtained contradiction completes the proof.
\end{proof}

\begin{lemma}
\label{lemma:A5-irreducible-curves-in-S}
The surface $S$ does not contain $G$-invariant irreducible curves of degree $<12$.
\end{lemma}

\begin{proof}
Let $C$ be an irreducible $G$-invariant curve contained in $S$, let $\overline{C}\to C$ be its normalization, and let $g$ be the genus of $\overline{C}$. As in the proof of Lemma~\ref{lemma:A5-irreducible-curves-in-Q}, we see that $G$ acts faithfully on $C$, and this action lifts to $\overline{C}$. Moreover, if $g\leqslant 15$, then  Lemma~\ref{lemma:A5-curves} gives
$$
g\in\{0,4,5,6,9,10,11,13,15\}.
$$
Similarly, it follows from  Lemma~\ref{lemma:A5-curves} that the lengths of $G$-orbits in $\overline{C}$ are $12$, $20$, $30$ or $60$.

Let $d$ be the degree of the curve $C$. Suppose that $d<12$. Let us seek for a contradiction.

Observe that $C\ne \mathcal{C}_{18}=Q\cap S$, because $\mathcal{C}_{18}$ is an irreducible curve of degree $18$. Thus, $C\not\subset Q$. On the other hand, it follows from Lemma~\ref{lemma:A5-orbits} that $G$-orbits in $Q$ have length $12$, $20$, $30$ or $60$, which also follows from \cite[Lemma 6.7.1]{CheltsovShramov2016}. Therefore, we have
$$
2d=Q\cdot C=12\times a+20\times b+30\times c
$$
for some non-negative integers $a$, $b$, $c$, and, therefore $11\geqslant d=6a+10b+15c$, which gives $c=0$, and either $(d,a,b)=(6,1,0)$ or $(d,a,b)=(10,0,1)$. Thus, we see that either $d=6$ or $d=10$. Moreover, our computations and Lemma~\ref{lemma:A5-orbits} also imply that one of the following cases holds:
\begin{enumerate}
\item $d=6$, the curve $C$ contains one orbit of length $12$, and $C$ is smooth at this orbit;
\item $d=10$, and $C$ does not contain $G$-orbits of length $12$.
\end{enumerate}

Let $H_S$ be a general hyperplane section of the surface $S$. Then it follows from that Hodge index theorem that
$$
9C^2-d^2=\mathrm{det}\begin{pmatrix}
C^2 &d \\
d & 9
\end{pmatrix}=\mathrm{det}\begin{pmatrix}
C^2 &H_{S}\cdot C \\
H_{S}\cdot C & H_S^2
\end{pmatrix}\leqslant 0,
$$
which gives
$$
C^2\leqslant \Big\lfloor\frac{d^2}{9}\Big\rfloor=
\left\{\aligned
&4\ \text{if $d=6$},\\
&11\ \text{if $d=10$}
\endaligned
\right.
$$

Suppose that $d=10$. Then $C$ is contained in the smooth locus of the surface $S$, because the surface $S$ is smooth away from the locus $\Sigma_{12}\cup\Sigma_{12}^\prime$ by Lemma~\ref{lemma:A5-Q-S-smooth}, and $C$ does not contain $G$-orbits of length $12$. Thus, if $d=10$, then it follows from the adjunction formula that
$$
2g-2\leqslant 2\mathrm{p}_a(C)-2=(K_S+C)\cdot C=d+C^2=10+C^2\leqslant 21,
$$
where $\mathrm{p}_a(C)$ is the arithmetic genus of the curve $C$. Hence, we have $g\leqslant  \mathrm{p}_a(C)\leqslant 11$, so arguing as in the proof of Lemma~\ref{lemma:A5-irreducible-curves-in-Q}, we see that $C$ is smooth.
Then $g\in\{0,4,5,6,9,10,11\}$, and it follows from  Lemma~\ref{lemma:A5-curves} that $C$ has a $G$-orbit of length $12$ or $20$. But $C$ has no $G$-orbits of length~$12$, and it does not contain $G$-orbits of length $20$ by Lemma~\ref{lemma:A5-Q-S-orbits}.

Thus, $d=6$, and $C$ has exactly one $G$-orbit of length $12$, which must be $\Sigma_{12}$ or $\Sigma_{12}^\prime$ by Lemma~\ref{lemma:A5-orbits}. Without loss of generality, we may assume that $\Sigma_{12}\subset C$, and $\Sigma_{12}^\prime\not\subset C$. Recall that $C$ is smooth at the points of $\Sigma_{12}$, and it follows from  Lemma~\ref{lemma:A5-Q-S-smooth} that $S$ has nodes at points of $\Sigma_{12}$.

Let $\pi\colon\widetilde{S}\rightarrow S$ be the blowup of  $\Sigma_{12}$, let $E$ be the $\pi$-exceptional divisor, and $\widetilde{C}$ be the strict transform on $\widetilde{S}$ of the curve $C$. Then, since $C$ is smooth at the points of $\Sigma_{12}$, we have
$$
\widetilde{C}\sim_{\mathbb{Q}} \pi^*(C)-\frac{1}{2}E.
$$
Then, it follows from the adjunction formula that
$$
2g-2\leqslant 2\mathrm{p}_a\big(\widetilde{C}\big)-2=\big(K_{\widetilde{S}}+\widetilde{C}\big)\cdot\widetilde{C}=\Big(\pi^*\big(\mathcal{O}_{X}(1)\vert_{S}\big)+\pi^*(C)-\frac{1}{2}E\Big)\cdot\widetilde{C}=C^2\leqslant 4,
$$
where $\mathrm{p}_a(\widetilde{C})$ is the arithmetic genus of the curve $\widetilde{C}$.
Thus, $0\leqslant g\leqslant\mathrm{p}_a(\widetilde{C})\leqslant 3$, which gives $g=0$.

Now, arguing as in the proof of Lemma~\ref{lemma:A5-irreducible-curves-in-Q}, we see that the curve $\widetilde{C}$ is smooth, so $C$ is also smooth. Thus, $C$ contains a $G$-orbit of length $20$ by \cite[Lemma 5.1.5]{CheltsovShramov2016}, but $S$ does not contain $G$-orbits of length $20$, which is a contradiction.
 \end{proof}

\begin{corollary}
\label{corollary:A5-irreducible-curves}
Let $C$ be an irreducible $G$-invariant curve in $X$ of degree $d<12$ such that $C\not\subset Q$.
Then $d=10$, and $C$ does not contain $G$-orbits of length $12$ or $15$.
\end{corollary}

\begin{proof}
Recall that all $G$-orbits in the smooth K3 surface $Q$ are of length $12$, $20$, $30$ and $60$, which also follows from \cite[Lemma 6.7.1]{CheltsovShramov2016}. Thus, since $C\not\subset Q$, we have
$$
2d=Q\cdot C=12\times a+20\times b+30\times c
$$
for some non-negative integers $a$, $b$, $c$. So, either $(d,a,b,c)=(6,1,0,0)$ or $(d,a,b,c)=(10,0,1,0)$. 

Recall that $C\not\subset S$ by Lemma~\ref{lemma:A5-irreducible-curves-in-S}, 
and $S$ is singular at the points of $\Sigma_{12}\cup\Sigma_{12}^\prime$ by Lemma~\ref{lemma:A5-Q-S-smooth}.
Moreover, it follows from Lemmas~\ref{lemma:A5-orbits} and \ref{lemma:A5-Q-S-orbits} that $\Sigma_{15}$ is the unique $G$-orbit of length $15$ in $S$, and it follows from Lemma~\ref{lemma:A5-curves} that  $\Sigma_{15}\cap(C\setminus\mathrm{Sing}(C))=\varnothing$. Thus, Lemma~\ref{lemma:A5-Q-S-orbits} gives 
$$
3d=S\cdot C=12\times 2\times a^\prime+15\times 2\times b^\prime,
$$
for some non-negative integers $a^\prime$ and $b^\prime$. Hence, either $(d,a^\prime,b^\prime)=(8,1,0)$ or $(d,a^\prime,b^\prime)=(10,0,1)$. Therefore, $d=10$, and $C$ does not contain $G$-orbits of length $12$ by Lemmas~\ref{lemma:A5-orbits} and \ref{lemma:A5-Q-S-orbits}.

To complete the proof, we must show that the curve $C$ does not contain $G$-orbits of length $15$.
Suppose that this is not the case. Then $C$ contains a $G$-orbit $\Sigma$ such that $|\Sigma|=15$. 
Then it follows from Lemma~\ref{lemma:A5-orbits} that either $\Sigma=\Sigma_{15}$ or $\Sigma\subset\mathcal{L}_5\cap X$. Moreover, by Lemma~\ref{lemma:A5-curves}, we get $\Sigma\subset\mathrm{Sing}(C)$. Fix a point $O\in \Sigma$. Let $\mathcal{M}$ be the linear system in $|\mathcal{O}_X(3)|$ consisting of all surfaces that contain the $G$-orbit $\Sigma$ and are singular at $O$. Explicit computations show that the base locus of $\mathcal{M}$ does not contain curves of degree $\ne 1$, so, in particular, the curve $C$ is not contained in this base locus. Let $M$ be a general surface in $\mathcal{M}$. Then $C\not\subset M$, so that
\begin{multline*}
30=M\cdot C=\sum_{P\in C\cap M}\big(M\cdot C\big)_P\geqslant \sum_{P\in \Sigma}\big(M\cdot C\big)_P\geqslant \sum_{P\in \Sigma}\mathrm{mult}_P(M)\mathrm{mult}_P(C)\geqslant\\ 
\geqslant 2\sum_{P\in\Sigma}\mathrm{mult}_P(M)\geqslant 
4+2\sum_{P\in\Sigma\setminus O}\mathrm{mult}_P(M)\geqslant 4+\sum_{P\in\Sigma\setminus O}2=32,
\end{multline*}
which is absurd.
\end{proof}

\begin{lemma}
\label{lemma:A5-irreducible-curves-degree-10}
Let $C$ be an irreducible $G$-invariant curve in $X$ of degree $d<12$ such that $C\not\subset Q$.
Then $d=10$, $C$ is a smooth curve of genus $6$, $\Sigma_{12}\not\subset C$ and $\Sigma_{12}^\prime\not\subset C$.
\end{lemma}

\begin{proof}
By Corollary~\ref{corollary:A5-irreducible-curves}, $d=10$, and $C$ has no $G$-orbits of length $12$. So, $\Sigma_{12}\not\subset C$ and $\Sigma_{12}^\prime\not\subset C$. Since the representation $\mathbb{V}_5$ is irreducible, $C$ spans $\mathbb{P}^4$, so $G$ acts faithfully on $C$. Let $\overline{C}\to C$ be its normalization, and let $g$ be the genus of $\overline{C}$. The Castelnuovo bound \cite{ACGH} gives $\mathrm{p}_a(C)\leqslant 9$, where $\mathrm{p}_a(C)$ is the arithmetic genus of $C$.

Suppose that $C$ is singular. Then $\mathrm{Sing}(C)$ contains a $G$-orbit of length at least $12$ by Lemma~\ref{lemma:A5-orbits}, so it follows from \cite[Chapter IV, Exercise 1.8]{Hartshorne} that
$$
12\leqslant\sum_{P\in\mathrm{Sing}(C)}\delta_P(C)=\mathrm{p}_a(C)-g\leqslant 9,
$$
where $\delta_P(C)\geqslant 1$ is the delta invariant of the singular point $P$. Therefore,  $C$ must be  smooth. Since $g\leqslant 9$ and $C$ contains no $G$-orbits of length $12$, Lemma~\ref{lemma:A5-curves} gives $g=6$.
\end{proof}

To proceed, we need the following lemma.

\begin{lemma}
\label{lemma:A5-3-regular}
Let $C$ be a smooth irreducible $G$-invariant curve of degree $10$ and genus $6$ in $\mathbb{P}^4$, and let  $\mathcal{I}_C$ be the ideal sheaf of the curve $C$. Suppose that $C\subset X$ and $C\not\subset Q$. Then 
\begin{equation}
\label{equation:3-regular}
H^i\big(\mathcal{I}_C\otimes\mathcal{O}_{\mathbb{P}^4}(3-i)\big)=0
\end{equation}
for every $i>0$, i.e., the curve $C$ is $3$-regular \cite{Mumford}.
\end{lemma}

\begin{proof}
We claim that $C$ is not contained in any surface in $|-K_X|$. Indeed, let $\mathcal{M}$ be the linear subsystem in $|-K_X|$ consisting of surfaces that pass through $C$. Suppose that $\mathcal{M}$ is not empty. Then $\mathcal{M}$ does not have fixed components, since $C\not\subset Q$. Moreover, the base locus of the linear system $\mathcal{M}$ does not contain curves except $C$. Indeed, if $C^\prime$ is another $G$-irreducible curve contained in the base locus of the linear system $\mathcal{M}$, then 
$$
M\cdot M^\prime=mC+m^\prime C^\prime+\Delta,
$$
for general surfaces $M$ and $M^\prime$ in $\mathcal{M}$, an effective one-cycle $\Delta$, and $m,m^\prime\in\mathbb{Z}_{>0}$, which gives 
$$
\mathrm{deg}(C^\prime)\leqslant\mathrm{deg}\big(M\cdot M^\prime\big)-\mathrm{deg}(C)=12-\mathrm{deg}(C)=2,
$$
which is impossible by Lemmas~\ref{lemma:A5-reducible-curves}, \ref{lemma:A5-irreducible-curves-in-Q}, \ref{lemma:A5-irreducible-curves-in-S} and Corollary~\ref{corollary:A5-irreducible-curves}. Thus, $C$ is the only curve in the base locus of the linear system $\mathcal{M}$. Furthermore, the same arguments imply that two general surfaces in the linear system $\mathcal{M}$ are smooth at general point of the curve $C$, and they are not tangent to each other along $C$. Now, we let $\pi\colon\widetilde{X}\to X$ be the blow up of the curve $C$, let $E$ be the exceptional surface, and let $\widetilde{\mathcal{M}}$ be the strict transform of the linear system $\mathcal{M}$ on the threefold $\widetilde{X}$. Then
$$
\widetilde{\mathcal{M}}\sim\pi^*(-K_X)-E,
$$ 
and $\widetilde{\mathcal{M}}$ does not have base curves except possibly fibers of the~projection $E\to C$, which implies that the divisor $\pi^*(-K_X)-E$ is nef. But $\pi^*(-K_X)-E$  cannot be nef, since $(\pi^*(-K_X)-E)^3=-6$.
Hence, $C$ is not contained in any surface in $|-K_X|$, which implies that $C$ is not contained in any quadric hypersurface in $\mathbb{P}^4$. Let us use this to prove \eqref{equation:3-regular} for $i=1$.

Recall that $C$ is not contained in any hyperplane in $\mathbb{P}^4$, because the representation $\mathbb{V}_5$ is irreducible.
Let $H$ be a hyperplane section of the curve $C$. Consider the following exact sequence: 
\begin{multline*}
\quad \quad \quad 0=H^0\big(\mathcal{I}_C\otimes\mathcal{O}_{\mathbb{P}^4}(2)\big)\longrightarrow H^0\big(\mathcal{O}_{\mathbb{P}^4}(2)\big)\longrightarrow H^0\big(C,\mathcal{O}_{C}(2H)\big)\longrightarrow \\
\longrightarrow H^1\big(\mathcal{I}_C\otimes\mathcal{O}_{\mathbb{P}^4}(2)\big)\longrightarrow H^1\big(\mathcal{O}_{\mathbb{P}^4}(2)\big)=0.\quad \quad \quad \quad 
\end{multline*}
Now, using the~Riemann--Roch formula and Serre duality, we get 
$$
h^1\big(\mathcal{I}_C\otimes\mathcal{O}_{\mathbb{P}^4}(2)\big)=h^0\big(C,\mathcal{O}_{C}(2H)\vert_{C}\big)-15=h^0\big(C,\mathcal{O}_{C}(K_C-2H)\big)=0,
$$
because $\mathrm{deg}(K_C-2H)=-10$. This gives \eqref{equation:3-regular} for $i=1$.

Similarly, to prove \eqref{equation:3-regular} for $i=2$, we consider the following exact sequence: 
$$
0=H^1\big(\mathcal{O}_{\mathbb{P}^4}(1)\big)\longrightarrow H^1\big(C,\mathcal{O}_{C}(H)\big)\longrightarrow H^2\big(\mathcal{I}_C\otimes\mathcal{O}_{\mathbb{P}^4}(1)\big)\longrightarrow H^2\big(\mathcal{O}_{\mathbb{P}^4}(1)\big)=0.
$$
Thus, it follows from the~Serre duality that
$$
h^2\big(\mathcal{I}_C\otimes\mathcal{O}_{\mathbb{P}^4}(1)\big)=h^1\big(C,\mathcal{O}_{C}(H)\big)=h^0\big(C,\mathcal{O}_{C}(K_C-H)\big)=\left\{\aligned
&1\ \text{if $K_C\sim H$},\\
&0\ \text{if $K_C\not\sim H$}.
\endaligned
\right.
$$
On the other hand, if $K_C\sim H$, then the embedding $C\hookrightarrow\mathbb{P}^4$ is factored through the $G$-equivariant canonical embedding $\phi\colon C\hookrightarrow\mathbb{P}^5$ followed by the linear projection $\mathbb{P}^5\dasharrow \mathbb{P}^4$ from a $G$-fixed point. However, if $\mathbb{P}^5$ contains a $G$-fixed point, then $\mathbb{P}^5$ also contains a $G$-invariant hyperplane, which must intersect $C$ by $\leqslant 10$ points, which contradicts Lemma~\ref{lemma:A5-curves}. So, $K_C\not\sim H$, which gives \eqref{equation:3-regular} for $i=2$. 

For $i=3$, \eqref{equation:3-regular} follows from the exact sequence
$$
0=H^2\big(C,\mathcal{O}_{C}\big)\longrightarrow H^3\big(\mathcal{I}_C\otimes\mathcal{O}_{\mathbb{P}^4}\big)\longrightarrow H^3\big(\mathcal{O}_{\mathbb{P}^4}\big)=0.
$$
For $i=4$, \eqref{equation:3-regular} follows from the exact sequence
$$
0=H^3\big(C,\mathcal{O}_{C}(-H)\big)\longrightarrow H^4\big(\mathcal{I}_C\otimes\mathcal{O}_{\mathbb{P}^4}(-1)\big)\longrightarrow H^4\big(\mathcal{O}_{\mathbb{P}^4}(-1)\big)=0.
$$
This shows that the curve $C$ is $3$-regular.
\end{proof}

\begin{corollary}
\label{corollary:A5-3-regular}
In the assumptions of Lemma~\ref{lemma:A5-3-regular}, $\mathcal{I}_C\otimes\mathcal{O}_{\mathbb{P}^4}(3)$ is generated by global sections.
\end{corollary}

\begin{proof}
The required assertion follows from Lemma~\ref{lemma:A5-3-regular} and Proposition in  \cite[Lecture 14]{Mumford}.
\end{proof}

Let $H$ be a general hyperplane section of the cubic threefold $X$.

\begin{proposition}
\label{proposition:A5-NFI-curves}
Let $Z$ be a $G$-irreducible curve in $X$, and let $\mathcal{M}$ be a non-empty $G$-invariant mobile linear subsystem in $|nH|$, where $n\in\mathbb{Z}_{>0}$. Then
$$
\mathrm{mult}_Z\big(\mathcal{M}\big)\leqslant\frac{n}{2}.
$$
\end{proposition}

\begin{proof}
Suppose that $\mathrm{mult}_Z(\mathcal{M})>\frac{n}{2}$. Let $M$ and $M^\prime$ be general surfaces in $\mathcal{M}$. Then
$$
M\cdot M^\prime=mZ+\Delta,
$$
where $\Delta$ is an effective one-cycle whose support does not contain $Z$, and $m$ is a positive integer such that
$m\geqslant \mathrm{mult}_Z^2(\mathcal{M})>\frac{n^2}{4}$. Thus, we have
$$
3n^2=H\cdot M\cdot M^\prime=m\mathrm{deg}(Z)+H\cdot\Delta\geqslant\mathrm{mult}^2_Z\big(\mathcal{M}\big)\mathrm{deg}\big(Z\big)>\frac{n^2}{4}\mathrm{deg}(Z).
$$
Then $\mathrm{deg}(Z)<12$, and it follows from Lemmas~\ref{lemma:A5-reducible-curves}, \ref{lemma:A5-irreducible-curves-in-Q}, \ref{lemma:A5-irreducible-curves-in-S}, \ref{lemma:A5-irreducible-curves-degree-10} that one of the following cases holds:
\begin{enumerate}
\item[$(\mathrm{i})$] $Z=\mathcal{L}_6$ or $Z=\mathcal{L}_6^\prime$;
\item[$(\mathrm{ii})$] $Z=\mathcal{L}^t_{10}$ for some $t\in\mathbb{C}$ such that $(b+3)t^3+3b\zeta_3t^2+1=0$;
\item[$(\mathrm{iii})$] $Z$ is a smooth irreducible rational curve of degree $8$, and $Z\subset Q$;
\item[$(\mathrm{iv})$] $Z$ is a smooth irreducible curve of degree $10$ and genus $6$.
\end{enumerate}

Let $\pi\colon \widetilde{X}\to X$ be the~blow up~of~the~curve $Z$, let $F$ be the $\pi$-exceptional divisor, let $\widetilde{H}=\pi^*(H)$, let~$\widetilde{M}$ and $\widetilde{M}^\prime$ be the~strict transforms on $\widetilde{X}$ of the surfaces $M$ and $M^\prime$, respectively. Then
$$
\widetilde{M}\sim \widetilde{M}^\prime\sim n\widetilde{H}-\mathrm{mult}_Z\big(\mathcal{M}\big)F.
$$
Note that $\widetilde{M}\cdot\widetilde{M}^\prime$ is an effective one-cycle.
Take $k\in\mathbb{Z}_{>0}$ such that the divisor $k\widetilde{H}-F$ is nef. Then
\begin{equation}
\label{equation:A5-test-class}
0\leqslant \widetilde{M}\cdot \widetilde{M}^\prime\cdot \big(k\widetilde{H}-F\big)=3kn^2-2n\mathrm{deg}(Z)\mathrm{mult}_Z\big(\mathcal{M}\big)-(k\mathrm{deg}(Z)+F^3)\mathrm{mult}^2_Z\big(\mathcal{M}\big).
\end{equation}
Note that $Z$ is smooth unless $Z=\mathcal{L}^t_{10}$ and $t=\pm\frac{\zeta_3+2}{3}$. Moreover, we have
$$
F^3=\left\{\aligned
&0\ \text{if $Z=\mathcal{L}_6$ or $Z=\mathcal{L}_6^\prime$},\\
&0\ \text{if $Z=\mathcal{L}^t_{10}$ for $t\in\mathbb{C}$ such that $(b+3)t^3+3b\zeta_3t^2+1=0$ and $t\ne\pm\frac{\zeta_3+2}{3}$},\\
&-30\ \text{if $Z=\mathcal{L}^t_{10}$ for $t=\pm\frac{\zeta_3+2}{3}$},\\
&-14\ \text{if $Z$ is a smooth rational curve of degree $8$},\\
&-30\ \text{if $Z$ is a smooth curve of degree $10$ and genus $6$}.
\endaligned
\right.
$$

If $Z=\mathcal{L}_6$ or $Z=\mathcal{L}_6^\prime$, then $3\widetilde{H}-F$ is nef by Lemma~\ref{lemma:A5-L6}, and we can let $k=3$ in \eqref{equation:A5-test-class} to get 
$$
0\leqslant \widetilde{M}\cdot \widetilde{M}^\prime\cdot \big(3\widetilde{H}-F\big)=
9n^2-12n\mathrm{mult}_Z\big(\mathcal{M}\big)-18\mathrm{mult}^2_Z\big(\mathcal{M}\big),
$$
which gives $\mathrm{mult}_Z(\mathcal{M})\leqslant\frac{n}{2}$. 

If $Z=\mathcal{L}^t_{10}$ for  $t\in\mathbb{C}$ such that $(b+3)t^3+3b\zeta_3t^2+1=0$ and $t\ne\pm\frac{\zeta_3+2}{3}$, then $Z$ is a disjoint union of $10$ lines, so the divisor $10\widetilde{H}-F$ is nef, and, therefore, we can let $k=10$ in \eqref{equation:A5-test-class} to get 
$$
0\leqslant \widetilde{M}\cdot \widetilde{M}^\prime\cdot \big(10\widetilde{H}-F\big)=
30n^2-20n\mathrm{mult}_Z\big(\mathcal{M}\big)-100\mathrm{mult}^2_Z\big(\mathcal{M}\big),
$$
which gives $\mathrm{mult}_Z(\mathcal{M})\leqslant\frac{n}{2}$. 

If $Z=\mathcal{L}^t_{10}$ for  $t=\pm\frac{\zeta_3+2}{3}$, then $3\widetilde{H}-F$ is nef by Lemma~\ref{lemma:A5-L10}. If $Z$ is a curve of degree $10$ and genus $6$ such that $Z\not\subset Q$, then $3\widetilde{H}-F$ is nef by Corollary~\ref{corollary:A5-3-regular}. In both cases,  \eqref{equation:A5-test-class} gives
$$
0\leqslant \widetilde{M}\cdot \widetilde{M}^\prime\cdot \big(3\widetilde{H}-F\big)=
9n^2-20n\mathrm{mult}_Z\big(\mathcal{M}\big),
$$
which implies $\mathrm{mult}_Z(\mathcal{M})\leqslant\frac{n}{2}$. 

Hence, we may further assume that $Z\subset Q$, and either $Z$ is a smooth rational curve of degree~$8$, or~$Z$ is a smooth curve of degree $10$ and genus $6$. Let us show that $\mathrm{mult}_Z(\mathcal{M})\leqslant\frac{n}{2}$.

Set $H_{Q}=H\vert_{Q}$. We claim that $|3H_Q-Z|$ does not have fixed curves. Indeed, write
$$
|3H_Q-Z|=\mathscr{F}+|\mathscr{M}|,
$$
where $\mathscr{F}$ is the~fixed part of this linear system, and $\mathscr{M}$ is a  divisor  such that $|\mathscr{M}|$ is the mobile part of the linear system $|3H_Q-Z|$. We may assume that $\mathscr{M}$ is a general curve in $|\mathscr{M}|$. Then
$$
H_Q\cdot \mathscr{M}=\mathrm{deg}\big(\mathscr{M}\big)\geqslant 3,
$$
because $Q$ is a smooth K3 surface, so it is not uniruled. On the other hand, $\mathscr{F}$ is $G$-invariant, and
$$
18-\mathrm{deg}(Z)=H_Q\cdot\big(3H_Q-Z\big)=H_Q\cdot\big(\mathscr{F}+\mathscr{M}\big)\geqslant H_Q\cdot\mathscr{F}+3=\mathrm{deg}\big(\mathscr{F}\big)+3.
$$
Thus, $\mathrm{deg}(\mathscr{F})\leqslant 7$, which gives $\mathscr{F}=0$ by Lemmas~\ref{lemma:A5-reducible-curves}, \ref{lemma:A5-irreducible-curves-in-Q}, \ref{lemma:A5-irreducible-curves-in-S} and Corollary~\ref{corollary:A5-irreducible-curves}, because $Q$ contains neither $\mathcal{L}_{6}$ nor $\mathcal{L}_{6}^\prime$ by Lemma~\ref{lemma:A5-Q-S-orbits}. Hence, $|3H_Q-Z|$ does not have fixed curves. 

In particular, the divisor $3H_Q-Z$ is nef.  Let $M$ be a general surface in $\mathcal{M}$. Then
$$
M\big\vert_{Q}=mZ+D
$$
for an integer $m\geqslant \mathrm{mult}_Z(\mathcal{M})$ and an
effective divisor $D$ on the~surface $Q$. If $Z$ is a rational curve of degree $8$, then $H_Q\cdot Z=8$ and $Z^2=-2$, so 
$$
0\leqslant \big(3H_Q-Z\big)\cdot D=\big(3H_Q-Z\big)\cdot\big(nH_Q-mZ\big)=10n-26m,
$$
which gives $\mathrm{mult}_Z(\mathcal{M})\leqslant m\leqslant\frac{n}{2}$. If $Z$ is a curve of
degree $10$, then $H_Q\cdot Z=10$ and $Z^2=10$, so 
$$
0\leqslant \big(3H_Q-Z\big)\cdot D=\big(3H_Q-Z\big)\cdot\big(nH_Q-mZ\big)=8n-20m,
$$
which gives $\mathrm{mult}_Z(\mathcal{M})\leqslant m\leqslant\frac{n}{2}$. 
\end{proof}

Now, we are ready to prove the following result, which implies Theorem~\ref{theorem:A5}. 

\begin{proposition}
\label{proposition:A5-non-standard}
Suppose that $G=\langle M_2,M_3\rangle$. Then $X$ is $G$-birationally superrigid.
\end{proposition}

\begin{proof}
Suppose that $X$ is not $G$-birationally superrigid. By \cite[Corollary 3.3.3]{CheltsovShramov2016}, there exists  a~non-empty $G$-invariant mobile linear subsystem $\mathcal{M}\subset |nH|$, for a positive integer~$n$, such that the~singularities of the log pair $(X,\frac{2}{n}\mathcal{M})$ are not canonical.
Let 
$$
\Sigma=\Big\{P\in X\ \text{such that the~log pair $\Big(X,\frac{2}{n}\mathcal{M}\Big)$ is not canonical at~$P$}\Big\}. 
$$
Then $|\Sigma|<\infty$ by \cite[Exercise 6.18]{KollarSmithCorti} and Proposition~\ref{proposition:A5-NFI-curves}, so the~log pair $(X,\frac{3}{n}\mathcal{M})$ is not log canonical at every point of $\Sigma$ by \cite[Remark 3.6]{CheltsovSarikyanZhuang2024}. Take $\mu\in\mathbb{Q}_{>0}$ such that $(X,\mu\mathcal{M})$ is strictly log canonical. Then $\mu\leqslant \frac{3}{n}$.

Let $C$ be an irreducible subvariety in $X$ that is a minimal center of log canonical singularities of the~log pair $(X,\mu\mathcal{M})$, and let $Z$ be a $G$-irreducible subvariety in $X$ that has $C$ as an irreducible component. Then all components of $Z$ are minimal centers of log canonical singularities of  $(X,\mu\mathcal{M})$, which implies that these irreducible components are disjoint \cite[Proposition 1.5]{Kawamata-1}. Note that 
$$
\mathrm{dim}(Z)\ne 2,
$$
because $\mathcal{M}$ is mobile. Hence, either $Z$ is a $G$-irreducible curve or $Z$ is the~$G$-orbit of a point. 

Now, arguing as in the proof of Theorem~\ref{theorem:405-15}, we can replace the~rational number $\mu$ and the~mobile linear system $\mathcal{M}$ by a rational number $\lambda<3$ and an effective $G$-invariant $\mathbb{Q}$-divisor $D\sim_{\mathbb{Q}} H$
such that the~singularities of the~log pair $(X,\lambda D)$ are strictly log canonical, $C$ is a minimal center of the~non-Kawamata log terminal singularities of this log pair, and
$$
\mathrm{Nklt}\big(X,\lambda D\big)=Z,
$$
where $\mathrm{Nklt}(X,\lambda D)=\big\{P\in X\ \text{such that $(X,\lambda D)$ is not Kawamata log terminal at $P$}\big\}$.

Let $\mathcal{I}_Z$ be the~ideal sheaf of the~subvariety $Z\subset X$. Then $\mathcal{I}_Z$ is the~multiplier ideal sheaf of the~log pair $(X,\lambda D)$, and it follows from the Nadel vanishing theorem \cite[Theorem 9.4.8]{Lazarsfeld2004} that 
$$
h^1\big(X,\mathcal{I}_Z\otimes\mathcal{O}_{X}(H)\big)=0.
$$
Hence, we have the~following exact sequence of $G$-representations:
$$
0\longrightarrow H^0\big(X,\mathcal{I}_Z\otimes\mathcal{O}_{X}(H)\big)\longrightarrow H^0\big(X,\mathcal{O}_{X}(H)\big)\longrightarrow H^0\big(Z,\mathcal{O}_{Z}(H\vert_{Z})\big)\longrightarrow 0.
$$
On the other hand, $H^0(X,\mathcal{I}_Z\otimes\mathcal{O}_{X}(H))=0$,
because $\mathbb{V}_5$ is irreducible. Thus, we get 
\begin{equation}
\label{equation:A5-final}
h^0\big(Z,\mathcal{O}_{Z}(H\vert_{Z})\big)=5.
\end{equation}
Thus, if $Z$ is a $G$-orbit of a point, then $|Z|=5$, which is impossible by Lemma~\ref{lemma:A5-orbits}.

Thus, we conclude that $Z$ is a curve. Then it follows from \cite[Lemma 1.8]{Corti2000} that $\mathrm{deg}(Z)\leqslant 6$. Indeed, let $M$ and $M^\prime$ be general surfaces in $\mathcal{M}$. Then it follows from \cite[Lemma 1.8]{Corti2000} that
$$
\big(M\cdot M^\prime\big)_Z\geqslant\frac{4}{\mu^2}>\frac{4n^2}{9}
$$
which gives $\mathrm{deg}(Z)\leqslant 6$. Then  $Z=\mathcal{L}_6$ or $Z=\mathcal{L}_6^\prime$ by  
Lemmas~\ref{lemma:A5-reducible-curves}, \ref{lemma:A5-irreducible-curves-in-Q}, \ref{lemma:A5-irreducible-curves-in-S} and Corollary~\ref{corollary:A5-irreducible-curves}. Then 
$$
h^0\Big(Z,\mathcal{O}_{Z}\big(H\vert_{Z}\big)\Big)=12,
$$
because $Z$ is a disjoint union of $6$ lines. This contradicts \eqref{equation:A5-final}. 
\end{proof}

By Propositions~\ref{proposition:A5-standard} and \ref{proposition:A5-non-standard},  Theorem~\ref{theorem:A5} is proved.

\section{Groups $C_3^2\rtimes C_4$ and $\mathfrak{S}_3\times\mathfrak{S}_3$}
\label{section:36-9-36-10}

Let $X$ be a smooth cubic threefold in $\mathbb{P}^4$. Suppose that $\mathrm{Aut}(X)$ has a subgroup $G$ such that
\begin{itemize}
\item either $G\simeq C_3^2\rtimes C_4$ with GAP ID [36,9],
\item or $G\simeq\mathfrak{S}_3\times\mathfrak{S}_3$ with GAP ID [36,10].
\end{itemize}
Then the~$G$-action on $X$ lifts to $\mathbb{P}^4$, so we can consider $G$ as a subgroup in $\mathrm{PGL}_5(\mathbb{C})$.
We suppose, in~addition, that $\mathbb{P}^4$ contains no $G$-invariant planes. In this section, we prove the following result:

\begin{theorem}
\label{theorem:36-9-smooth}
The cubic threefold $X$ is $G$-birationally superrigid.
\end{theorem}

First, we state few auxiliary results about the groups $C_3^2\rtimes C_4$ and $\mathfrak{S}_3\times\mathfrak{S}_3$.

\begin{lemma}
\label{lemma:36-9-subgroups}
Conjugacy classes of non-trivial proper subgroups of $C_3^2\rtimes C_4$ are described as follows:
\begin{center}\renewcommand{\arraystretch}{1.5}
\begin{tabular}{|c||c|c|c|c|c|c|c|c|}
\hline
$\mathcal{C}$&$\mathcal{C}_2$&$\mathcal{C}_3$&$\mathcal{C}_3^\prime$&$\mathcal{C}_4$&$\mathcal{C}_6$&$\mathcal{C}_6^\prime$&$\mathcal{C}_9$&$\mathcal{C}_{18}$\\\hline
$H\in \mathcal{C}$ &$C_2$& $C_3$&$C_3$& $C_4$&$\mathfrak{S}_3$&$\mathfrak{S}_3$&$C_3^2$&$C_3\rtimes\mathfrak{S}_3$\\
\hline
$|\mathcal{C}|$ &$9$& $2$&$2$& $9$&$6$&$6$&$1$&$1$\\
\hline
$[C_3^2\rtimes C_4:H]$ & $18$ & $12$ & $12$ & $9$ & $6$ & $6$ & $4$ & $2$\\
\hline
\end{tabular}
\end{center}
\end{lemma}

\begin{corollary}
\label{corollary:36-9-act}
Let $\Sigma$ be a finite set such that $|\Sigma|>1$ and $C_3^2\rtimes C_4$ acts transitively on $\Sigma$. Then
$$
|\Sigma|\in \{2,4,6,9,12,18,36\}.
$$
\end{corollary}

\begin{lemma}
\label{lemma:36-10-subgroups}
Conjugacy classes of non-trivial proper subgroups of $\mathfrak{S}_3\times\mathfrak{S}_3$ are described as follows:
\begin{center}\renewcommand{\arraystretch}{1.5}
\begin{tabular}{|c||c|c|c|c|c|c|c|c|c|}
\hline
$\mathcal{C}$&$\mathcal{C}_2$&$\mathcal{C}_2^\prime$&$\mathcal{C}_2^{\prime\prime}$&$\mathcal{C}_3$&$\mathcal{C}_3^\prime$&$\mathcal{C}_3^{\prime\prime}$&$\mathcal{C}_4$&$\mathcal{C}_6$&$\mathcal{C}_{6}^\prime$\\
\hline
$H\in \mathcal{C}$ &$C_2$& $C_2$&$C_2$& $C_3$&$C_3$&$C_3$&$C_2^2$&$C_6$&$C_6$\\
\hline
$|\mathcal{C}|$ &$3$& $3$&$9$& $1$&$1$&$2$&$9$&$3$&$3$\\
\hline
$[\mathfrak{S}_3\times\mathfrak{S}_3:H]$ & $18$ & $18$ & $18$ & $12$ & $12$ & $12$ & $9$& $6$ & $6$\\
\hline
\end{tabular}
\end{center}
\begin{center}\renewcommand{\arraystretch}{1.5}
\begin{tabular}{|c||c|c|c|c|c|c|c|c|c|c|c|}
\hline
$\mathcal{C}$&$\mathcal{C}_6^{\prime\prime}$&$\mathcal{C}_6^{\prime\prime\prime}$&$\mathcal{C}_6^{\prime\prime\prime\prime}$&$\mathcal{C}_6^{\prime\prime\prime\prime\prime}$&$\mathcal{C}_6^{\prime\prime\prime\prime\prime\prime}$&$\mathcal{C}_9$&$\mathcal{C}_{12}$&$\mathcal{C}_{12}^\prime$&$\mathcal{C}_{18}$ & $\mathcal{C}_{18}^\prime$& $\mathcal{C}_{18}^{\prime\prime}$\\
\hline
$H\in \mathcal{C}$ &$\mathfrak{S}_3$& $\mathfrak{S}_3$&$\mathfrak{S}_3$& $\mathfrak{S}_3$&$\mathfrak{S}_3$&$C_3^2$&$\mathfrak{D}_6$&$\mathfrak{D}_6$&$C_3\times \mathfrak{S}_3$ & $C_3\times \mathfrak{S}_3$ &$C_3\rtimes \mathfrak{S}_3$ \\
\hline
$|\mathcal{C}|$ &$1$& $1$&$3$& $3$&$6$&$1$&$3$&$3$&$1$&$1$ &$1$ \\
\hline
$[\mathfrak{S}_3\times\mathfrak{S}_3:H]$ & $6$ & $6$ & $6$ & $6$ & $6$ & $4$ & $3$& $3$ & $2$ & $2$& $2$\\
\hline
\end{tabular}
\end{center}
\end{lemma}

\begin{corollary}
\label{corollary:36-10-act}
Let $\Sigma$ be a finite set such that $|\Sigma|>1$ and $\mathfrak{S}_3\times\mathfrak{S}_3$ acts transitively on $\Sigma$. Then
$$
|\Sigma|\in \{2,3,4,6,9,12,18,36\}.
$$
\end{corollary}

\begin{lemma}
\label{lemma:36-9-reps}
Irreducible complex representations of the~group $C_3^2\rtimes C_4$ can be described as follows:
\begin{enumerate}
\item[$(\mathrm{1})$] $\mathbb{I}$ is the~trivial $1$-dimensional representation,
\item[$(\mathrm{2})$] $\mathbb{V}_1$ is the~$1$-dimensional representation of order $2$,
\item[$(\mathrm{3})$] $\mathbb{V}_1^\prime$ is a $1$-dimensional representation of order $4$,
\item[$(\mathrm{4})$] $\mathbb{V}_1^{\prime\prime}$ is a $1$-dimensional representation of order $4$,
\item[$(\mathrm{5})$] $\mathbb{V}_4$ is a faithful $4$-dimensional  representation,
\item[$(\mathrm{6})$] $\mathbb{V}_4^\prime$ is a faithful $4$-dimensional  representation.
\end{enumerate}
The representations $\mathbb{V}_4$ and $\mathbb{V}_4^\prime$ differ by an automorphism of the~group $G$.
\end{lemma}

\begin{lemma}
\label{lemma:36-10-reps}
Irreducible complex representations of the~group $\mathfrak{S}_3\times\mathfrak{S}_3$ can be described as follows:
\begin{enumerate}
\item[$(\mathrm{1})$] $\mathbb{I}$ is the~trivial $1$-dimensional representation,
\item[$(\mathrm{2})$] $\mathbb{V}_1$ is a~$1$-dimensional representation of order $2$,
\item[$(\mathrm{3})$] $\mathbb{V}_1^\prime$ is a $1$-dimensional representation of order $2$,
\item[$(\mathrm{4})$] $\mathbb{V}_1^{\prime\prime}$ is a $1$-dimensional representation of order $2$,
\item[$(\mathrm{5})$] $\mathbb{V}_2$ is a $2$-dimensional  representation lifted from $\mathfrak{S}_3$,
\item[$(\mathrm{6})$] $\mathbb{V}_2^\prime$ is a $2$-dimensional representation lifted from $\mathfrak{D}_6$,
\item[$(\mathrm{7})$] $\mathbb{V}_2^{\prime\prime}$ is a $2$-dimensional  representation lifted from $\mathfrak{S}_3$,
\item[$(\mathrm{8})$] $\mathbb{V}_2^{\prime\prime\prime}$ is a $2$-dimensional  representation lifted from $\mathfrak{D}_6$,
\item[$(\mathrm{9})$] $\mathbb{V}_4$ is a faithful $4$-dimensional  representation.
\end{enumerate}
\end{lemma}

We proceed to actions of the~groups $C_3^2\rtimes C_4$ and $\mathfrak{S}_3\times\mathfrak{S}_3$ on curves.

\begin{lemma}
\label{lemma:36-9-curves}
Let $C$ be a smooth irreducible curve that admits a faithful action of the~group $C_3^2\rtimes C_4$, and let $g$ be the genus of the curve $C$. Then the~following assertions hold:
\begin{itemize}
\item orbits in $C$ are of length $9$, $12$, $18$ or $36$;
\item the~curve $C$ is not hyperelliptic;
\item if $g<10$, then either $g=1$ or $g=4$;
\item if $g=4$, then orbits of length $<36$ in $C$ can be described as follows:
\begin{itemize}
\item one orbit of length $12$;
\item two orbits of length $9$;
\end{itemize}
\item if $g=1$ and $C$ has at least one orbit of length $9$, then $C$ has an orbit of length $18$, it has two orbits of length $9$, and all other orbits in $C$ are of length $36$.
\end{itemize}
\end{lemma}

\begin{proof}
If $P$ is a point in $C$, its stabilizer in $C_3^2\rtimes C_4$ acts faithfully on the~tangent space $T_{P}(C)$, so it must be cyclic, which implies that orbits in $C$ have length $9$, $12$, $18$ or $36$ by Lemma~\ref{lemma:36-9-subgroups}.

The curve $C$ is not hyperelliptic, since $C_3^2\rtimes C_4$ is not contained in $\mathrm{PGL}_2(\mathbb{C})$ and it does not have normal subgroups isomorphic to $C_2$ by Lemma~\ref{lemma:36-9-subgroups}.

The genus bound follows from \cite{Breuer,Jen,lmfdb}.

If $g=4$, then it follows from \cite{lmfdb} that the~curve $C$ has one orbit of length $12$, and $C$ has two orbits of length $9$.

Finally, we suppose that $g=1$, and the curve $C$  contains at least one $C_3^2\rtimes C_4$-orbit of length~$9$. Then the quotient of the curve $C$ by the group $C_3^2\rtimes C_4$
is isomorphic to $\mathbb{P}^1$. Moreover, applying the~Hurwitz formula to the~quotient map $C\to\mathbb{P}^1$, we get
$$
72=1\times 18a_{18}+2\times 12a_{12}+3\times 9a_{9}.
$$
where $a_{18}$, $a_{12}$ and $a_{9}$ are the~numbers of the~orbits in $C$ of length $18$, $12$ and $9$, respectively.
Since $a_{9}\geqslant 1$, we get $a_{18}=1$, $a_{12}=0$ and $a_9=2$ as required.
\end{proof}

\begin{lemma}
\label{lemma:36-10-curves}
Let $C$ be a smooth irreducible curve that admits a faithful action of the~group $\mathfrak{S}_3\times\mathfrak{S}_3$, and let $g$ be the genus of the curve $C$. Then the~following assertions hold:
\begin{itemize}
\item orbits in $C$ are of length $6$, $12$, $18$ or $36$;
\item the~curve $C$ is not hyperelliptic;
\item if $g<10$, then either $g=4$ or $g=7$;
\item if $g=7$, then orbits of length $<36$ in $C$ can be described as follows:
\begin{itemize}
\item three orbits of length $18$;
\item one orbit of length $6$;
\end{itemize}
\item if $g=4$, then orbits of length $<36$ in $C$ can be described as follows:
\begin{itemize}
\item either $C$ has one orbit of length $18$ and two orbits of length $6$;
\item or $C$ has  three orbits of length $18$ and one orbit of length $12$.
\end{itemize}
\end{itemize}
\end{lemma}

\begin{proof}
If $g=1$, $\mathfrak{S}_3\times\mathfrak{S}_3$ has an abelian subgroup $H$ such that $\mathfrak{S}_3\times\mathfrak{S}_3/H\simeq C_n$ for $n\in\{2,3,4,6\}$, which is impossible by Lemma~\ref{lemma:36-10-subgroups}.
This gives  $g\ne 1$. Now, arguing as in the proof of Lemma~\ref{lemma:36-9-curves} and using Lemma~\ref{lemma:36-10-subgroups}, we obtain all remaining assertions.
\end{proof}

Let us classify $G$ up to conjugation in  $\mathrm{PGL}_5(\mathbb{C})$. To do this, consider the following projective transformations:
$$
M_1=\begin{pmatrix}
\zeta_3 &  0  & 0 & 0 & 0 \\
 0  & 1 & 0 & 0 & 0 \\
 0  &   0   & \zeta_3^2  & 0 & 0 \\
 0  &   0   & 0 & 1 & 0 \\
 0  &   0   & 0 & 0 & 1
\end{pmatrix}, \quad
M_2=\begin{pmatrix}
1 & 0 & 0 & 0 & 0 \\
0 & \zeta_3 & 0 & 0 & 0 \\
0 & 0 & 1 & 0 & 0 \\
0 & 0 & 0 & 1 & 0 \\
0 & 0 & 0 & 0 & \zeta_3^2
\end{pmatrix},
$$
$$
M_3=\begin{pmatrix}
0 & 1 & 0 & 0 & 0 \\
0 & 0 & 1 & 0 & 0 \\
0 & 0 & 0 & 0 & 1 \\
0 & 0 & 0 & 1 & 0 \\
1 & 0 & 0 & 0 & 0
\end{pmatrix},\quad
M_4=\begin{pmatrix}
0 & 0 & 1 & 0 & 0 \\
0 & 0 & 0 & 0 & 1 \\
1 & 0 & 0 & 0 & 0 \\
0 & 0 & 0 & 1 & 0 \\
0 & 1 & 0 & 0 & 0
\end{pmatrix},\quad 
M_5=\begin{pmatrix}
0 & 1 & 0 & 0 & 0 \\
1 & 0 & 0 & 0 & 0 \\
0 & 0 & 0 & 0 & 1 \\
0 & 0 & 0 & 1 & 0 \\
0 & 0 & 1 & 0 & 0 
\end{pmatrix}.
$$
Let $G_{36,9}$, $G_{36,10}$ and $G_{72,40}$ be subgroups in $\mathrm{PGL}_5(\mathbb{C})$ defined as follows:
\begin{align*}
G_{36,9}&=\langle M_1,M_2,M_3\rangle,\\
G_{36,10}&=\langle M_1,M_2,M_4,M_5\rangle,\\
G_{72,40}&=\langle M_1,M_2,M_3,M_4,M_5\rangle.
\end{align*}
Then $G_{36,9}\simeq C_3^2\rtimes C_4$, $G_{36,10}\simeq \mathfrak{S}_3\times\mathfrak{S}_3$, $G_{72,40}\simeq \mathfrak{S}_3\wr C_2$,
and the subgroup $G_{72,40}$ contains both subgroups $G_{36,9}$ and $G_{36,10}$ as normal subgroups.
We claim that $G$ is conjugated to $G_{36,9}$ or $G_{36,10}$.

Indeed, the subgroup $G\subset\mathrm{PGL}_5(\mathbb{C})$ lifts isomorphically to $\mathrm{GL}_5(\mathbb{C})$ by \cite[Lemma 2.4]{AbbanCheltsovKishimotoMangolte2025}, and the~embedding $G\hookrightarrow\mathrm{GL}_5(\mathbb{C})$ is given by a $5$-dimensional $G$-representation $\mathbb{U}$,
which does not contain $3$-dimensional subrepresentations.
Thus, it follows from Lemmas~\ref{lemma:36-9-reps} and \ref{lemma:36-10-reps} that
$$
\mathbb{U}\simeq \mathbb{U}_1\oplus\mathbb{U}_4,
$$
where $\mathbb{U}_1$ is a one-dimensional representation of the group $G$, and $\mathbb{U}_4$ is its $4$-dimensional irreducible representation.
Now, tensoring $\mathbb{U}$ by the representation $\mathbb{U}_1^\vee$, we may  assume that $\mathbb{U}_1$ is the trivial representation.
Hence, using Lemmas~\ref{lemma:36-9-reps} and \ref{lemma:36-10-reps} and the notations introduced in these lemmas, we may assume that
$$
\mathbb{U}=\mathbb{I}\oplus\mathbb{V}_4,
$$
Thus, $G$ is conjugated to $G_{36,9}$ or $G_{36,10}$.
Hence, to prove  Theorem~\ref{theorem:36-9}, we may assume that
\begin{itemize}
\item either $G=G_{36,9}\simeq C_3^2\rtimes C_4$,
\item or $G=G_{36,10}\simeq\mathfrak{S}_3\times\mathfrak{S}_3$.
\end{itemize}
Now, we can find the defining equation of $X$. Namely, for every $a\in \mathbb{C}$, let
$$
X_a=\Big\{x_0^3+x_1^3+x_2^3+x_3^3+x_4^3+ax_3(x_0x_2+x_1x_4)=0\Big\}\subset\mathbb{P}^4.
$$
Then $X_a$ is an irreducible normal cubic hypersurface in $\mathbb{P}^4$, which has at most isolated singularities. Note that $X_a$ is $G_{72,40}$-invariant, so it is $G_{36,9}$-invariant and $G_{36,10}$-invariant.

\begin{lemma}
\label{lemma:36-9-equation}
One has $X=X_a$ for some $a\in\mathbb{C}$.
\end{lemma}

\begin{proof}
Let us only consider the case $G=G_{36,9}$, since the proof is similar in the case $G=G_{36,10}$.
In the notations of Lemma~\ref{lemma:36-9-reps}, we have
$$
\mathrm{Sym}^3\big(\mathbb{I}\oplus\mathbb{V}_4\big)\simeq \mathbb{I}\oplus \mathbb{I}\oplus \mathbb{I}\oplus \mathbb{V}_1\oplus \mathbb{V}_1\oplus \mathbb{V}_1^\prime \oplus \mathbb{V}_1^{\prime\prime}\oplus \mathbb{V}_4\oplus \mathbb{V}_4\oplus \mathbb{V}_4\oplus \mathbb{V}_4\oplus \mathbb{V}_4^\prime\oplus \mathbb{V}_4^\prime\oplus\mathbb{V}_4^\prime.
$$
In this decomposition, the summand $\mathbb{V}_1\oplus \mathbb{V}_1$ corresponds to the singular cubics given by
$$
\lambda x_3(x_0x_2-x_1x_4)+\mu (x_0^3-x_1^3+x_2^3-x_4^3)=0,
$$
where $[\lambda:\mu]\in\mathbb{P}^1$. Similarly, the summands $\mathbb{V}_1^\prime$ and $\mathbb{V}_1^{\prime\prime}$ correspond to the cubic cones
$$
\{x_0^3\pm ix_1^3-x_2^3\mp ix_4^3=0\}\subset\mathbb{P}^4.
$$
Finally, the summand $\mathbb{I}\oplus \mathbb{I}\oplus \mathbb{I}$ corresponds to the cubic threefolds
\begin{equation}
\label{equation:36-9-cubics}
ax_3(x_0x_2+x_1x_4)+bx_3^3+c(x_0^3+x_1^3+x_2^3+x_3^3+x_4^3)=0,
\end{equation}
where $[a:b:c]\in\mathbb{P}^2$. Thus, $X$ is given by \eqref{equation:36-9-cubics} with $c\ne 0$.
Hence, we may assume that $c=1$. Moreover, scaling $x_3$ appropriately, we may further assume that $b=0$, so $[a:b:c]=[a:0:1]$, which simply means that $X=X_a$.
\end{proof}

If the cubic $X_a$ is smooth, then 
$$
-27\ne a^3\ne -\frac{27}{2}.
$$
Indeed, if $a^3=-27$, then $X_a$ has six nodes. If $a^3=-\frac{27}{2}$, then $X_a$ is the unique nine-nodal cubic threefold that admits an effective $G_{72,40}$-action \cite{CheltsovTschinkelZhang2025}, so $X_a$ is isomorphic to the cubic \eqref{equation:cubic-9-nodes}.

In this section, we will prove the following result, which implies both Theorems~\ref{theorem:36-9-smooth} and \ref{theorem:singular}, and it also implies \cite[Theorem 9.6]{CheltsovTschinkelZhang2025}.

\begin{theorem}
\label{theorem:36-9}
If $X_a$ is smooth or $a^3=-\frac{27}{2}$, then $X_a$ is $G$-birationally superrigid.
\end{theorem}

In the remaining part of this section, we will always assume that 
\begin{itemize}
\item either $X_a$ is smooth,
\item or $a^3=-\frac{27}{2}$ and $X_a$ has nine nodes.
\end{itemize}
In particular, $a^3\ne -27$. If $a^3=-\frac{27}{2}$, we may further assume that $a=-\frac{3}{\sqrt[3]{2}}$, where $\sqrt[3]{2}\in\mathbb{R}$.

\begin{remark}
If $a=-\frac{3}{\sqrt[3]{2}}$, the singular locus $\mathrm{Sing}(X_a)$ is the $G$-orbit of the point $[1:1:1:\sqrt[3]{2}:1]$, which consists of $9$ nodes. In this case, $X_a$ contains $9$ planes $\Pi_{i,j}$ given by
$$
\left\{\aligned
&2\zeta_3^ix_0+2\zeta_3^{2i}x_2+\sqrt[3]{4}x_3=0,\\
&2\zeta_3^jx_1+2\zeta_3^{2j}x_4+\sqrt[3]{4}x_3=0,
\endaligned
\right.
$$
where $i,j\in\{0,1,2\}$. Each plane $\Pi_{i,j}$ contains exactly $4$ singular points of $X_a$. Moreover, the subgroup $\langle M_1,M_2\rangle\subset G$ permutes these $9$ planes transitively. Since the classes of the planes $\Pi_{i,j}$ generate $\mathrm{Cl}(X_a)$ by \cite[Section~9]{CheltsovTschinkelZhang2025}, we conclude that $\mathrm{rk}\,\mathrm{Cl}(X_a)^G=1$, so   $X_a$ is indeed a $G$-Mori fiber space (over a point) as claimed in \cite{Avilov2016b,CheltsovTschinkelZhang2025}.
\end{remark}

Let $H$ be a general hyperplane section of $X_{a}$, and let
$$
S=\{x_3=0\}\cap X_{a}.
$$
Then $S$ is a smooth $G$-invariant cubic surface. Note that $S$ is the~only $G$-invariant surface in $|H|$,
the surface $S$ is $G_{72,40}$-invariant, and the~group $G_{72,40}$ acts faithfully on $S$.

\begin{lemma}
\label{lemma:36-9-S-Pic}
Suppose that $G=G_{36,9}$. Then $\mathrm{Pic}(S)^G=\mathbb{Z}[-K_S]$.
\end{lemma}

\begin{proof}
The required assertion follows from \cite[Theorem 6.14]{DolgachevIskovskikh2009} or \cite{CheltsovTschinkelZhang2026}, and can be easily checked explicitly by looking how $G=G_{36,9}$ acts on the lines in $S$.
\end{proof}

If $G=G_{36,10}$, then $\mathrm{Pic}(S)^G\ne\mathbb{Z}[-K_S]$.
Indeed, let $\mathcal{L}_3$ be the union of the following lines:
\begin{align*}
\big\{x_3&=0, x_0+x_1=0, x_2+x_4=0\big\},\\
\big\{x_3&=0, x_0+\zeta_3 x_1=0, x_2+\zeta_3^2x_4=0\big\},\\
\big\{x_3&=0, x_0+\zeta_3^2x_1=0, x_2+\zeta_3x_4=0\big\},
\end{align*}
and let $\mathcal{L}_3^\prime$ be the union of the following lines:
\begin{align*}
\big\{x_3&=0, x_0+x_4=0, x_1+x_2=0\big\},\\
\big\{x_3&=0, x_0+\zeta_3x_4=0, x_1+\zeta_3x_2=0\big\},\\
\big\{x_3&=0, x_0+\zeta_3^2x_4=0, x_1+\zeta_3^2x_2=0\big\}.
\end{align*}
Then $\mathcal{L}_3$ and $\mathcal{L}_3^\prime$ are both $G_{36,10}$-irreducible curves in $S$, and each of these curves consists of three disjoint lines. In particular, if $G=G_{36,10}$, then $\mathrm{Pic}(S)^G\ne\mathbb{Z}[-K_S]$. Set $\mathcal{L}_6=\mathcal{L}_3+\mathcal{L}_3^\prime$. Then
\begin{equation}
\label{equation:39-9-L6}
\mathcal{L}_6=\big\{x_0x_2-x_1x_4=0\big\}\cap S,
\end{equation}
and $\mathcal{L}_6$ is $G_{36,9}$-irreducible.

\begin{lemma}
\label{lemma:36-10-S-Pic}
Suppose that $G=G_{36,10}$. Then $\mathrm{rk}\,\mathrm{Pic}(S)^G=2$, and there is a $G$-Sarkisov link
$$
\xymatrix{
&S\ar@{->}[ld]_\pi\ar@{->}[rd]^{\pi^\prime} & \\
Y\ar@{-->}[rr]_\tau  && Y}
$$
where $Y$ is the sextic smooth del Pezzo surface with $\mathrm{Pic}(Y)^G=\mathbb{Z}[-K_Y]$, $\pi$ and $\pi^\prime$ are birational contractions of the curves $\mathcal{L}_3$ and $\mathcal{L}_3^\prime$, respectively, and
$\tau$ is a $G_{36,10}$-birational involution.
\end{lemma}

\begin{proof}
The surface $S$ contains $27$ lines that form five $G_{36,10}$-irreducible curves: $\mathcal{L}_3$, $\mathcal{L}_3^\prime$, and the curves
$$
\aligned
\mathcal{L}_6^\prime
&=\bigcup_{\substack{0\leqslant i,j\leqslant 2\\i\ne j}}
\big\{x_3=0,\ x_0+\zeta_3^ix_1=0,\ x_2+\zeta_3^{-j}x_4=0\big\},\\
\mathcal{L}_6^{\prime\prime}
&=\bigcup_{\substack{0\leqslant i,j\leqslant 2\\i\ne j}}
\big\{x_3=0,\ x_0+\zeta_3^ix_4=0,\ x_1+\zeta_3^jx_2=0\big\},\\
\mathcal{L}_9
&=\bigcup_{0\leqslant i,j\leqslant 2}
\big\{x_3=0,\ x_0+\zeta_3^ix_2=0,\ x_1+\zeta_3^jx_4=0\big\}.
\endaligned
$$
Note that $\mathcal{L}_3\cap\mathcal{L}_6^{\prime\prime}=\varnothing$
and $\mathcal{L}_3^\prime\cap\mathcal{L}_6^\prime=\varnothing$.
Moreover, we have $\mathcal{L}_3+\mathcal{L}_3^\prime\sim -2K_S$, and
$$
\mathcal{L}_3+\mathcal{L}_6^\prime\sim
\mathcal{L}_3^\prime+\mathcal{L}_6^{\prime\prime}\sim
\mathcal{L}_9\sim -3K_S.
$$
Then $\mathcal{L}_3+3\mathcal{L}_3^\prime\sim 2\mathcal{L}_6^\prime$,  $3\mathcal{L}_3+\mathcal{L}_3^\prime\sim 2\mathcal{L}_6^{\prime\prime}$, $\mathcal{L}_6^\prime+\mathcal{L}_6^{\prime\prime}\sim -4K_S$. Therefore, since $\mathrm{Pic}(S)$ is generated by the~classes of lines, $\mathrm{rk}\,\mathrm{Pic}(S)^G=2$, and  $(\mathrm{Pic}(S)\otimes\mathbb{Q})^G$ is spanned by the classes of $\mathcal{L}_3$ and $\mathcal{L}_3^\prime$.

Contracting $\mathcal{L}_3$, we obtain a birational morphism $\pi\colon S\to Y$ such that $Y$ is a smooth del Pezzo surface of degree $6$, and $\mathrm{Pic}(Y)^G=\mathbb{Z}[-K_Y]$. The involution $\iota\in G_{72,40}$ given by
$$
[x_0:x_1:x_2:x_3:x_4]\mapsto[x_0:x_4:x_2:x_3:x_1]
$$
normalizes $G$ and interchanges $\mathcal{L}_3$ and $\mathcal{L}_3^\prime$. Therefore, we may take $\pi^\prime=\pi\circ\iota$ and $\tau=\pi\circ\iota\circ\pi^{-1}$. Then $\pi^\prime$ contracts $\mathcal{L}_3^\prime$, and $\tau$ is a birational involution normalizing the induced $G$-action on $Y$. This gives the required $G$-Sarkisov link \cite{Iskovskikh1996,CheltsovTschinkelZhang2026}.
\end{proof}

Let us describe small $G_{36,9}$-orbits in $X_{a}$.
Let $\Sigma_6$ be the~$G_{36,9}$-orbit of the~point $[0:1:0:0:-1]$.
Then $\Sigma_6$ is a $G_{36,9}$-orbit of length $6$. Observe that
$$
[0:1:0:s:1]\in X_{a}\iff s^3+as+2=0.
$$
For every $s\in\mathbb{C}$ such that $s^3+as+2=0$, let $\Sigma_6^s$ be the~$G_{36,9}$-orbit of the~point $[0:1:0:s:1]$.
Note that the~discriminant of the~polynomial $s^3+as+2$ is $-4(a^3+27)\ne 0$ by our assumption, so there are exactly three possibilities for $\Sigma_6^s$. In each case, $\Sigma_6^s$ is a $G_{36,9}$-orbit of length $6$.

Let $\Sigma_9$ and $\Sigma_9^{\pm}$ be the~$G_{36,9}$-orbits of the points $[1:-1:1:0:-1]$ and $[1:\mp i:-1:0:\pm i]$, respectively.
Then $\Sigma_9$, $\Sigma_9^{+}$ and $\Sigma_9^{-}$ are $G_{36,9}$-orbits of length $9$. Observe that
$$
[1:1:1:t:1]\in X_{a}\iff t^3+2at+4=0.
$$
For every $t\in\mathbb{C}$ such that \mbox{$t^3+2at+4=0$}, let $\Sigma_9^t$ be the~$G_{36,9}$-orbit of the~point $[1:1:1:t:1]$.
Then $\Sigma_9^t$ is a $G_{36,9}$-orbit of length $9$. If $X_a$ is smooth ($a^3\ne-\frac{27}{2}$), the polynomial $t^3+2at+4$ does not have multiple roots, so there are three possibilities for  the~$G_{36,9}$-orbit $\Sigma_9^t$.
If $a^3=-\frac{27}{2}$, then 
$$
\mathrm{Sing}\big(X_a\big)=\mathrm{Orb}_G([a:a:a:-3:a])=\Sigma_9^{t}
$$
for $t=-\frac{3}{a}$. 
For instance, if $a=-\frac{3}{\sqrt[3]{2}}$, then 
$$
t^3+2at+4=(t+2\sqrt[3]{2})(t-\sqrt[3]{2})^2,
$$
and $\Sigma_9^{\sqrt[3]{2}}=\mathrm{Sing}(X_a)$, while $\mathrm{Orb}_G([1:1:1:-2\sqrt[3]{2}:1])\subset X_a\setminus\mathrm{Sing}(X_a)$.

Let $\Sigma_{12}$ be the~$G_{36,9}$-orbit of the~point $[1:-1:0:0:0]$.
Then $\Sigma_{12}$ has length $12$. Set
$$
\mathcal{R}_6=\big(\{x_0=0,x_2=0\}\cap X_{a}\big)\cup\big(\{x_1=0,x_4=0\}\cap X_{a}\big).
$$
Then $\mathcal{R}_6$ is a $G_{36,9}$-irreducible curve of degree $6$ that is a disjoint union of two cubic curves in $X_{a}$,
which are smooth \cite{ArtebaniDolgachev}, and they are contained in the smooth locus of the cubic $X_a$ in the case when $X_a$ is singular. These curves are the~fixed loci in $X_{a}$ of the~transformations $M_1$ and~$M_2$.
Thus, the length of the~$G_{36,9}$-orbit of every point of the curve $\mathcal{R}_6$ is at most $12$.

\begin{lemma}
\label{lemma:36-9-orbits}
Let $\Sigma$ be a $G_{36,9}$-orbit in $X_{a}$ of length $r\leqslant 12$. Then one of the~following cases holds:
\begin{itemize}
\item[$(\mathrm{a})$] $r=6$ and $\Sigma=\Sigma_6$;
\item[$(\mathrm{b})$] $r=6$ and $\Sigma=\Sigma_6^s$ for some $s\in\mathbb{C}$ such that $s^3+as+2=0$;
\item[$(\mathrm{c})$] $r=9$ and $\Sigma$ is one of the $G$-orbits $\Sigma_9$, $\Sigma_9^{+}$, $\Sigma_9^{-}$;
\item[$(\mathrm{d})$] $r=9$ and $\Sigma=\Sigma_9^t$ for some $t\in\mathbb{C}$ such that $t^3+2at+4=0$;
\item[$(\mathrm{e})$] $r=12$ and $\Sigma=\Sigma_{12}$;
\item[$(\mathrm{f})$] $r=12$, and $\Sigma\subset  \mathcal{R}_6$.
\end{itemize}
\end{lemma}

\begin{proof}
Let $P$ be a point in $\Sigma$, and let $G_P$ be its stabilizer in $G_{36,9}$.
Since $G_{36,9}$ fixes no points in $X_{a}$, it follows from Lemma~\ref{lemma:36-9-subgroups} that $G_P$ is isomorphic to $C_3\rtimes\mathfrak{S}_3$, $C_3^2$, $\mathfrak{S}_3$, $C_4$ or $C_3$. Then
$$
r\in\{2,4,6,9,12\}.
$$
In fact, if $G_P\simeq C_3\rtimes\mathfrak{S}_3$ or $G_P\simeq C_3^2$, then $G_P$ does not fix points in $X_{a}$, which is a contradiction. Therefore, we see that $r\ne 2$ and $r\ne 4$. Then one of the following three cases holds:
\begin{itemize}
\item $r=6$ and $G_P\simeq \mathfrak{S}_3$;
\item $r=9$ and $G_P\simeq C_4$;
\item $r=12$ and $G_P\simeq C_3$.
\end{itemize}
Let us deal with these cases one by one.

Suppose that $r=6$ and $G_P\simeq \mathfrak{S}_3$.
By Lemma~\ref{lemma:36-9-subgroups}, the group $G_{36,9}$ contains two subgroups isomorphic to $\mathfrak{S}_3$ up to conjugation, so we may assume that $G_P=\langle M_1,M_3^2\rangle$, or $G_P=\langle M_1M_2,M_3^2\rangle$.
If $G_P=\langle M_1M_2,M_3^2\rangle$, then $G_P$ does not fix points in $X_{a}$.
Thus, we conclude that $G_P=\langle M_1,M_3^2\rangle$. In this case, $G_P$ fixes the~following points:
$$
[0:1:0:0:-1], [0:1:0:s:1],
$$
where $s\in\mathbb{C}$ such that $s^3+as+2=0$. This gives cases $(\mathrm{a})$ or $(\mathrm{b})$.

Now, we suppose that $r=9$ and $G_P\simeq C_4$.
By Lemma~\ref{lemma:36-9-subgroups}, $G_{36,9}$ has a unique subgroup isomorphic to $C_4$ up to conjugation, so we may assume $G_P=\langle M_3\rangle$. Then $G_P$-fixed points are
$$
[1:-1:1:0:-1], [1:\mp i:-1:0:\pm i], [1:1:1:t:1],
$$
where $t\in\mathbb{C}$ such that $t^3+2at+4=0$. This gives cases $(\mathrm{c})$ or $(\mathrm{d})$.

Finally, we suppose that $r=12$ and $G_P\simeq C_3$. Up to conjugation, $G$ contains two subgroups isomorphic to $C_3$.
Therefore, we may assume that $G_P=\langle M_1\rangle$ or $G_P=\langle M_1M_2\rangle$.
If $G_P=\langle M_1\rangle$, then $P\in  \mathcal{R}_6$, which gives case $(\mathrm{f})$.
If $G_P=\langle M_1M_2\rangle$, $G_P$ fixes the~following points:
\begin{align*}
[1:-1:0:0:0], [1:-\zeta_3:0:0:0], [1:-\zeta_3^2:0:0:0], \\
[0:0:1:0:-1], [0:0:1:0:-\zeta_3], [0:0:1:0:-\zeta_3^2].
\end{align*}
All of them are contained in $\Sigma_{12}$. This gives case $(\mathrm{e})$.
\end{proof}

\begin{corollary}[{cf. \cite[Theorem 1.12]{CheltsovWilson2013}}]
\label{corollary:36-9-alpha}
Suppose that $G=G_{36,9}$. Then $\alpha_G(S)=2$.
\end{corollary}

\begin{proof}
Since $|-2K_S|$ contains a $G$-invariant curve, $\alpha_G(S)\leqslant 2$. Since $\mathrm{rk}\,\mathrm{Pic}(S)^G=1$ by Lemma~\ref{lemma:36-9-S-Pic}, and~$|-K_S|$  has no $G$-invariant curves, we get $\alpha_G(S)=2$ by \cite[Lemma 5.1]{Cheltsov2008} and Lemma~\ref{lemma:36-9-orbits}.
\end{proof}

Observe that $\Sigma_6$ and $\Sigma_6^s$ are also $G_{36,10}$-orbits of length $6$, where $s\in\mathbb{C}$ such that $s^3+as+2=0$.
Similarly, both $\Sigma_9$ and $\Sigma_9^t$ are $G_{36,10}$-orbits of length $9$, where $t\in\mathbb{C}$ such that $t^3+2at+4=0$.
Now, let~$\Sigma_{6}^\prime$ and $\Sigma_{6}^{\prime\prime}$ be the~$G_{36,10}$-orbits of the~points $[1:-1:0:0:0]$ and $[0:1:-1:0:0]$, respectively. Then both
$\Sigma_{6}^\prime$ and $\Sigma_{6}^{\prime\prime}$ are $G_{36,10}$-orbits of length $6$, and
$$
\Sigma_{12}=\Sigma_{6}^\prime\cup\Sigma_{6}^{\prime\prime}.
$$
Let $\Sigma_9^\prime$ and $\Sigma_9^{\prime\prime}$ be the~$G_{36,10}$-orbits of $[1:1:-1:0:-1]$ and $[1:-1:-1:0:1]$, respectively. Then $\Sigma_9^\prime$ and $\Sigma_9^{\prime\prime}$ are $G_{36,10}$-orbits of length $9$.

\begin{lemma}
\label{lemma:36-10-orbits}
Let $\Sigma$ be a $G_{36,10}$-orbit in $X_{a}$ of length $r\leqslant 9$, let $P$ be a point in $\Sigma$, and let $G_P$ be its stabilizer in $G_{36,10}$. Then one of the~following cases holds:
\begin{itemize}
\item[$(\mathrm{a})$] $r=6$, $G_P\simeq \mathfrak{S}_3$, and $\Sigma=\Sigma_6$;
\item[$(\mathrm{b})$] $r=6$, $G_P\simeq C_6$, and $\Sigma$ is one of the $G$-orbits $\Sigma_{6}^\prime$ or $\Sigma_6^{\prime\prime}$;
\item[$(\mathrm{c})$] $r=6$, $G_P\simeq\mathfrak{S}_3$, and $\Sigma=\Sigma_6^s$ for some $s\in\mathbb{C}$ such that $s^3+as+2=0$;
\item[$(\mathrm{d})$] $r=9$, $G_P\simeq C_2^2$, and $\Sigma$ is one of the $G$-orbits $\Sigma_{9}$, $\Sigma_9^\prime$, $\Sigma_9^{\prime\prime}$;
\item[$(\mathrm{e})$] $r=9$, $G_P\simeq C_2^2$, and $\Sigma=\Sigma_9^t$ for some $t\in\mathbb{C}$ such that $t^3+2at+4=0$.
\end{itemize}
\end{lemma}

\begin{proof}
Since $G_{36,10}$ does not fix points in $X_{a}$, we see that $r\ne 1$. If $r=2$ or $r=3$, then it follows from Lemma~\ref{lemma:36-10-subgroups} that $G_P$ is isomorphic to one of the following groups $\mathfrak{D}_6$, $C_3\times\mathfrak{S}_3$, $C_3\rtimes\mathfrak{S}_3$. In~each of these cases, the restriction of the representation $\mathbb{V}_4$ to the subgroup $G_P$ does not have one-dimensional subrepresentations, so $G_P$ does not fix points in $X_{a}$, which is a contradiction. Similarly, if $r=4$, then it follows from Lemma~\ref{lemma:36-10-subgroups} that $G_P=\langle M_1,M_2\rangle\simeq C_3^2$, so $G_P$ does not fix points in $X_{a}$ either. Hence, it follows from Lemma~\ref{lemma:36-10-subgroups} that one of the following cases holds:
\begin{itemize}
\item $r=6$ and $G_P\simeq C_6$;
\item $r=6$ and $G_P\simeq\mathfrak{S}_3$;
\item $r=9$ and $G_P\simeq C_2^2$.
\end{itemize}
Let us deal with these cases one by one.

Suppose that $G_P\simeq C_6$. By Lemma~\ref{lemma:36-10-subgroups}, we may assume that $G_P=\langle M_1M_5\rangle$ or $G_P=\langle M_1M_4M_5\rangle$. If $G_P=\langle M_1M_5\rangle$, then all $G_P$-fixed points are contained in $\Sigma_6^\prime$, so $\Sigma=\Sigma_6^\prime$. If $G_P=\langle M_1M_4M_5\rangle$, then all $G_P$-fixed points are contained in $\Sigma_6^{\prime\prime}$, so $\Sigma=\Sigma_6^{\prime\prime}$.

Suppose that $r=6$ and $G_P\simeq\mathfrak{S}_3$.
By Lemma~\ref{lemma:36-10-subgroups}, $G_{36,10}$ has five subgroups isomorphic to $\mathfrak{S}_3$ up to conjugation.
If $G_P$ is conjugated to $\langle M_1,M_4\rangle$, the $G_P$-fixed locus in $X_{a}$ consists of the points
$$
[0:1:0:0:-1], [0:1:0:s:1],
$$
where $s\in\mathbb{C}$ such that $s^3+as+2=0$. Thus, in this case, either $\Sigma=\Sigma_6$ or $\Sigma_6^s$. On the other hand, if $G_P$ is not conjugated to $\langle M_1,M_4\rangle$, the restriction of the representation $\mathbb{V}_4$ to $G_P$ does not have one-dimensional subrepresentations, so $G_P$ does not fix points in $X_{a}$, which is a contradiction.

Finally, we consider the case $r=9$. Then $G_P\simeq C_2^2$.
By Lemma~\ref{lemma:36-10-subgroups}, $G_{36,10}$ has a unique subgroup isomorphic to $C_2^2$ up to conjugation.
Thus, we may assume that $G_P=\langle M_4,M_5\rangle$. Then $G_P$ fixes the~following points in $X_{a}$:
$$
[1:-1:1:0:-1], [1:1:-1:0:-1], [1:-1:-1:0:1],  [1:1:1:t:1],
$$
where $t\in\mathbb{C}$ such that $t^3+2at+4=0$. Then $\Sigma$ is one of the $G$-orbits $\Sigma_{9}$, $\Sigma_9^\prime$, $\Sigma_9^{\prime\prime}$, $\Sigma_9^t$.
\end{proof}

\begin{corollary}
\label{corollary:36-9-36-10-curves-in-S}
Let $Z$ be a $G$-invariant curve in the cubic $X_{a}$ such that $Z\subset S$ and $\mathrm{deg}(Z)<6$. Then $G=G_{36,10}$, and either $Z=\mathcal{L}_3$ or $Z=\mathcal{L}_3^\prime$.
\end{corollary}

\begin{proof}
If $G=G_{36,9}$, then $\mathrm{deg}(Z)\geqslant 6$ by Lemma~\ref{lemma:36-9-S-Pic}, because $|-K_S|$ does not have $G$-invariant curves. Thus, we have $G=G_{36,10}$. Let us show that either $Z=\mathcal{L}_3$ or $Z=\mathcal{L}_3^\prime$. To do this, we may assume that $Z$ is $G$-irreducible. Suppose that $Z\ne\mathcal{L}_3$ and $Z\ne \mathcal{L}_3^\prime$. Then
$$
|Z\cap\mathcal{L}_3|+|Z\cap\mathcal{L}_3^\prime|\leqslant Z\cdot\mathcal{L}_3+Z\cdot\mathcal{L}_3^\prime=Z\cdot(\mathcal{L}_3+\mathcal{L}_3^\prime)=Z\cdot(-2K_S)=2\mathrm{deg}(Z)<12,
$$
because $\mathcal{L}_3+\mathcal{L}_3^\prime\sim -2K_S$ by \eqref{equation:39-9-L6}.
Thus, we see that either $|Z\cap\mathcal{L}_3|<6$ or $|Z\cap\mathcal{L}_3^\prime|<6$ (or both). Without loss of generality, we may assume that $|Z\cap\mathcal{L}_3|<6$. Then $Z\cap\mathcal{L}_3=\varnothing$ by Lemma~\ref{lemma:36-10-orbits}.

Let us use notations of Lemma~\ref{lemma:36-10-S-Pic}. Since $\mathrm{Pic}(Y)^G=\mathbb{Z}[-K_Y]$, there exists  $n\in\mathbb{Z}_{>0}$ such that
$$
\pi(Z)\sim n(-K_Y).
$$
Thus, since $Z$ is disjoint from $\mathcal{L}_3$, we have
$$
Z\sim \pi^*\big(n(-K_Y)\big)\sim n(-K_S)+n\mathcal{L}_3,
$$
which implies that $\mathrm{deg}(Z)=-K_S\cdot Z=6n\geqslant 6$, which contradicts our assumption.
\end{proof}

\begin{corollary}[{cf. \cite[Theorem 1.12]{CheltsovWilson2013}}]
\label{corollary:36-10-alpha}
Suppose that $G=G_{36,10}$. Then $\alpha_G(S)=2$.
\end{corollary}

\begin{proof}
Since $\mathcal{L}_3+\mathcal{L}_3^\prime\sim 2(-K_S)$ by \eqref{equation:39-9-L6}, we get $\alpha_G(S)\leqslant 2$. Suppose $\alpha_G(S)<2$. Then there~exists an~effective $G$-invariant $\mathbb{Q}$-divisor $D$ on the surface $S$ such that $D\sim_{\mathbb{Q}} -K_S$, and  $(S,\lambda D)$ is strictly log canonical for some $\lambda<2$.
The locus $\mathrm{Nklt}(S,\lambda D)$ is $G$-invariant, and it is not empty. 

Suppose that $\mathrm{Nklt}(S,\lambda D)$ contains a curve. Then $\mathrm{Nklt}(S,\lambda D)$ contains a $G$-irreducible curve $C$ such that $\lambda D=C+D^\prime$, where $D^\prime$ is an effective $\mathbb{Q}$-divisor on $S$ whose support does not contain the curve $C$. Then
$$
6>3\lambda=-K_S\cdot (\lambda D)=-K_S\cdot (C+D^\prime)=\mathrm{deg}(C)-K_S\cdot D^\prime\geqslant \mathrm{deg}(C),
$$
so either $C=\mathcal{L}_3$ or $C=\mathcal{L}_3^\prime$ by Corollary~\ref{corollary:36-9-36-10-curves-in-S}. Write
$$
D=a\mathcal{L}_3+b\mathcal{L}_3^\prime+\Delta_S,
$$
where $\Delta_S$ is an effective $\mathbb{Q}$-divisor on $S$ whose support does not contain the curve $\mathcal{L}_3$ and  $\mathcal{L}_3^\prime$, and $a$ and $b$ are non-negative rational numbers such that at least one of them is equal to $\frac{1}{\lambda}>\frac{1}{2}$. Then
\begin{align*}
3=\mathcal{L}_3\cdot D&=\mathcal{L}_3\cdot \big(a\mathcal{L}_3+b\mathcal{L}_3^\prime+\Delta_S)=9b-3a+\mathcal{L}_3\cdot \Delta_S\geqslant 9b-3a,\\
3=\mathcal{L}_3^\prime\cdot D&=\mathcal{L}_3^\prime\cdot \big(a\mathcal{L}_3+b\mathcal{L}_3^\prime+\Delta_S)=9a-3b+\mathcal{L}_3^\prime\cdot \Delta_S\geqslant 9a-3b.
\end{align*}
These inequalities imply that $a\leqslant\frac{1}{2}$ and $b\leqslant\frac{1}{2}$, which is a contradiction.

Hence, $\mathrm{Nklt}(S,\lambda D)$ is a finite $G$-invariant subset. Set $\Sigma=\mathrm{Nklt}(S,\lambda D)$, and let $\mathcal{I}_\Sigma$ be the ideal sheaf of the subset $\Sigma$. Then it follows from the Nadel vanishing \cite[Theorem 9.4.8]{Lazarsfeld2004} that
$$
h^1\big(S,\mathcal{I}_\Sigma\otimes\mathcal{O}_{S}(-K_S)\big)=0.
$$
Hence, we have the~following exact sequence of groups
$$
0\longrightarrow H^0\big(S,\mathcal{I}_\Sigma\otimes\mathcal{O}_{S}(-K_S)\big)\longrightarrow H^0\big(S,\mathcal{O}_{S}(-K_S)\big)\longrightarrow H^0\big(\Sigma,\mathcal{O}_{\Sigma}\big)\longrightarrow 0,
$$
which gives
$$
|\Sigma|=h^0\big(\Sigma,\mathcal{O}_{\Sigma}\big)=h^0\big(S,\mathcal{O}_{S}(-K_S)\big)-h^0\big(S,\mathcal{I}_\Sigma\otimes\mathcal{O}_{S}(-K_S)\big)=4-h^0\big(S,\mathcal{I}_\Sigma\otimes\mathcal{O}_{S}(-K_S)\big)\leqslant 4,
$$
which is impossible, because $S$ does not contain $G$-orbits of length $\leqslant 4$ by Lemma~\ref{lemma:36-10-orbits}.
\end{proof}

Observe that $|2H|$ contains a pencil consisting of $G_{72,40}$-invariant surfaces. This can be checked using the~following GAP code:
\begin{verbatim}
    G:=SmallGroup(72,40);
    T:=CharacterTable(G);
    Ir:=Irr(T);
    U:=Ir[1]+Ir[6];
    S:=SymmetricParts(T,[U],2);
    Print(MatScalarProducts(Ir,S));
\end{verbatim}
We denote this pencil by $\mathcal{P}$. Then $\mathcal{P}$ is free from fixed components, and $\mathcal{P}$ consists of surfaces
$$
S_\lambda=\big\{x_0x_2+x_1x_4+\lambda x_3^2=0\big\}\cap X_{a},
$$
where $\lambda\in\mathbb{C}\cup\{\infty\}$. Note also that $S_{\infty}=2S$. We set
$$
\mathcal{C}_6=\{x_0x_2+x_1x_4=0\}\cap S=\{x_3=0, x_0x_2+x_1x_4=0\}\cap X_{a}.
$$
Then $\mathcal{C}_6$ is a $G_{72,40}$-invariant smooth curve of genus $4$, so  $\mathcal{C}_6$ is $G_{36,9}$-invariant and $G_{36,10}$-invariant. The~smoothness of this curve can be checked using the following Magma code:
\begin{verbatim}
    Q:=RationalField();
    P4<x0,x1,x2,x3,x4>:=ProjectiveSpace(Q,4);
    C6:=Scheme(P4,[x3,x0*x2+x1*x4,x0^3+x1^3+x2^3+x3^3+x4^3]);
    IsNonsingular(C6);
\end{verbatim}
Moreover, since $\mathcal{C}_6$ is the~base locus of the~pencil $\mathcal{P}$ (taken with reduced structure), a general surface in the pencil $\mathcal{P}$ is a smooth K3 surface. Furthermore, using Lemma~\ref{lemma:36-9-orbits}, we obtain

\begin{corollary}
\label{corollary:36-9-Du-Val}
For every $\lambda\in\mathbb{C}$, the surface $S_\lambda$ has at most isolated ordinary double points (nodes).
Moreover, if the surface $S_\lambda$ is singular, then one of the following two cases holds:
\begin{itemize}
\item $\mathrm{Sing}(S_\lambda)=\Sigma_9^t$ for $t\in\mathbb{C}$ such that $t^3+2at+4=0$, and $\lambda=-\frac{2}{t^2}$;
\item $\mathrm{Sing}(S_\lambda)=\Sigma_6^s$ for $s\in\mathbb{C}$ such that $s^3+as+2=0$, and $\lambda=-\frac{1}{s^2}$.
\end{itemize}
\end{corollary}

\begin{proof}
Since $\mathcal{C}_6=S\cap S_\lambda$ and the curve $\mathcal{C}_6$ is smooth,
$S_\lambda$ is smooth at every point of this intersection, which implies, in particular, that $S_\lambda$ has isolated singularities, so $S_\lambda$ is normal.

Let $P$ be a singular point of $S_\lambda$ (if any). Then its $G_{36,9}$-orbit has at least $6$ points by Lemma~\ref{lemma:36-9-orbits}.
But $S_{\lambda}$ can have at most two non-Du Val singular points \cite{Umezu}, see also \cite[Theorem 6.9]{Shokurov92} and \cite{Fujino}.
Thus, we conclude that $S_\lambda$ has at most Du Val singularities.

Suppose that $S_\lambda$ is singular.
Let $\widetilde{S}_\lambda$ be its minimal resolution. Then $\widetilde{S}_\lambda$ is a smooth K3 surface, so $\mathrm{rk}\,\mathrm{Pic}\big(\widetilde{S}_\lambda\big)\leqslant 20$, and $|\mathrm{Sing}(S_\lambda)|\leqslant 16$ by \cite{Nikulin}. By Lemma~\ref{lemma:36-9-orbits}, one of the~following holds:
\begin{enumerate}
\item $\mathrm{Sing}(S_\lambda)$ is a $G_{36,9}$-orbit of length $12$ that is contained in $\mathcal{R}_6$.
\item $\mathrm{Sing}(S_\lambda)=\Sigma_9^t$ for $t\in\mathbb{C}$ such that  $t^3+2at+4=0$, and $\lambda=-\frac{2}{t^2}$;
\item $\mathrm{Sing}(S_\lambda)=\Sigma_6^s$ for $s\in\mathbb{C}$ such that $s^3+as+2=0$, and $\lambda=-\frac{1}{s^2}$.
\end{enumerate}
However, the first case is impossible, since otherwise we would have
$$
12=\mathcal{R}_6\cdot S_\lambda\geqslant 2\big|\mathrm{Sing}(S_\lambda)\big|\geqslant 24,
$$
because $ \mathcal{R}_6\not\subset S_\lambda$. If $\lambda=-\frac{2}{t^2}$ for $t\in\mathbb{C}$ such that $t^3+2at+4=0$, then $S_\lambda$ is singular at $\Sigma_9^t$, and its singularities are nodes (explicit computations).
If  $\lambda=-\frac{1}{s^2}$ for $s\in\mathbb{C}$ such that $s^3+as+2=0$, then $S_\lambda$ is singular at $\Sigma_6^s$, and its singularities are nodes.
\end{proof}

Recall that $\mathcal{C}_6$, $\mathcal{L}_6$, $\mathcal{R}_6$ are $G_{36,9}$-irreducible curves in $X_{a}$.
Now, let us present $G_{36,9}$-irreducible curves of degree $9$ in $X_{a}$, whose irreducible components are lines, which are not contained in $S$. For~every $t\in\mathbb{C}$ such that $t^3+2at+4=0$, let $L_t^{\pm}$ be the line in $\mathbb{P}^4$ that is given by
$$
\left\{\aligned
&\pm itx_2-(1\pm i)x_3+tx_4=0,\\
&\mp itx_1\mp itx_4\pm 2ix_3=0,\\
&\mp itx_0-tx_1+(1\pm i)x_3=0.
\endaligned
\right.
$$
Then $L_t^{\pm}\subset X_{a}$, and $L_t^{\pm}$ contains the~points $[1:\mp i:-1:0:\pm i]\in \Sigma_9^{\pm}$ and $[1:1:1:t:1]\in \Sigma_9^t$.
Let~$\mathcal{L}_t^{\pm}$ be the $G_{36,9}$-irreducible curve whose irreducible component is $L_t^\pm$.
Then it consists of $9$ lines, and these lines are disjoint.
Moreover, $\mathcal{L}_t^{\pm}$ contains $\Sigma_9^{\pm}$ and $\Sigma_9^t$, and
$\mathrm{Sing}(S_\lambda)=\Sigma_9^t\subset \mathcal{L}_t^{\pm}\subset S_\lambda$,
where $\lambda=-\frac{2}{t^2}$. Furthermore, we have
\begin{equation}
\label{equation:36-9-nine-lines}
\mathcal{L}_t^{+}+\mathcal{L}_t^{-}=\big\{6x_3(x_0x_2-x_1x_4)+t(x_0^3-x_1^3+x_2^3-x_4^3)=0\big\}\cap S_{\lambda}.
\end{equation}
Recall that $\mathcal{C}_6$ and $\mathcal{L}_6$ are contained in $S$. But $\mathcal{R}_6$ is not contained in any surface in  $\mathcal{P}$.

\begin{lemma}
\label{lemma:36-9-reducible-curves}
Let $Z$ be a reducible $G_{36,9}$-irreducible curve in $X_{a}$ of degree $d<12$ such that $Z\not\subset S$.
Then one of the following two cases holds:
\begin{itemize}
\item $d=6$ and $Z=\mathcal{R}_6$,
\item $d=9$ and either $Z=\mathcal{L}_t^{+}$ or $Z=\mathcal{L}_t^{-}$ for some $t\in\mathbb{C}$ such that $t^3+2at+4=0$.
\end{itemize}
\end{lemma}

\begin{proof}
The intersection $S\cap Z$ consists of $d<12$ points, and it is $G_{36,9}$-invariant. Hence, it follows from
Lemma~\ref{lemma:36-9-orbits} that $d=6$ or $9$.

Let $C$ be an irreducible component of $Z$, and let $G_C$ be its
stabilizer in $G_{36,9}$. By Corollary~\ref{corollary:36-9-act},
the curve $Z$ cannot have $3$ irreducible components. Thus, it follows
from Lemma~\ref{lemma:36-9-subgroups} that one of the following cases holds:
\begin{itemize}
\item $Z$ is a union of $6$ lines, and $G_C\simeq\mathfrak{S}_3$;
\item $Z$ is a union of $2$ irreducible curves of degree $3$, and
$G_C\simeq C_3\rtimes\mathfrak{S}_3$;
\item $Z$ is a union of $9$ lines, and $G_C\simeq C_4$.
\end{itemize}
Let us deal with these cases one by one.

Suppose that $Z$ is a union of $6$ lines. Up to conjugation,  $G_C=\langle M_1,M_3^2\rangle$ or $G_C=\langle M_1M_2,M_3^2\rangle$. Since $G_C$ fixes the point $C\cap S$, the second possibility is excluded by Lemma~\ref{lemma:36-9-orbits}, so $G_C=\langle M_1,M_3^2\rangle$. On the other hand, a finite group acting faithfully on $\mathbb{P}^1$ and fixing a point is cyclic, which implies that the image of $G_C$ in $\mathrm{Aut}(C)$ is cyclic, so its commutator subgroup $\langle M_1\rangle$ fixes $C$ pointwise. Therefore, $C$ is contained in the curve $\{x_0=x_2=0\}\cap X_a$. However, this plane cubic is smooth, since $a^3\ne-27$, so it cannot contain a line. This excludes the first case.

Now, suppose that $Z$ is a union of $2$ irreducible curves of degree $3$.
Then $G_C=\langle M_1,M_2,M_3^2\rangle$.
The only $G_C$-invariant hyperplane in $\mathbb{P}^4$ is $\{x_3=0\}$.
Thus, $C$ cannot be a twisted cubic, since its linear span would be
this hyperplane, contrary to $C\not\subset S$.
Hence, $C$ is a plane cubic.
But the only $G_C$-invariant planes in $\mathbb{P}^4$ are
$\{x_0=x_2=0\}$ and $\{x_1=x_4=0\}$.
Therefore, $Z=\mathcal{R}_6$.

Finally, we suppose that the curve $Z$ is a union of $9$ lines. Since $G_{36,9}$ has a unique subgroup isomorphic to $C_4$ (up to conjugation), we may assume that $G_C=\langle M_3\rangle$. Let $P=C\cap S$. Then it follows from the proof of Lemma~\ref{lemma:36-9-orbits}, that $P$ is one of the following points:
$$
[1:-1:1:0:-1],
[1:-i:-1:0:i],
[1:i:-1:0:-i].
$$
Moreover, the cyclic group $G_C$ has a fixed point in $C\setminus S$.
The same computation shows that this point is
$Q=[1:1:1:t:1]$, where $t^3+2at+4=0$.

If $P=[1:-1:1:0:-1]$, then the line passing through the points $P$ and $Q$ is given parametrically as 
$[u+v:u-v:u+v:tu:u-v]$, where $[u:v]\in\mathbb{P}^1$.
Substituting this into the~defining equation of $X_a$, we obtain
$$
u^3(t^3+2at+4)+2uv^2(at+6)=0.
$$
Since $C\subset X_a$, this gives $at=-6$. Together with
$t^3+2at+4=0$, it implies $t^3=8$ and $a^3=-27$, which contradicts
our assumption. Therefore, we have $P=[1:\mp i:-1:0:\pm i]$, so $C=L_t^\pm$.
Hence, $Z=\mathcal{L}_t^+$ or $Z=\mathcal{L}_t^-$, as required.
\end{proof}

Recall that $\mathcal{L}_3$, $\mathcal{L}_3^\prime$, $\mathcal{C}_6$, $\mathcal{R}_6$ are $G_{36,10}$-irreducible curves. Set
\begin{align*}
\mathcal{R}_6^\prime&=\big(\{x_0=0,x_1=0\}\cup\{x_2=0,x_4=0\}\big)\cap X_{a},\\
\mathcal{R}_6^{\prime\prime}&=\big(\{x_0=0,x_4=0\}\cup\{x_1=0,x_2=0\}\big)\cap X_{a}.
\end{align*}
Then $\mathcal{R}_6^\prime$ and $\mathcal{R}_6^{\prime\prime}$ are $G_{36,10}$-irreducible curves of degree $6$, and each of them is a union of two disjoint smooth plane cubic curves. Note that $\mathcal{R}_6^\prime$ and $\mathcal{R}_6^{\prime\prime}$ are contained in $S_0$. Similarly, set
\begin{align*}
\mathcal{R}_9&=\big(\{x_0=x_1,x_2=x_4\}\cup\{x_0=\zeta_3 x_1,x_2=\zeta_3^2 x_4\}\cup\{x_0=\zeta_3^2 x_1,x_2=\zeta_3 x_4\}\big)\cap X_{a},\\
\mathcal{R}_9^{\prime}&=\big(\{x_0=x_4,x_1=x_2\}\cup\{x_0=\zeta_3 x_4,x_1=\zeta_3 x_2\}\cup\{x_0=\zeta_3^2 x_4,x_1=\zeta_3^2 x_2\}\big)\cap X_{a}.
\end{align*}
Then $\mathcal{R}_9$ and $\mathcal{R}_9^{\prime}$ are $G_{36,10}$-irreducible curves. If $X_a$ is smooth ($a^3\ne-\frac{27}{2}$), then $\mathcal{R}_9$ and $\mathcal{R}_9^{\prime}$ are unions of $3$ disjoint smooth plane cubic curves. On the other hand, if $X_a$ is singular ($a^3=-\frac{27}{2}$), then $\mathcal{R}_9$ and $\mathcal{R}_9^{\prime}$ are unions of $9$ lines, which are unions of three  disjoint reducible cubic curves consisting of three lines. For instance, if $a=-\frac{3}{\sqrt[3]{2}}$, an irreducible component of  $\mathcal{R}_9$ is the curve
$$
\big\{2x_0+2x_2+\sqrt[3]{4}x_3=0, 2x_1+2x_4+\sqrt[3]{4}x_3=0, x_2=x_4\big\}\subset\Pi_{0,0}.
$$
and an irreducible component of  $\mathcal{R}_9^{\prime}$ is the curve
$$
\big\{2x_0+2x_2+\sqrt[3]{4}x_3=0, 2x_1+2x_4+\sqrt[3]{4}x_3=0, x_0=x_4\big\}\subset\Pi_{0,0}.
$$
Moreover, if the cubic threefold $X_a$ is singular, then the curves $\mathcal{R}_9$ and $\mathcal{R}_9^\prime$ contain the singular locus of the cubic $X_a$, and $\mathrm{Sing}(\mathcal{R}_9)=\mathrm{Sing}(\mathcal{R}_9^{\prime})=\mathrm{Sing}(X_a)$. Note that $\mathcal{R}_9$ and $\mathcal{R}_9^{\prime}$ are not contained in any surface in $\mathcal{P}$. Finally, let $\mathcal{K}_6$ be the union of the following three curves:
\begin{align*}
\big\{x_0+x_1=0, x_2+x_4&=0, 2ax_1x_4+x_3^2=0\big\},\\
\big\{x_0+\zeta_3 x_1=0, x_2+\zeta_3^2x_4&=0, 2ax_1x_4+x_3^2=0\big\},\\
\big\{x_0+\zeta_3^2x_1=0, x_2+\zeta_3x_4&=0, 2ax_1x_4+x_3^2=0\big\}.
\end{align*}
and let $\mathcal{K}_6^{\prime}$ be the union of the following three curves:
\begin{align*}
\big\{x_0+x_4=0, x_1+x_2&=0, 2ax_2x_4-x_3^2=0\big\},\\
\big\{x_0+\zeta_3x_4=0, x_1+\zeta_3x_2&=0, 2ax_2x_4+(\zeta_3+1)x_3^2=0\big\},\\
\big\{x_0+\zeta_3^2x_4=0, x_1+\zeta_3^2x_2&=0, 2ax_2x_4+(\zeta_3^2+1)x_3^2=0\big\}.
\end{align*}
If $a\ne 0$, then $\mathcal{K}_6$ and $\mathcal{K}_6^\prime$ are
$G_{36,10}$-irreducible curves of degree $6$ whose irreducible components are disjoint (smooth) conics.
In this case, we have
\begin{equation}
\label{equation:36-10-six-conics}
\mathcal{K}_6+\mathcal{K}_6^\prime=\big\{x_0x_2-x_1x_4=0\big\}\cap S_{\lambda},
\end{equation}
where $\lambda=\frac{1}{a}$, and $S_\lambda$ is smooth by Corollary~\ref{corollary:36-9-Du-Val}.
If $a=0$, then $\mathcal{K}_6=2\mathcal{L}_3$ and $\mathcal{K}_6^\prime=2\mathcal{L}_3^\prime$.

\begin{lemma}
\label{lemma:36-10-reducible-curves}
Let $Z$ be a reducible $G_{36,10}$-irreducible curve in $X_{a}$ of degree $d<12$ such that $Z\not\subset S$.
Then one of the following three cases holds:
\begin{itemize}
\item $d=9$, and $Z$ is one of the curves $\mathcal{R}_9$ or $\mathcal{R}_9^\prime$;
\item $d=6$, and $Z$ is one of the curves $\mathcal{R}_6$, $\mathcal{R}_6^\prime$ or $\mathcal{R}_6^{\prime\prime}$;
\item $d=6$, $a\ne 0$, and $Z$ is one of the curves $\mathcal{K}_6$ or $\mathcal{K}_6^{\prime}$.
\end{itemize}
\end{lemma}

\begin{proof}
The intersection $S\cap Z$ consists of $d<12$ points (counted with multiplicities), and this intersection is a union of $G_{36,10}$-orbits. Hence, it follows from Lemma~\ref{lemma:36-10-orbits} that $d=6$ or $9$.

Let $C$ be an irreducible component of the~curve $Z$, and let $G_C$ be its stabilizer in the group~$G_{36,10}$.
Then one of the~following holds:
\begin{itemize}
\item $Z$ is a union of $2$ plane cubic curves;
\item $Z$ is a union of $2$ twisted cubic curves;
\item $Z$ is a union of $3$ conics;
\item $Z$ is a union of $6$ lines;
\item $Z$ is a union of $3$ plane cubic curves;
\item $Z$ is a union of $3$ twisted cubic curves;
\item $Z$ is a union of $9$ lines.
\end{itemize}
Let us deal with these cases one by one.

Suppose that $Z$ is a union of $2$ irreducible curves of degree $3$.
Then it follows from Lemma~\ref{lemma:36-10-subgroups} that $G_C$ is a normal subgroup of $G$, and either $G_C\simeq C_3\times\mathfrak{S}_3$ or $G_C\simeq C_3\rtimes\mathfrak{S}_3$. In each case, the~restriction of the representation $\mathbb{V}_4$ to $G_C$ splits as a sum of two non-isomorphic two-dimensional irreducible representations of $G_C$, which can be verified by the following Magma code:
\begin{verbatim}
    G:=SmallGroup(36,10);
    K:=CyclotomicField(36);
    L:=IrreducibleModules(G,K);
    R:=L[9];
    S:=Subgroups(G);
    H:=S[21]`subgroup;
    print(IdentifyGroup(H));
    W:=Restriction(R,H);
    print(W);
    ID:=Decomposition(W);
    print(ID);
    IsIsomorphic(ID[1],ID[2]);
\end{verbatim}
Hence, $C$ is not a twisted cubic curve, because otherwise $C$ would be contained in $S$, which is not the case by assumption. Hence, $C$ is a plane cubic. If $G_C\simeq C_3\rtimes\mathfrak{S}_3$, then $G_C=\langle M_1,M_2,M_4\rangle$, and it follows from the proof of Lemma~\ref{lemma:36-9-reducible-curves} that we have $Z=\mathcal{R}_6$. Similarly, if $G_C\simeq C_3\times\mathfrak{S}_3$, then either $G_C=\langle M_1,M_2,M_5\rangle$ or $G_C=\langle M_1,M_2,M_4M_5\rangle$, so $C$ is one of the following curves:
\begin{align*}
\{x_0=0,x_1=0\}\cap X_{a},\\
\{x_2=0,x_4=0\}\cap X_{a},\\
\{x_0=0,x_4=0\}\cap X_{a},\\
\{x_1=0,x_2=0\}\cap X_{a}.
\end{align*}
Hence, in this case, either $Z=\mathcal{R}_6^\prime$ or $Z=\mathcal{R}_6^{\prime\prime}$.

Suppose that $Z$ is a union of $3$ irreducible conics. Then it follows from Lemma~\ref{lemma:36-10-subgroups} that $G_C\simeq\mathfrak{D}_6$. Up to conjugation, $G$ contains two subgroups isomorphic to $\mathfrak{D}_6$. In each case, the restriction of the representation $\mathbb{V}_4$ to $G_C$ splits as a sum of two non-isomorphic two-dimensional irreducible representations of the group $G_C$. Let $\Pi_C$ be the plane in $\mathbb{P}^4$ spanned by $C$. Then we have exactly two choices for the plane $\Pi_C$, and the plane $\Pi_C$ contains the $G_{36,10}$-fixed point $[0:0:0:1:0]\not\in X_a$, which implies that $\Pi_C\not\subset X_a$. Then $X_{a}\cap\Pi_C=C+L$,
where $L$ is a line in $X_{a}$. Hence, it follows from Lemma~\ref{lemma:36-10-S-Pic} that $L$ is a component of $\mathcal{L}_3$ or $\mathcal{L}_3^\prime$, so $\Pi_C$ is one of the planes
\begin{align*}
\{x_0+x_1=0, x_2+x_4=0\},\\
\{x_0+\zeta_3 x_1=0, x_2+\zeta_3^2x_4=0\},\\
\{x_0+\zeta_3^2x_1=0, x_2+\zeta_3x_4=0\},\\
\{x_0+x_4=0, x_1+x_2=0\},\\
\{x_0+\zeta_3x_4=0, x_1+\zeta_3x_2=0\},\\
\{x_0+\zeta_3^2x_4=0, x_1+\zeta_3^2x_2=0\}.
\end{align*}
Intersecting these planes with $X_{a}$ and using the fact that $C$ is a smooth conic, we see that $a\ne 0$, and either $Z=\mathcal{K}_6$ or $Z=\mathcal{K}_6^\prime$.

Suppose that $Z$ is a union of $6$ lines. Then $C\cap S$ is a $G_C$-fixed point, which implies that $G_C$ fixes a point $O\in C$ such that $O$ is not contained in $S$, since otherwise we would have $Z\subset S$, which contradicts our assumption. On the other hand, it follows from Lemma~\ref{lemma:36-10-subgroups} that 
\begin{itemize}
\item either $G_C\simeq C_6$, 
\item or $G_C\simeq\mathfrak{S}_3$. 
\end{itemize}
If $G_C\simeq C_6$, $S$ contains all $G_C$-fixed points  by Lemma~\ref{lemma:36-10-orbits}, which is a contradiction, since $O\not\in S$. Thus, $G_C\simeq \mathfrak{S}_3$. By Lemma~\ref{lemma:36-10-subgroups}, $G_{36,10}$ has five subgroups isomorphic to $\mathfrak{S}_3$ up to conjugation. But only one of them fixes a point in $X_{a}$. Hence, we may assume that $G_C=\langle M_1,M_4\rangle$. Then 
$$
C\cap S=[0:1:0:0:-1],
$$ 
and $O=[0:1:0:s:1]$, where $s\in\mathbb{C}$ such that $s^3+as+2=0$. But the line in $\mathbb{P}^4$ that passes through these points is not contained in $X_{a}$. Therefore, $Z$ is not a union of $6$ lines.

Now, we suppose that $Z$ is a union of $3$ irreducible cubic curves. Then $G_C\simeq\mathfrak{D}_6$ by Lemma~\ref{lemma:36-10-subgroups}. As we already mentioned above, the restriction of the representation $\mathbb{V}_4$ to $G_C$ splits as a sum of two non-isomorphic two-dimensional irreducible representations. Hence, $C$ is not a twisted cubic curve, because $C\not\subset S$. Thus, $C$ is a plane cubic. As above, we denote by $\Pi_C$ the plane in $\mathbb{P}^4$ that contains $C$. Then $\Pi_C$ is $G_C$-invariant. If $X_a$ is singular, $\Pi_C\not\subset X_a$, because
$$
\Pi_C\ne\Pi_{i,j}
$$
for any $i,j\in\{0,1,2\}$, since $G$ acts transitively on the planes $\Pi_{i,j}$. So, since $G$ contains $6$ subgroups isomorphic to $\mathfrak{D}_6$ by Lemma~\ref{lemma:36-10-subgroups}, we have exactly $12$ possibilities for the plane $\Pi_C$, and $6$ of them are already listed above --- they intersect $X_{a}$ by a reducible cubic consisting of a line and a conic. The remaining $6$ planes are
\begin{align*}
\{x_0=x_1&,  x_2=x_4\},\\
\{x_0=\zeta_3 x_1&,   x_2=\zeta_3^2x_4\},\\
\{x_0=\zeta_3^2x_1&,   x_2=\zeta_3x_4\},\\
\{x_0=x_4&,   x_1=x_2\},\\
\{x_0=\zeta_3x_4&,  x_1=\zeta_3x_2\},\\
\{x_0=\zeta_3^2x_4&,  x_1=\zeta_3^2x_2\}.
\end{align*}
They intersect $X_{a}$ in irreducible components of $\mathcal{R}_9+\mathcal{R}_9^\prime$, so either $Z=\mathcal{R}_9$ or $Z=\mathcal{R}_9^\prime$, which implies that $X_a$ is smooth in this case, since otherwise $\mathcal{R}_9$ and $\mathcal{R}_9^\prime$ would be consisting of $9$ lines.

Therefore, we may assume that $Z$ is a union of $9$ lines, so we have $G_C\simeq C_2\times C_2$ by Lemma~\ref{lemma:36-10-subgroups}, and~$G$ has one subgroup isomorphic to $C_2^2$ up to conjugation. We may assume that $G_C=\langle M_4,M_5\rangle$. Note that $C\cap S$ is a $G_C$-fixed point. Thus, this point is one of the~points
$$
[1:-1:1:0:-1], [1:1:-1:0:-1], [1:-1:-1:0:1]
$$
But $G_C$ fixes at least two points in $C$. Thus, by Lemma~\ref{lemma:36-10-orbits}, $[1:1:1:t:1]\in C$ for some $t\in\mathbb{C}$ such that $t^3+2at+4=0$.
If $a^3\ne -\frac{27}{2}$, then $X_a$ does not contain the line in $\mathbb{P}^4$ passing through the point $[1:1:1:t:1]$ and any of the points 
\begin{center}
$[1:-1:1:0:-1]$, $[1:1:-1:0:-1]$, $[1:-1:-1:0:1]$.
\end{center}
Thus, $a^3=-\frac{27}{2}$, so $X_a$ is singular. Then 
\begin{itemize}
\item either $t=-\frac{3}{a}$ and $[1:1:1:t:1]$ is a singular point of $X_a$, 
\item or $t=\frac{6}{a}$ and $[1:1:1:t:1]$ is a smooth point of $X_a$. 
\end{itemize}
This implies that either $C$ is the line passing through $[a:a:a:6:a]$ and $[1:1:-1:0:-1]$, or~$C$~is~the~line passing through  $[a:a:a:6:a]$ and $[1:-1:-1:0:1]$, so $Z=\mathcal{R}_9$ or $Z=\mathcal{R}_9^\prime$. 
\end{proof}

Now, we are ready to describe irreducible $G$-invariant curves of degree $<12$ in the cubic~$X_{a}$, which are not contained in the cubic surface $S$. If $G=G_{36,10}$, such curves do not exist:

\begin{lemma}
\label{lemma:36-9-irreducible-curves}
Let $C\subset X_{a}$ be an irreducible $G$-invariant curve of degree  $d<12$ such that $C\not\subset S$. Then $G=G_{36,9}$, $d=9$, the~curve $C$ is smooth, and its genus is either $1$ or $4$. Moreover, we have
$$
\mathrm{Sing}(S_\lambda)=\Sigma_{9}^t\subset C\subset S_\lambda
$$
where $t\in\mathbb{C}$ such that $t^3+2at+4=0$, and $\lambda=-\frac{2}{t^2}$.
\end{lemma}

\begin{proof}
Since $C$ is not contained in $S$, it is not contained in any hyperplane in $\mathbb{P}^4$,
which implies that $G$ acts faithfully on $C$. Let $\widehat{C}\to C$ be the~normalization, and let $g$ be the~genus of $\widehat{C}$. The~$G$-action lifts to $\widehat{C}$. By Lemma~\ref{lemma:36-9-orbits}, if $G=G_{36,9}$, the length of a $G$-orbit in $\widehat{C}$ is $9$, $12$, $18$~or~$36$. Similarly, by Lemma~\ref{lemma:36-10-orbits}, if $G=G_{36,10}$, the length of a $G$-orbit in $\widehat{C}$ is $6$, $12$, $18$ or $36$.

As in the~proof of Lemmas~\ref{lemma:36-9-reducible-curves} and \ref{lemma:36-10-reducible-curves}, we see that $S\cap C$ consists of $d<12$ points (counted with multiplicities). On the~other hand, the intersection $S\cap C$ is a union of $G$-orbits.
Hence, it follows from Lemmas~\ref{lemma:36-9-orbits} and \ref{lemma:36-10-orbits} that $d=6$ or $9$, and $S\cap C$ is a $G$-orbit of length $6$ or $9$, respectively. Moreover,  $S$ intersects $C$ transversally at each point of $S\cap C$, which implies that $C$ is smooth at every point of the~intersection $S\cap C$. In particular, if $d=6$, then $G=G_{36,10}$, because $\widehat{C}$ does not contain $G$-orbits of length $6$ in the case when $G=G_{36,9}$.
Similarly, if $d=9$, then $G=G_{36,9}$, because $\widehat{C}$ does not contain $G$-orbits of length $9$ in the case when $G=G_{36,10}$.

Let $P$ be a general point in $C$, and let $\Sigma$ be the~$G$-orbit of $P$. Then $|\Sigma|=36$, and there is $\lambda\in\mathbb{C}$ such that $P\in S_\lambda$. Moreover, $C\subset S_\lambda$, since otherwise we would have
$$
18\geqslant 2d=S_\lambda\cdot C\geqslant |S_\lambda\cap C|\geqslant |\Sigma|=36.
$$
By Corollary~\ref{corollary:36-9-Du-Val}, either the surface $S_\lambda$ is smooth, or  $\mathrm{Sing}(S_\lambda)$ is a $G$-orbit of length $6$ or $9$, and the singularities of $S_\lambda$ are nodes. Set $H_{S_{\lambda}}=H\vert_{S_{\lambda}}$. Then it follows from the Hodge index theorem that
$$
6C^2-d^2=\mathrm{det}\begin{pmatrix}
C^2 &d \\
d & 6
\end{pmatrix}=\mathrm{det}\begin{pmatrix}
C^2 &H_{S_{\lambda}}\cdot C \\
H_{S_{\lambda}}\cdot C & H_{S_{\lambda}}^2
\end{pmatrix}\leqslant 0.
$$
This gives $C^2\leqslant\frac{d^2}{6}$. Thus, if $d=9$, then $C^2\leqslant\frac{27}{2}$.
Similarly, if $d=6$, then $C^2\leqslant 6$.

Suppose that $C$ is contained in the~smooth locus of the surface $S_\lambda$.
Then, by adjunction formula, we have $C^2=2\mathrm{p}_a(C)-2$, where $\mathrm{p}_a(C)$ is the~arithmetic genus of the~curve $C$.
Thus, we have
$$
2g-2\leqslant 2\mathrm{p}_a(C)-2=C^2\leqslant
\left\{\aligned
&12\ \text{if $d=9$},\\
&6\ \text{if $d=6$}.
\endaligned
\right.
$$
Therefore, if $d=9$, then $g\leqslant \mathrm{p}_a(C)\leqslant 7$, so $g\in\{1,4\}$ by Lemma~\ref{lemma:36-9-curves}. If $d=6$, then $g\leqslant \mathrm{p}_a(C)\leqslant 4$, so it follows from Lemma~\ref{lemma:36-10-curves} that $g=\mathrm{p}_a(C)=4$. On the other hand, if the curve $C$ is singular, then $\mathrm{Sing}(C)$ is a union of $G$-orbits $\Sigma_1,\ldots\Sigma_r$, so
$$
g+\sum_{i=1}^rn_i|\Sigma_i|=\mathrm{p}_a(C)\leqslant\left\{\aligned
&7\ \text{if $d=9$},\\
&4\ \text{if $d=6$}.
\endaligned
\right.
$$
for some positive integers $n_1,\ldots,n_r$, which contradicts Lemmas~\ref{lemma:36-9-curves}, \ref{lemma:36-10-curves}, \ref{lemma:36-9-orbits}, \ref{lemma:36-10-orbits}.
Hence, we conclude that the curve $C$ is smooth, and one of the following three cases holds:
\begin{enumerate}
\item $G=G_{36,9}$, $d=9$ and $g=1$;
\item $G=G_{36,9}$, $d=9$ and $g=4$;
\item $G=G_{36,10}$, $d=6$ and $g=4$.
\end{enumerate}
If $G=G_{36,10}$, $d=6$, $g=4$, it follows from Lemma~\ref{lemma:36-10-curves} that $C$ has at least two $G$-orbits
of length~$6$, but one of them is the intersection $C\cap S$, so the set $C\setminus(S\cap C)$ contains a $G$-orbit of length $6$, and the stabilizer in $G$ of a point in this orbit is isomorphic to $C_6$. This contradicts Lemma~\ref{lemma:36-10-orbits}. Thus, we conclude that $G=G_{36,9}$, $d=9$ and
\begin{enumerate}
\item either $g=1$;
\item or $g=4$.
\end{enumerate}
In both cases, $S\cap C$ is a $G$-orbit of length $9$, and $C$ has two $G$-orbits of length $9$ by Lemma~\ref{lemma:36-9-curves}. Hence, by Lemma~\ref{lemma:36-9-orbits}, the curve $C$ contains the~$G$-orbit $\Sigma_9^t$ for some $t\in\mathbb{C}$ such that $t^3+2at+4=0$, so it follows from Corollary~\ref{corollary:36-9-Du-Val} that $\lambda=-\frac{2}{t^2}$, and $\mathrm{Sing}(S_\lambda)=\Sigma_9^t$, which contradicts our assumption.

Hence, the surface $S_\lambda$ is singular, and $C$ contains some of its singular points.
By Corollary~\ref{corollary:36-9-Du-Val}, the surface $S_\lambda$ has isolated ordinary double points, and $\mathrm{Sing}(S_\lambda)$ is a $G$-orbit of length $6$ or $9$, so, in particular, the curve $C$ contains all singular points of the surface $S_{\lambda}$.

Let $\widetilde{S}_\lambda\to S_\lambda$ be the~minimal resolution, let $\widetilde{C}$ be the~strict transform on $\widetilde{S}_\lambda$ of the~curve $C$, and let $\widetilde{C}_6$ be the~strict transform on $\widetilde{S}_\lambda$ of the~curve $C_6$. Then the~$G$-action lifts to $\widetilde{S}_\lambda$, and the~intersection $\widetilde{C}\cap\widetilde{C}_6$ is a $G$-orbit of length $d\in\{6,9\}$. Moreover, by Lemma~\ref{lemma:DuVal} below and adjunction formula, we have
$$
2g-2\leqslant 2\mathrm{p}_a(\widetilde{C})-2=\widetilde{C}^2\leqslant C^2-\frac{|\mathrm{Sing}(S_\lambda)|}{2}\leqslant \frac{d^2}{6}-\frac{|\mathrm{Sing}(S_\lambda)|}{2}\leqslant\frac{27-|\mathrm{Sing}(S_\lambda)|}{2}\leqslant\frac{21}{2},
$$
where $\mathrm{p}_a(\widetilde{C})$ is the~arithmetic genus of the~curve $\widetilde{C}$.
Therefore, if $d=6$, then we must have $g<4$, which contradicts Lemma~\ref{lemma:36-10-curves}, since $G=G_{36,10}$ in this case.
Hence, $d=9$ and $G=G_{36,9}$. Then
$$
2g-2\leqslant 2\mathrm{p}_a(\widetilde{C})-2=\widetilde{C}^2\leqslant 10,
$$
Then $\mathrm{p}_a(\widetilde{C})\leqslant 6$. On the other hand, we know that $g\ne 0$ by Lemma~\ref{lemma:36-9-curves}.
Thus, arguing as above, we see that the curve $\widetilde{C}$ is smooth, and either $g=1$ or $g=4$.

By Lemma~\ref{lemma:36-9-curves}, $\widetilde{C}$ contains a $G$-orbit of length $9$ that is different from  $\widetilde{C}\cap\widetilde{C}_6$. Thus, it follows from Lemma~\ref{lemma:36-9-orbits} that the~image of this $G$-orbit in $C$ is a $G$-orbit of length $9$ not contained in $S$. Hence, it follows from Lemma~\ref{lemma:36-9-orbits} and Corollary~\ref{corollary:36-9-Du-Val} that $\mathrm{Sing}(S_\lambda)=\Sigma_9^t\subset C\subset S_\lambda$ for some $t\in\mathbb{C}$ such that $t^3+2at+4=0$, as claimed.

To complete the~proof of the~lemma, we have to show that $C$ is smooth.
If $C$ is singular, then it follows from Lemma~\ref{lemma:DuVal} that
$$
0\leqslant 2g-2=\widetilde{C}^2\leqslant C^2-2|\mathrm{Sing}(S_\lambda)|=C^2-18\leqslant\frac{27}{2}-18<-2,
$$
which is absurd. Hence, we conclude that $C$ is smooth.
\end{proof}

In the proof of Lemma~\ref{lemma:36-9-irreducible-curves}, we used the following result, which is interesting on its own.

\begin{lemma}
\label{lemma:DuVal} Let $Y$ be a normal singular surface, let $P$ be its singular point, let $C\subset Y$ be a curve such that $P\in C$, let $f\colon\widetilde{Y}\to Y$ be the~minimal resolution of the~point~$P$, let $\widetilde{C}$ be the~strict transform on $\widetilde{Y}$ of the~curve $C$. Suppose that $(P\in Y)$ is a Du Val singularity. Then
\begin{equation}
\label{equation:DuVal}
C^2-\widetilde{C}^2\geqslant \left\{\aligned
&\frac{n}{n+1}\ \text{if $(P\in Y)$ is a singular point of type $\mathbb{A}_n$},\\
&1\ \text{if $(P\in Y)$ is a singular point of type $\mathbb{D}_n$},\\
&\frac{4}{3}\ \text{if $(P\in Y)$ is a singular point of type $\mathbb{E}_6$},\\
&\frac{3}{2}\ \text{if $(P\in Y)$ is a singular point of type $\mathbb{E}_7$},\\
&2\ \text{if $(P\in Y)$ is a singular point of type $\mathbb{E}_8$}.
\endaligned
\right.
\end{equation}
Moreover, if the curve $C$ is singular at $P$, then $C^2-\widetilde{C}^2\geqslant 2$.
\end{lemma}

\begin{proof}
If $(P\in Y)$ is a singular point of type $\mathbb{E}_8$, then $C$ is singular at $P$ by \cite[Lemma~A.4.3]{CheltsovPrzyjalkowski2025}. Moreover, if $C$ is smooth at $P$, then the required assertion follows from \cite[Proposition A.1.3]{CheltsovPrzyjalkowski2025}. Thus, we may assume that the curve $C$ is singular at the point $P$. Hence,  $\mathrm{mult}_P(C)\geqslant 2$.

Let $E_1,\ldots,E_n$ be $f$-exceptional curves in $\widetilde{Y}$. Then
$$
\widetilde{C}\sim_{\mathbb{Q}}f^*(C)-\sum_{i=1}^{n}a_iE_i,
$$
where each $a_i$ is a positive rational number. Define the function $\Delta\colon\mathbb{R}^n\to\mathbb{R}$ as follows:
$$
\Delta(x_1,\ldots,x_n)=-\Big(\sum_{i=1}^{n}x_iE_i\Big)^2.
$$
Recall from \cite{Artin1962} that $\Delta(x_1,\ldots,x_n)\geqslant 0$ for every $(x_1,\ldots,x_n)\in\mathbb{R}^n$, because the intersection form of the curves $E_1,\ldots,E_n$ is negative definite. On the other hand, we have
$$
C^2-\widetilde{C}^2=\Delta(a_1,\ldots,a_n).
$$
Thus, we have to show that $\Delta(a_1,\ldots,a_n)\geqslant 2$.

Suppose that $(P\in Y)$ is a Du Val singular point of type $\mathbb{A}_n$. We may assume that
$$
E_i\cdot E_j=\left\{\aligned
&0\ \text{if $|i-j|\geqslant 2$},\\
&1\ \text{if $|i-j|=1$},\\
&-2\ \text{if $i=j$}.
\endaligned
\right.
$$
If $n=1$, then $\Delta(a_1)=2a_1^2$ and $2a_1=E_1\cdot \widetilde{C}=\mathrm{mult}_P(C)$, so
$$
\Delta(a_1)=\frac{\mathrm{mult}^2_P(C)}{2}\geqslant 2.
$$
Hence, we may further assume $n\geqslant 2$. Then
$$
\Delta(a_1,\ldots,a_n)=2\sum_{i=1}^na_i^2-2\sum_{i=1}^{n-1}a_ia_{i+1}=a_1^2+a_n^2+\sum_{i=1}^{n-1}(a_i-a_{i+1})^2,
$$
and it follows from \cite{Artin1962,Artin1966} that
$$
a_1+a_n=\widetilde{C}\cdot\Big(\sum_{i=1}^{n}E_i\Big)=\mathrm{mult}_P(C)\geqslant 2.
$$
On the other hand, we have
$$
\Delta(a_1,\ldots,a_n)=\Delta(a_1-1,a_2-1,\ldots,a_n-1)+2(a_1+a_{n}-2)+2\geqslant 2(a_1+a_{n}-2)+2\geqslant 2
$$
as claimed.

Suppose that $(P\in Y)$ is a Du Val singular point of type $\mathbb{D}_4$. Then $n=4$, and we may assume that the intersection form of the curves $E_1$, $E_2$, $E_3$, $E_4$ is given~in the following table:
\begin{center}\renewcommand\arraystretch{1.42}
\begin{tabular}{|c||c|c|c|c|}
\hline
 $\bullet$  & $E_1$ & $E_2$& $E_3$ & $E_4$\\
\hline\hline
 $E_1$ &  $-2$ & $1$ & $1$ & $1$ \\
\hline
 $E_2$ &  $1$ & $-2$ & $0$ & $0$\\
\hline
 $E_3$ &  $1$ & $0$ & $-2$ & $0$\\
\hline
 $E_4$ &  $1$ & $0$ & $0$ & $-2$\\
\hline
\end{tabular}
\end{center}
Then
$$
\Delta(a_1,a_2,a_3,a_4)=2a_1^2+2a_2^2+2a_3^2+2a_4^2-2a_1a_2-2a_1a_3-2a_1a_4
$$
Moreover, since $2E_1+E_2+E_3+E_4$ is the fundamental cycle of the singular point $P$, it follows from \cite{Artin1962,Artin1966} that $a_1=\widetilde{C}\cdot(2E_1+E_2+E_3+E_4)=\mathrm{mult}_{P}(C)\geqslant 2$.
As above, we see that
$$
\Delta(a_1,a_2,a_3,a_4)=\Delta(a_1-2,a_2-1,a_3-1,a_4-1)+2(a_{1}-2)+2\geqslant 2(a_{1}-2)+2\geqslant 2.
$$

Suppose that $(P\in Y)$ is a Du Val singular point of type $\mathbb{D}_n$ with $n\geqslant 5$. In this case, we may assume that the intersection form of the curves $E_1,\ldots,E_n$ is given~in the following table:
\begin{center}\renewcommand\arraystretch{1.42}
\begin{tabular}{|c||c|c|c|c|c|c|c|c|}
\hline
 $\bullet$  & $E_1$ & $E_2$& $E_3$ & $E_4$ & $E_5$ & $\ldots$ & $E_{n-1}$ & $E_n$\\
\hline\hline
 $E_1$ &  $-2$ & $1$ & $1$ & $1$ & $0$ & $\ldots$ & $0$ & $0$\\
\hline
 $E_2$ &  $1$ & $-2$ & $0$ & $0$ & $0$ & $\ldots$ & $0$ & $0$\\
\hline
 $E_3$ &  $1$ & $0$ & $-2$ & $0$ & $0$ & $\ldots$ & $0$ & $0$\\
\hline
 $E_4$ &  $1$ & $0$ & $0$ & $-2$ & $1$ & $\ldots$ & $0$ & $0$\\
\hline
 $E_5$ &  $0$ & $0$ & $0$ & $1$ & $-2$ & $\ldots$ & $0$ & $0$ \\
\hline
 $\ldots$ & $\ldots$ & $\ldots$ & $\ldots$ & $\ldots$ & $\ldots$ & $\ddots$ & $\ldots$ & $\ldots$\\
\hline
 $E_{n-1}$ &  $0$ & $0$ & $0$ & $0$ & $0$ & $\ldots$ & $-2$ & $1$\\
\hline
 $E_n$ &  $0$ & $0$ & $0$ & $0$ & $0$ & $\ldots$ & $1$ & $-2$\\
\hline
\end{tabular}
\end{center}
Then
$$
\Delta(a_1,\ldots,a_n)=2\sum_{i=1}^na_i^2-2a_1(a_2+a_3+a_4)-2\sum_{i=4}^{n-1}a_ia_{i+1}.
$$
Since  $2E_1+E_2+E_3+2E_4+\cdots+2E_{n-1}+E_n$ is the fundamental cycle of the singular point $P$,  it follows from \cite{Artin1962,Artin1966} that
$$
a_{n-1}=\widetilde{C}\cdot\big(2E_1+E_2+E_3+2E_4+\cdots+2E_{n-1}+E_n\big)=\mathrm{mult}_{P}\big(C\big)\geqslant 2.
$$
Then
$$
\Delta(a_1,\ldots,a_n)=\Delta(a_1-2,a_2-1,a_3-1,a_4-2,\ldots,a_{n-1}-2,a_n-1)+2(a_{n-1}-2)+2\geqslant 2(a_{n-1}-2)+2\geqslant 2.
$$

Suppose that $(P\in Y)$ is a Du Val singular point of type $\mathbb{E}_6$. Then $n=6$, and we may assume that the intersection form of the curves  $E_1$, $E_2$, $E_3$, $E_4$, $E_5$, $E_6$  is given~in the following table:
\begin{center}\renewcommand\arraystretch{1.42}
\begin{tabular}{|c||c|c|c|c|c|c|}
\hline
 $\bullet$  & $E_1$ & $E_2$& $E_3$ & $E_4$ & $E_5$ & $E_6$\\
\hline\hline
 $E_1$ &  $-2$ & $1$ & $0$ & $0$ & $0$ & $0$\\
\hline
 $E_2$ &  $1$ & $-2$ & $1$ & $0$ & $0$ & $0$\\
\hline
 $E_3$ &  $0$ & $1$ & $-2$ & $1$ & $1$ & $0$\\
\hline
 $E_4$ &  $0$ & $0$ & $1$ & $-2$ & $0$ & $0$\\
\hline
 $E_5$ &  $0$ & $0$ & $1$ & $0$ & $-2$ & $1$\\
\hline
 $E_6$ &  $0$ & $0$ & $0$ & $0$ & $1$ & $-2$\\
\hline
\end{tabular}
\end{center}
Then
$$
\Delta(a_1,a_2,a_3,a_4,a_5,a_6)=2a_1^2+2a_2^2+2a_3^2+2a_4^2+2a_5^2+2a_6^2-2a_1a_2-2a_2a_3-2a_3a_4-2a_3a_5-2a_5a_6.
$$
Since $E_1+2E_2+3E_3+2E_4+2E_5+E_6$ is the fundamental cycle of the singular point $P$, it follows from \cite{Artin1962,Artin1966} that
$$
a_{4}=\widetilde{C}\cdot\big(E_1+2E_2+3E_3+2E_4+2E_5+E_6\big)=\mathrm{mult}_{P}\big(C\big)\geqslant 2.
$$
Then
$$
\Delta(a_1,a_2,a_3,a_4,a_5,a_6)=\Delta(a_1-1,a_2-2,a_3-3,a_4-2,a_5-2,a_6-1)+2(a_{4}-2)+2\geqslant 2(a_{4}-2)+2\geqslant 2.
$$

Suppose that $(P\in Y)$ is a Du Val singular point of type $\mathbb{E}_7$. Then $n=7$, and we may assume that the intersection form of the curves $E_1$, $E_2$, $E_3$, $E_4$, $E_5$, $E_6$, $E_7$ is given~in the following table:
\begin{center}\renewcommand\arraystretch{1.42}
\begin{tabular}{|c||c|c|c|c|c|c|c|}
\hline
 $\bullet$  & $E_1$ & $E_2$& $E_3$ & $E_4$ & $E_5$ & $E_6$ & $E_7$\\
\hline\hline
 $E_1$ &  $-2$ & $1$ & $0$ & $0$ & $0$ & $0$ & $0$ \\
\hline
 $E_2$ &  $1$ & $-2$ & $1$ & $0$ & $0$ & $0$& $0$ \\
\hline
 $E_3$ &  $0$ & $1$ & $-2$ & $1$ & $1$ & $0$ & $0$ \\
\hline
 $E_4$ &  $0$ & $0$ & $1$ & $-2$ & $0$ & $0$& $0$ \\
\hline
 $E_5$ &  $0$ & $0$ & $1$ & $0$ & $-2$ & $1$& $0$ \\
\hline
 $E_6$ &  $0$ & $0$ & $0$ & $0$ & $1$ & $-2$ & $1$ \\
\hline
 $E_7$ &  $0$ & $0$ & $0$ & $0$ & $0$ & $1$ & $-2$ \\
\hline
\end{tabular}
\end{center}
Then
$$
\Delta(a_1,a_2,a_3,a_4,a_5,a_6,a_7)=2(a_1^2+a_2^2+a_3^2+a_4^2+a_5^2+a_6^2+a_7^2-a_1a_2-a_2a_3-a_3a_4-a_3a_5-a_5a_6-a_6a_7).
$$
Since $2E_1+3E_2+4E_3+2E_4+3E_5+2E_6+E_7$ is the fundamental cycle of  $(Y\ni P)$, we get
$$
a_{1}=\widetilde{C}\cdot\big(2E_1+3E_2+4E_3+2E_4+3E_5+2E_6+E_7\big)=\mathrm{mult}_{P}\big(C\big)\geqslant 2.
$$
As above, we have
$$
\Delta(a_1,a_2,a_3,a_4,a_5,a_6,a_7)=\Delta(a_1-2,a_2-3,a_3-4,a_4-2,a_5-3,a_6-2,a_7-1)+2(a_{1}-2)+2,
$$
which gives $\Delta(a_1,a_2,a_3,a_4,a_5,a_6,a_7)\geqslant 2(a_{1}-2)+2\geqslant 2$.

Suppose that $(P\in Y)$ is a singular point of type $\mathbb{E}_8$. Then $n=8$, and we may assume that the~intersection form of the curves $E_1$, $E_2$, $E_3$, $E_4$, $E_5$, $E_6$, $E_7$, $E_8$ is given~in the following table:
\begin{center}\renewcommand\arraystretch{1.42}
\begin{tabular}{|c||c|c|c|c|c|c|c|c|}
\hline
 $\bullet$  & $E_1$ & $E_2$& $E_3$ & $E_4$ & $E_5$ & $E_6$ & $E_7$ & $E_8$\\
\hline\hline
 $E_1$ &  $-2$ & $1$ & $0$ & $0$ & $0$ & $0$ & $0$ & $0$\\
\hline
 $E_2$ &  $1$ & $-2$ & $1$ & $0$ & $0$ & $0$& $0$ & $0$\\
\hline
 $E_3$ &  $0$ & $1$ & $-2$ & $1$ & $1$ & $0$ & $0$ & $0$\\
\hline
 $E_4$ &  $0$ & $0$ & $1$ & $-2$ & $0$ & $0$& $0$ & $0$\\
\hline
 $E_5$ &  $0$ & $0$ & $1$ & $0$ & $-2$ & $1$& $0$ & $0$\\
\hline
 $E_6$ &  $0$ & $0$ & $0$ & $0$ & $1$ & $-2$ & $1$ & $0$\\
\hline
 $E_7$ &  $0$ & $0$ & $0$ & $0$ & $0$ & $1$ & $-2$ & $1$\\
\hline
 $E_8$ &  $0$ & $0$ & $0$ & $0$ & $0$ & $0$ & $1$ & $-2$\\
\hline
\end{tabular}
\end{center}
Then
\begin{multline*}
\Delta(a_1,a_2,a_3,a_4,a_5,a_6,a_7,a_8)=2a_1^2+2a_2^2+2a_3^2+2a_4^2+2a_5^2+2a_6^2+2a_7^2+2a_8^2-\\
-2a_1a_2-2a_2a_3-2a_3a_4-2a_3a_5-2a_5a_6-2a_6a_7-2a_7a_8.\quad \quad
\end{multline*}
Since $2E_1+4E_2+6E_3+3E_4+5E_5+4E_6+3E_7+2E_8$ is the fundamental cycle of the point $P$, it~follows from \cite{Artin1962,Artin1966} that
$$
a_{8}=\widetilde{C}\cdot\big(2E_1+4E_2+6E_3+3E_4+5E_5+4E_6+3E_7+2E_8\big)=\mathrm{mult}_{P}\big(C\big)\geqslant 2.
$$
Since $\Delta$ is a non-negative function, we have
$$
\Delta(a_1,a_2,a_3,a_4,a_5,a_6,a_7,a_8)=\Delta(a_1-2,a_2-4,a_3-6,a_4-3,a_5-5,a_6-4,a_7-3,a_8-2)+2(a_{8}-2)+2,
$$
which implies that $\Delta(a_1,a_2,a_3,a_4,a_5,a_6,a_7,a_8)\geqslant 2(a_{8}-2)+2\geqslant 2$ as claimed.
\end{proof}

Now, we are ready to prove the~main technical result of this section.

\begin{proposition}
\label{proposition:36-9-curves}
Let $Z$ be a $G$-irreducible curve in $X_{a}$,
and let $\mathcal{M}$ be a non-empty $G$-invariant mobile linear subsystem in $|nH|$ for some $n\in\mathbb{Z}_{>0}$. Then
$$
\mathrm{mult}_Z\big(\mathcal{M}\big)\leqslant\frac{n}{2}.
$$
\end{proposition}

\begin{proof}
Suppose that $\mathrm{mult}_Z(\mathcal{M})>\frac{n}{2}$. Let $M$ and $M^\prime$ be general surfaces in $\mathcal{M}$. Then
$$
\mathrm{mult}_Z\big(M\big)=\mathrm{mult}_Z\big(M^\prime\big)=\mathrm{mult}_Z\big(\mathcal{M}\big)>\frac{n}{2}.
$$
Then, in particular, we have
$$
M\cdot M^\prime=\mathrm{mult}^2_Z\big(\mathcal{M}\big)Z+\Delta,
$$
where $\Delta$ is an effective one-cycle (whose support may, a priori, contain $Z$). Then
$$
3n^2=H\cdot M\cdot M^\prime=\mathrm{mult}^2_Z\big(\mathcal{M}\big)\mathrm{deg}(Z)+H\cdot\Delta\geqslant\mathrm{mult}^2_Z\big(\mathcal{M}\big)\mathrm{deg}(Z)>\frac{n^2}{4}\mathrm{deg}(Z),
$$
where $H$ is a general hyperplane section of the~cubic $X_{a}\subset\mathbb{P}^4$. Therefore, we conclude $\mathrm{deg}(Z)<12$. Let us seek for a contradiction.

If $Z\subset S$, then the log pair $(S,\frac{2}{n}\mathcal{M}\vert_{S})$ is not log canonical along the curve $Z$, so $\alpha_G(S)<2$, because $\mathcal{M}\vert_{S}\sim n(-K_S)$, and  $\mathcal{M}\vert_{S}$ is $G$-invariant. However, $\alpha_G(S)=2$ by Corollaries~\ref{corollary:36-9-alpha} and \ref{corollary:36-10-alpha}.

Hence, it follows from Lemmas~\ref{lemma:36-9-reducible-curves}, \ref{lemma:36-10-reducible-curves} and \ref{lemma:36-9-irreducible-curves} that one of the following seven cases holds:
\begin{enumerate}
\item[$(\mathrm{i})$] $Z=\mathcal{R}_6$;
\item[$(\mathrm{ii})$] $G=G_{36,10}$, and $Z$ is one of the curves $\mathcal{R}_6^\prime$ or $\mathcal{R}_6^{\prime\prime}$;
\item[$(\mathrm{iii})$] $G=G_{36,10}$, and $Z$ is one of the curves $\mathcal{R}_9$ or $\mathcal{R}_9^\prime$;
\item[$(\mathrm{iv})$] $G=G_{36,10}$, $a\ne 0$, and $Z$ is one of the curves $\mathcal{K}_6$ or $\mathcal{K}_6^{\prime}$;
\item[$(\mathrm{v})$] $G=G_{36,9}$, and $Z=\mathcal{L}_t^{+}$ or $Z=\mathcal{L}_t^{-}$ for some $t\in\mathbb{C}$ such that $t^3+2at+4=0$;
\item[$(\mathrm{vi})$] $G=G_{36,9}$, and $Z$ is an irreducible smooth curve of genus $1$ and degree $9$;
\item[$(\mathrm{vii})$] $G=G_{36,9}$, and $Z$ is an irreducible smooth curve of genus $4$ and degree $9$.
\end{enumerate}
Moreover, $Z$ is smooth unless $X_a$ is singular, $G=G_{36,10}$, and $Z$ is one of the curves $\mathcal{R}_9$ or $\mathcal{R}_9^\prime$.

Let us exclude cases $(\mathrm{i})$, $(\mathrm{ii})$, $(\mathrm{iii})$. Suppose that $Z$ is one of the curves $\mathcal{R}_6$, $\mathcal{R}_6^\prime$, $\mathcal{R}_6^{\prime\prime}$, $\mathcal{R}_9$, $\mathcal{R}_9^\prime$. Then $Z$ is a disjoint union of plane cubic curves (if $X_a$ is singular, and $Z=\mathcal{R}_9$ or $Z=\mathcal{R}_9^\prime$, these cubic curves are reducible). Let $C$ be one of these cubic curves, and let $S_C$ be a general hyperplane section of the cubic $X_a$ that contains $C$. Then $S_C$ is a normal cubic surface, which is actually smooth unless $X_a$ is singular, and $Z=\mathcal{R}_9$ or $Z=\mathcal{R}_9^\prime$. If $X_a$ is singular, and $Z=\mathcal{R}_9$ or $Z=\mathcal{R}_9^\prime$, then $S_C$ has three ordinary double points, which are contained in $C$. Furthermore, we have
$$
n(-K_{S_C})\sim_{\mathbb{Q}}\mathcal{M}\big\vert_{S_C}=\mathcal{M}_{S_C}+\mathrm{mult}_Z\big(\mathcal{M}\big)C,
$$
where $\mathcal{M}_{S_C}$ is a non-empty mobile linear system on the cubic surface $S_C$. Note that $C\sim -K_{S_C}$, the intersection $S_C\cap (Z\setminus C)$ is contained in the smooth locus of the surface $S_C$, and this intersection consists of $\mathrm{deg}(Z)-3\geqslant 3$ distinct points, which are not contained in $C$. Hence, we have
$$
\mathrm{mult}_O\big(\mathcal{M}_{S_C}\big)>\frac{n}{2}
$$
for every point $O\in S_C\cap (Z\setminus C)$. Thus, if $M_{S_C}$ and $M_{S_C}^\prime$ are general curves in $\mathcal{M}_{S_C}$, then
\begin{multline*}
\frac{3n^2}{4}=\Big(n-\frac{n}{2}\Big)^2\big(-K_{S_C}\big)^2>\Big(n-\mathrm{mult}_Z\big(\mathcal{M}\big)\Big)^2\big(-K_{S_C}\big)^2=\\
=M_{S_C}\cdot M_{S_C}^\prime\geqslant\sum_{O\in S_C\cap (Z\setminus C)}\Big(M_{S_C}\cdot M_{S_C}^\prime\Big)_O\geqslant
\sum_{O\in S_C\cap (Z\setminus C)}\mathrm{mult}_O\big(M_{S_C}\big)\mathrm{mult}_O\big(M_{S_C}^\prime\big)=\\
=\sum_{O\in S_C\cap (Z\setminus C)}\mathrm{mult}^2_O\big(\mathcal{M}_{S_C}\big)>
\sum_{O\in S_C\cap (Z\setminus C)}\Big(\frac{n}{2}\Big)^2=\big|S_C\cap (Z\setminus C)\big|\frac{n^2}{4}\geqslant\frac{3n^2}{4},
\end{multline*}
which is a contradiction. Thus, we see that $Z$ is not one of the curves $\mathcal{R}_6$, $\mathcal{R}_6^\prime$, $\mathcal{R}_6^{\prime\prime}$, $\mathcal{R}_9$, $\mathcal{R}_9^\prime$.

Let us present another way how to exclude cases $(\mathrm{i})$, $(\mathrm{ii})$, $(\mathrm{iii})$ in the case when $X_a$ is smooth, which exclude cases $(\mathrm{i})$, $(\mathrm{ii})$, $(\mathrm{iv})$ unconditionally. Namely, suppose that the curve $Z$ is smooth, and $Z\cap \mathrm{Sing}(X_a)=\varnothing$.
Let $\pi\colon \widetilde{X}_{a}\to X_{a}$ be the~blow up~of~$Z$, let $F$ be the $\pi$-exceptional divisor, let~$\widetilde{M}$ and $\widetilde{M}^\prime$ be the strict transforms on $\widetilde{X}_{a}$ of  $M$ and $M^\prime$, respectively. Set $\widetilde{H}=\pi^*(H)$. Then
$$
\widetilde{M}\sim \widetilde{M}^\prime\sim n\widetilde{H}-\mathrm{mult}_Z\big(\mathcal{M}\big)F.
$$
Note that $\widetilde{M}\cdot\widetilde{M}^\prime$ is an effective one-cycle.
Take $k\in\mathbb{Z}_{>0}$ such that $k\widetilde{H}-F$ is nef. Then
\begin{equation}
\label{equation:36-9-test-class}
0\leqslant \widetilde{M}\cdot \widetilde{M}^\prime\cdot \big(k\widetilde{H}-F\big)=3kn^2-2n\mathrm{deg}(Z)\mathrm{mult}_Z\big(\mathcal{M}\big)-(k\mathrm{deg}(Z)+F^3)\mathrm{mult}^2_Z\big(\mathcal{M}\big).
\end{equation}
Note that
$$
F^3=\left\{\aligned
&-12\ \text{if $Z=\mathcal{R}_6$, $Z=\mathcal{R}_6^\prime$ or $Z=\mathcal{R}_6^{\prime\prime}$},\\
&-18\ \text{if $Z=\mathcal{R}_9$ or $Z=\mathcal{R}_9^\prime$},\\
&-6\ \text{if $a\ne 0$, and either $Z=\mathcal{K}_6$ or $Z=\mathcal{K}_6^{\prime}$},\\
&0\ \text{if $Z=\mathcal{L}_t^{+}$ or $Z=\mathcal{L}_t^{-}$},\\
&-18\ \text{if $Z$ is a curve of degree $9$ and genus $1$},\\
&-12\ \text{if $Z$ is a curve of degree $9$ and genus $4$}.
\endaligned
\right.
$$
If $Z$ is one of the curves $\mathcal{R}_6$, $\mathcal{R}_6^\prime$ or $\mathcal{R}_6^{\prime\prime}$, then the linear system $|2\widetilde{H}-F|$ is base point free and gives an elliptic fibration $\widetilde{X}_{a}\to\mathbb{P}^1\times \mathbb{P}^1$, so, in particular, we can set $k=2$. Then \eqref{equation:36-9-test-class} gives
$$
0\leqslant \widetilde{M}\cdot \widetilde{M}^\prime\cdot \big(2\widetilde{H}-F\big)=6n^2-12n\mathrm{mult}_Z\big(\mathcal{M}\big),
$$
which gives $\mathrm{mult}_Z(\mathcal{M})\leqslant\frac{n}{2}$, which contradicts our assumption. Similarly, if $Z=\mathcal{R}_9$ or $Z=\mathcal{R}_9^\prime$, the linear system $|3\widetilde{H}-F|$ is base point free, so we can set $k=3$. In this case,  \eqref{equation:36-9-test-class} gives
$$
0\leqslant \widetilde{M}\cdot \widetilde{M}^\prime\cdot \big(3\widetilde{H}-F\big)=9n^2-18n\mathrm{mult}_Z\big(\mathcal{M}\big)-9\mathrm{mult}^2_Z\big(\mathcal{M}\big),
$$
so $\mathrm{mult}_Z(\mathcal{M})\leqslant\frac{n}{2}$, which contradicts our assumption. Likewise, if $a\ne 0$, and $Z=\mathcal{K}_6$ or $Z=\mathcal{K}_6^\prime$, then the divisor $3\widetilde{H}-F$ is nef, so applying \eqref{equation:36-9-test-class} with $k=3$, we get
$$
0\leqslant \widetilde{M}\cdot \widetilde{M}^\prime\cdot \big(3\widetilde{H}-F\big)=
3\Big(3n+2\mathrm{mult}_Z\big(\mathcal{M}\big)\Big)\Big(n-2\mathrm{mult}_Z\big(\mathcal{M}\big)\Big),
$$
which gives $\mathrm{mult}_Z(\mathcal{M})\leqslant\frac{n}{2}$ that contradicts our assumption.

Summarizing, we see that cases $(\mathrm{i})$, $(\mathrm{ii})$, $(\mathrm{iv})$  are impossible, because in these cases $Z$ is always smooth, and $Z$ is always contained in the smooth locus of the cubic $X_a$. Similarly, we again see that case $(\mathrm{iii})$ is impossible if $X_a$ is smooth. So, we conclude that $G=G_{36,9}$, the curve $Z$ is smooth, and one of the following cases holds:
\begin{enumerate}
\item[$(\mathrm{v})$] $Z=\mathcal{L}_t^{+}$ or $Z=\mathcal{L}_t^{-}$ for $t\in\mathbb{C}$ such that $t^3+2at+4=0$;
\item[$(\mathrm{vi})$] $Z$ is an irreducible curve of degree $9$ and genus $1$;
\item[$(\mathrm{vii})$] $Z$ is an irreducible curve of degree $9$ and genus $4$.
\end{enumerate}
Then it follows from Lemmas~\ref{lemma:36-9-reducible-curves} and \ref{lemma:36-9-irreducible-curves} that
$$
\mathrm{Sing}(S_\lambda)=\Sigma_{9}^t\subset Z\subset S_\lambda
$$
for $\lambda=-\frac{2}{t^2}$, where $t$ is a non-zero number such that $t^3+2at+4=0$. Furthermore, by Corollary~\ref{corollary:36-9-Du-Val}, the~singularities of the~surface $S_\lambda$ are (ordinary) nodes. Let $\varpi\colon \widetilde{S}_\lambda\to S_{\lambda}$ be the~minimal~resolution, let $E$ be the~$\varpi$-exceptional divisor (consisting of $9$ disjoint $(-2)$-curves), and let $\widetilde{Z}$ be the~strict transform on $\widetilde{S}_\lambda$ of the~curve $Z$. Then $\widetilde{S}_\lambda$ is a smooth K3 surface, the~$G$-action lifts to $\widetilde{S}_\lambda$, and
$$
\widetilde{Z}\sim_{\mathbb{Q}} \varpi^*(Z)-\frac{1}{2}E,
$$
because $Z$ is a smooth curve. On the surface  $\widetilde{S}_\lambda$, by adjunction formula, we have
$$
\widetilde{Z}^2=\left\{\aligned
&-18\ \text{if $Z=\mathcal{L}_t^{\pm}$},\\
&0\ \text{if $Z$ is a curve of genus $1$},\\
&6\ \text{if $Z$ is a curve of genus $4$}.
\endaligned
\right.
$$
Set $H_{S_{\lambda}}=H\vert_{S_{\lambda}}$ and $\widetilde{H}_{\widetilde{S}_{\lambda}}=\varpi^*(H_{S_{\lambda}})$.  Then the intersections of the divisors $\widetilde{Z}$, $\widetilde{H}_{S_{\lambda}}$, $E$  are given in the following table:
\begin{center}
\centering\renewcommand{\arraystretch}{1.5}
\begin{tabular}{|c||c|c|c|}
\hline
$\bullet$         &  $\widetilde{Z}$    & $\widetilde{H}_{\widetilde{S}_{\lambda}}$ & $E$\\
\hline\hline
$\widetilde{Z}$   & $\widetilde{Z}^2$ & $9$ & $9$  \\
\hline
$\widetilde{H}_{\widetilde{S}_{\lambda}}$   & $9$ & $6$& $0$\\
\hline
$E$   & $9$ & $0$&  $-18$\\
\hline
\end{tabular}
\end{center}
Recall that $M$ is a general surface in $\mathcal{M}$. Then $M\vert_{S_{\lambda}}=mZ+D$ for an integer $m\geqslant \mathrm{mult}_Z(\mathcal{M})$~and an effective divisor $D$ on $S_{\lambda}$. Let $\widetilde{D}$ be the strict transform of the divisor $D$ on $\widetilde{S}_\lambda$. Then
$$
m\widetilde{Z}+\widetilde{D}\sim_{\mathbb{Q}} n\widetilde{H}_{\widetilde{S}_\lambda}-qE
$$
for some non-negative integer $q\geqslant 0$. Since $0\leqslant E\cdot \widetilde{D}=18q-9m$, we have $q\geqslant\frac{m}{2}$.

Consider the $G$-invariant linear system $|3\widetilde{H}_{\widetilde{S}_\lambda}-\widetilde{Z}-E|$. If $Z=\mathcal{L}_t^{\pm}$, it follows from \eqref{equation:36-9-nine-lines} that the~linear system $|3\widetilde{H}_{\widetilde{S}_\lambda}-\widetilde{Z}-E|$ contains the strict transform of the curve $\mathcal{L}_t^{\mp}$. This immediately implies that the divisor $5\widetilde{H}_{\widetilde{S}_\lambda}-\widetilde{Z}-E$ is nef. Hence, if $Z=\mathcal{L}_t^{\pm}$, then
$$
0\leqslant \big(5\widetilde{H}_{\widetilde{S}_\lambda}-\widetilde{Z}-E\big)\cdot \widetilde{D}=\big(5\widetilde{H}_{\widetilde{S}_\lambda}-\widetilde{Z}-E\big)\cdot\big(n\widetilde{H}_{\widetilde{S}_\lambda}-qE-m\widetilde{Z})=21n-54m-9q,
$$
which gives that $m<\frac{n}{2}$, which contradicts our assumption. Thus, we see that $Z\ne\mathcal{L}_t^{\pm}$. Then
$$
\big(3\widetilde{H}_{\widetilde{S}_\lambda}-\widetilde{Z}-E\big)^2=\widetilde{Z}^2\geqslant 0,
$$
so it follows from the~Riemann--Roch formula and Serre duality that
$$
h^0\big(\mathcal{O}_{\widetilde{S}_{\lambda}}(3\widetilde{H}_{\widetilde{S}_\lambda}-\widetilde{Z}-E)\big)\geqslant 2+\frac{\widetilde{Z}^2}{2}\geqslant 2.
$$
In particular, the~linear system $|3\widetilde{H}_{\widetilde{S}_\lambda}-\widetilde{Z}-E|$ has a non-empty mobile part.

We claim that the linear system $|3\widetilde{H}_{\widetilde{S}_\lambda}-\widetilde{Z}-E|$ does not have fixed curves. Indeed, write
$$
|3\widetilde{H}_{\widetilde{S}_\lambda}-\widetilde{Z}-E|=\mathscr{F}+|\mathscr{M}|,
$$
where $\mathscr{F}$ is the~fixed part of this linear system, and $\mathscr{M}$ is a  divisor  such that $|\mathscr{M}|$ is the mobile part of the linear system $|3\widetilde{H}_{\widetilde{S}_\lambda}-\widetilde{Z}-E|$. Then $\mathscr{F}$ is $G$-invariant, and
$$
9=\widetilde{H}_{\widetilde{S}_\lambda}\cdot\big(3\widetilde{H}_{\widetilde{S}_\lambda}-\widetilde{Z}-E\big)=\widetilde{H}_{\widetilde{S}_\lambda}\cdot\big(\mathscr{F}+\mathscr{M}\big)>\widetilde{H}_{\widetilde{S}_\lambda}\cdot\mathscr{F}=\mathrm{deg}\big(\varpi(\mathscr{F})\big).
$$
On the other hand, by Lemmas~\ref{lemma:36-9-reducible-curves} and \ref{lemma:36-9-irreducible-curves}, the only $G$-invariant curves of degree~$<9$ contained in the surface $S_\lambda$ is the smooth curve $\mathcal{C}_6$, which is a very ample divisor on $S_\lambda$ contained in the smooth locus of the surface $S_\lambda$. Hence, we see that $\mathrm{deg}(\varpi(\mathscr{F}))=0$, so the divisor $\mathscr{F}$ is $\varpi$-exceptional. Then $\mathscr{F}=pE$ for some positive integer $p$, so
$$
0\leqslant \mathscr{M}^2=\big(3\widetilde{H}_{\widetilde{S}_\lambda}-\widetilde{Z}-E-\mathscr{F}\big)^2=\big(3\widetilde{H}_{\widetilde{S}_\lambda}-\widetilde{Z}-(1+p)E\big)^2=\widetilde{Z}^2-18p^2-18p\leqslant -30,
$$
which is absurd. Thus, the linear system $|3\widetilde{H}_{\widetilde{S}_\lambda}-\widetilde{Z}-E|$ does not have fixed curves.

In particular, the divisor $3\widetilde{H}_{\widetilde{S}_\lambda}-\widetilde{Z}-E$ is nef. Observe that
\begin{multline*}
0\leqslant \big(3\widetilde{H}_{\widetilde{S}_\lambda}-\widetilde{Z}-E\big)\cdot\widetilde{D}=\big(3\widetilde{H}_{\widetilde{S}_\lambda}-\widetilde{Z}-E\big)\cdot\big(n\widetilde{H}_{\widetilde{S}_\lambda}-qE-m\widetilde{Z}\big)=\\
=9n+m(\widetilde{Z}^2-18)-9q\leqslant 9n+m(\widetilde{Z}^2-18)-\frac{9m}{2}=9n+m\Big(\widetilde{Z}^2-\frac{45}{2}\Big),
\end{multline*}
since $q\geqslant\frac{m}{2}$. So, if $\widetilde{Z}^2=0$, then $m<\frac{n}{2}$, which contradicts our assumption. Then $Z$ is a smooth curve of genus $4$ and degree $9$, and $\widetilde{Z}^2=6$.

Now, we consider the $G$-invariant linear system $|5\widetilde{H}_{\widetilde{S}_\lambda}-2\widetilde{Z}-E|$. Since $(5\widetilde{H}_{\widetilde{S}_\lambda}-2\widetilde{Z}-E)^2=12$, it follows from the~Riemann--Roch formula and Serre duality that $\mathrm{dim}(|5\widetilde{H}_{\widetilde{S}_\lambda}-2\widetilde{Z}-E|)\geqslant 7$.
Moreover, we claim that  $|5\widetilde{H}_{\widetilde{S}_\lambda}-2\widetilde{Z}-E|$ does not have fixed curves. Indeed, as above, we write
$$
|5\widetilde{H}_{\widetilde{S}_\lambda}-2\widetilde{Z}-E|=\mathscr{F}^\prime+|\mathscr{M}^\prime|,
$$
where $\mathscr{F}^\prime$ is the~fixed part of the linear system $|5\widetilde{H}_{\widetilde{S}_\lambda}-2\widetilde{Z}-E|$, and $\mathscr{M}^\prime$ is an effective divisor on the surface $\widetilde{S}_\lambda$ such that the complete linear system $|\mathscr{M}^\prime|$ is the mobile part of $|5\widetilde{H}_{\widetilde{S}_\lambda}-2\widetilde{Z}-E|$. We may further assume that $\mathscr{M}^\prime$ is a general member of the linear system $|\mathscr{M}^\prime|$. Then
$$
\widetilde{H}_{\widetilde{S}_\lambda}\cdot \mathscr{M}^\prime=\mathrm{deg}\big(\varpi(\mathscr{M}^\prime)\big)\geqslant 3,
$$
because $S_\lambda$ is not uniruled. Moreover, if $\widetilde{H}_{\widetilde{S}_\lambda}\cdot \mathscr{M}^\prime=3$, then $\varpi(\mathscr{M}^\prime)$ is a smooth irreducible cubic curve, so $\mathscr{M}^\prime$ is an irreducible smooth curve of genus $1$ and, therefore,  $|\mathscr{M}^\prime|$ is a pencil \cite{SD}, which is impossible, since
$\mathrm{dim}(|\mathscr{M}^\prime|)=\mathrm{dim}(|5\widetilde{H}_{\widetilde{S}_\lambda}-2\widetilde{Z}-E|)\geqslant 7$.
Therefore,  $\widetilde{H}_{\widetilde{S}_\lambda}\cdot \mathscr{M}^\prime\geqslant 4$, which gives $\mathrm{deg}(\varpi(\mathscr{F}^\prime))\leqslant 8$, because
$$
12=\widetilde{H}_{\widetilde{S}_\lambda}\cdot\big(5\widetilde{H}_{\widetilde{S}_\lambda}-2\widetilde{Z}-E\big)=\widetilde{H}_{\widetilde{S}_\lambda}\cdot\big(\mathscr{F}^\prime+\mathscr{M}^\prime\big)\geqslant \widetilde{H}_{\widetilde{S}_\lambda}\cdot\mathscr{F}^\prime+4=\mathrm{deg}\big(\varpi(\mathscr{F}^\prime)\big)+4.
$$
Now, it follows from Lemmas~\ref{lemma:36-9-reducible-curves} and \ref{lemma:36-9-irreducible-curves} that $\mathscr{F}^\prime=p^\prime E$ for some $p^\prime\in\mathbb{Z}_{>0}$. Then
$$
0\leqslant (\mathscr{M}^\prime)^2=\big(5\widetilde{H}_{\widetilde{S}_\lambda}-2\widetilde{Z}-E-\mathscr{F}^\prime\big)^2=\big(5\widetilde{H}_{\widetilde{S}_\lambda}-2\widetilde{Z}-(1+p^\prime)E\big)^2=12-18(p^\prime)^2\leqslant -6.
$$
Thus, $|5\widetilde{H}_{\widetilde{S}_\lambda}-2\widetilde{Z}-E|$ does not have fixed curves, so $5\widetilde{H}_{\widetilde{S}_\lambda}-2\widetilde{Z}-E$ is nef. Then
$$
0\leqslant \big(5\widetilde{H}_{\widetilde{S}_\lambda}-2\widetilde{Z}-E\big)\cdot \widetilde{D}=\big(5\widetilde{H}_{\widetilde{S}_\lambda}-2\widetilde{Z}-E\big)\cdot\big(n\widetilde{H}_{\widetilde{S}_\lambda}-qE-m\widetilde{Z}\big)=12n-24m,
$$
which gives $m\leqslant\frac{n}{2}$, which is a contradiction. This completes the proof of Proposition~\ref{proposition:36-9-curves}.
\end{proof}

\begin{corollary}
\label{corollary:36-9-curves}
Let $n\in\mathbb{Z}_{>0}$, let $\mathcal{M}$ be a non-empty $G$-invariant mobile linear subsystem in $|nH|$, and let 
$$
\Sigma=\Big\{P\in X_{a}\ \text{such that $\Big(X_{a},\frac{2}{n}\mathcal{M}\Big)$ is not canonical at~$P$}\Big\}.
$$
Then $\Sigma$ is a finite subset, $\Sigma\cap S=\varnothing$ and $\Sigma\cap\mathcal{R}_6=\varnothing$.
If $G=G_{36,10}$, then 
$$
\Sigma\cap\big(\mathcal{R}_6^\prime\cup \mathcal{R}_6^{\prime\prime}\big)=\varnothing.
$$
If $G=G_{36,10}$ and $a\ne 0$, then $\Sigma\cap(\mathcal{K}_6\cup \mathcal{K}_6^\prime)=\varnothing$.
If $X_a$ is singular, then $\Sigma\cap\mathrm{Sing}(X_a)=\varnothing$.
\end{corollary}

\begin{proof}
If $\Sigma$ contains a $G$-irreducible curve $Z$, then it follows from \cite[Exercise 6.18]{KollarSmithCorti} that
$$
\mathrm{mult}_Z\big(\mathcal{M}\big)>\frac{n}{2},
$$
which is impossible by Proposition~\ref{proposition:36-9-curves}. Thus, we deduce that $\Sigma$ is a finite subset.

If there exists a point $P\in\Sigma$ such that $P$ is contained in the surface $S$, then it follows from the~inversion of adjunction \cite[Theorem 6.29]{KollarSmithCorti} that $(S,\frac{2}{n}\mathcal{M}\vert_{S})$ is not log canonical at $P$, which contradicts Corollaries~\ref{corollary:36-9-alpha} and \ref{corollary:36-10-alpha}. Thus, $\Sigma\cap S=\varnothing$.

Let $M$ and $M^\prime$ be general surfaces in $\mathcal{M}$. Then it follows from \cite{Pukhlikov1998,Corti2000} that $$
\mathrm{mult}_P\big(M\cdot M^\prime)>n^2
$$ 
for every point $P\in \Sigma$ such that $X_a$ is smooth at $P$. Let us use this to show that $\Sigma\cap\mathcal{R}_6=\varnothing$.
Namely, suppose that $\Sigma$ contains a point $P\in \mathcal{R}_6$. Then $X_a$ is smooth at $P$, and $M\cdot M^\prime=m\mathcal{R}_6+\Delta$, where $\Delta$ is an effective one-cycle on $X_{a}$ whose support does not contain $\mathcal{R}_6$. Then $m\leqslant \frac{n^2}{2}$, because
$$
3n^2=H\cdot M\cdot M^\prime=H\cdot\big(m\mathcal{R}_6+\Delta\big)=6m+H\cdot\Delta\geqslant 6m,
$$
where $H$ is a general hyperplane section of the cubic $X_{a}$. Let $\Sigma_P$ be the $G$-orbit of the point $P$. Then, since $\mathcal{R}_6$ is smooth, $\mathrm{mult}_O(\Delta)>n^2-m$ for every point $O\in\Sigma_P$. Let $\mathcal{D}$ be the linear subsystem in $|2H|$ consisting of all surfaces containing $\mathcal{R}_6$, and let $D$ be its general~member. Then~$D$ does not contain curves in $\mathrm{Supp}(\Delta)$, and we have $|\Sigma_P|\geqslant 6$ by Lemmas~\ref{lemma:36-9-orbits} and \ref{lemma:36-10-orbits}. Then
$$
6n^2-12m=D\cdot\Delta\geqslant\sum_{O\in \Sigma_P}\big(D\cdot\Delta)_O\geqslant \sum_{O\in\Sigma_P}\mathrm{mult}_O(\Delta)>|\Sigma_P|(n^2-m)\geqslant 6(n^2-m),
$$
which is a contradiction. Thus, we conclude that $\Sigma\cap\mathcal{R}_6=\varnothing$. 

Similarly, if $G=G_{36,10}$, we see that $\Sigma\cap\big(\mathcal{R}_6^\prime\cup\mathcal{R}_6^{\prime\prime}\big)=\varnothing$.

Let us show that $\Sigma\cap\mathcal{K}_6=\varnothing$ if $G=G_{36,10}$ and $a\ne 0$. Suppose that $G=G_{36,10}$ and $\Sigma\cap\mathcal{K}_6\ne\varnothing$. Let $P\in\Sigma\cap\mathcal{K}_6$. As above, write
$$
M\cdot M^\prime=m^\prime\mathcal{K}_6+\Delta^\prime,
$$
where $\Delta^\prime$ is an effective one-cycle such that $\mathcal{K}_6\not\subset\mathrm{Supp}(\Delta^\prime)$. Arguing as above, we get $m^\prime\leqslant\frac{n^2}{2}$. As above, we let $\Sigma_P$ be the $G$-orbit of the point $P$. If $a\ne 0$, then  $\mathcal{K}_6$ is smooth, which gives
$$
\mathrm{mult}_O(\Delta^\prime)>n^2-m^\prime
$$
for every point $O\in\Sigma_P$. Let $\mathcal{D}^\prime$ be the linear subsystem in $|3H|$ that consists of all surfaces passing through the curve $\mathcal{K}_6$, and let $D^\prime$ be its general member. Then $D^\prime$ contains no curves in $\mathrm{Supp}(\Delta^\prime)$. Moreover, it follows from Lemma~\ref{lemma:36-10-orbits} that
$|\Sigma_P|\geqslant 9$, since  $\Sigma\cap\mathcal{R}_6=\varnothing$. Thus, we have
$$
9n^2-18m^\prime\geqslant D^\prime\cdot\Delta^\prime
\geqslant\sum_{O\in\Sigma_P}\mathrm{mult}_O(\Delta^\prime)
>|\Sigma_P|(n^2-m^\prime)\geqslant 9(n^2-m^\prime),
$$
which is a contradiction. Thus, if $G=G_{36,10}$ and $a\ne 0$,
then $\Sigma\cap\mathcal{K}_6=\varnothing$.

Likewise, if  $G=G_{36,10}$ and $a\ne 0$, then  $\Sigma\cap\mathcal{K}_6^\prime=\varnothing$.

Finally, we suppose that the cubic $X_a$ is singular ($a^3=-\frac{27}{2}$). Let us show that $\Sigma\cap\mathrm{Sing}(X_a)=\varnothing$. Suppose  this is not true. Then $\mathrm{Sing}(X_a)\subseteq\Sigma$, and $X_a$ contains a plane $\Pi$ that contains $4$ singular points of the cubic $X_a$. Denote them by $P_1$, $P_2$, $P_3$,~$P_4$. Let $f\colon\widetilde{X}_a\to X_a$ be the blow up of~these points, let $E_1$, $E_2$, $E_3$, $E_4$ be $f$-exceptional surfaces such that $f(E_1)=P_1$, $f(E_2)=P_2$, $f(E_3)=P_3$, $f(E_4)=P_4$, let $\widetilde{\mathcal{M}}$ be the strict transform on $\widetilde{X}_a$ of~$\mathcal{M}$, let $\widetilde{M}$ be a general surface in $\widetilde{\mathcal{M}}$. Then
$$
\widetilde{M}\sim\widetilde{\mathcal{M}}\sim f^*\big(nH\big)-\mu(E_1+E_2+E_3+E_4)
$$
for some $\mu\in\mathbb{Z}_{>0}$. Moreover, one has $\mu>\frac{n}{2}$ by \cite[Theorem 3.10]{Corti2000} or \cite[Theorem 1.7.20]{Cheltsov2005}. Now, we let $\mathscr{C}$ be a general conic in $\Pi$ that passes through $P_1$, $P_2$, $P_3$, $P_4$, and let $\widetilde{\mathscr{C}}$ be its strict transform on $\widetilde{X}_a$. Then $\widetilde{\mathscr{C}}\not\subset\widetilde{M}$, because $\Pi$ is not contained in the base locus of $\mathcal{M}$. Then
$$
0\leqslant \widetilde{M}\cdot\widetilde{\mathscr{C}}=f^*\big(nH\big)\cdot\widetilde{\mathscr{C}}-\mu(E_1+E_2+E_3+E_4)\cdot\widetilde{\mathscr{C}}=2n-4\mu,
$$
which gives $\mu\leqslant\frac{n}{2}$. This is a contradiction.
\end{proof}

Now, we are ready to prove Theorem~\ref{theorem:36-9}.

\begin{proof}[Proof of Theorem~\ref{theorem:36-9}]
Suppose that $X_{a}$ is not $G$-birationally superrigid.
By \cite[Corollary 3.3.3]{CheltsovShramov2016}, there exists a non-empty $G$-invariant mobile linear subsystem
$\mathcal{M}\subset |nH|$, for a positive integer~$n$, such that the singularities of the log pair $(X_{a},\frac{2}{n}\mathcal{M})$ are not canonical. Set
$$
\Sigma=\Big\{P\in X_{a}\ \text{such that $\Big(X_{a},\frac{2}{n}\mathcal{M}\Big)$ is not canonical at~$P$}\Big\}.
$$
By Corollary~\ref{corollary:36-9-curves}, $\Sigma$ is finite, $\Sigma\cap\mathrm{Sing}(X_a)=\varnothing$, $\Sigma\cap S=\varnothing$, $\Sigma\cap\mathcal{R}_6=\varnothing$. Moreover, if $G=G_{36,10}$, then $\Sigma\cap\big(\mathcal{R}_6^\prime\cup \mathcal{R}_6^{\prime\prime}\big)=\varnothing$.
Furthermore, if $G=G_{36,10}$ and $a\ne 0$, then $\Sigma\cap(\mathcal{K}_6\cup \mathcal{K}_6^\prime)=\varnothing$.

By \cite[Remark 3.6]{CheltsovSarikyanZhuang2024}, the~singularities of the~log pair $(X_{a},\frac{3}{n}\mathcal{M})$ are not log canonical at every point of $\Sigma$. Let us use this to show a contradiction.

Let $P$ be a point in $\Sigma$, and let $\Sigma_P$ be the $G$-orbit of the point $P$. Choose $\mu\in\mathbb{Q}_{>0}$ such that the~singularities of the log pair $(X_{a},\mu\mathcal{M})$ are strictly log canonical at $P$. Then $\mu<\frac{3}{n}$. Recall that 
$$
\mathrm{Nklt}(X_{a},\mu\mathcal{M})=\big\{P\in X_{a}\ \text{such that $\big(X_{a},\mu\mathcal{M}\big)$ is not Kawamata log terminal at $P$}\big\}.
$$
Then the locus $\mathrm{Nklt}(X_{a},\mu\mathcal{M})$ does not contain curves that are not disjoint from the $G$-orbit $\Sigma_P$. Indeed, let $M$ and $M^\prime$ be general surfaces in $\mathcal{M}$.
If $Z$ is a $G$-irreducible curve in $\mathrm{Nklt}(X_{a},\mu\mathcal{M})$, then it follows from \cite[Lemma 1.8]{Corti2000} that
$$
\big(M\cdot M^\prime\big)_Z\geqslant\frac{4}{\mu^2}>\frac{4n^2}{9}
$$
which gives $\mathrm{deg}(Z)\leqslant 6$, so it follows from Lemmas~\ref{lemma:36-9-reducible-curves}, \ref{lemma:36-10-reducible-curves} and \ref{lemma:36-9-irreducible-curves} that $Z\cap \Sigma_P=\varnothing$.

We see that the points of the $G$-orbit $\Sigma_P$ are isolated components of the locus $\mathrm{Nklt}(X_{a},\mu\mathcal{M})$. Now, using the~Nadel vanishing theorem as in the~proof of Theorem~\ref{theorem:405-15}, we see that
$$
\big|\Sigma_P\big|\leqslant h^0\big(X_{a},\mathcal{O}_{X_{a}}(H)\big)=5,
$$
which contradicts Lemmas~\ref{lemma:36-9-orbits} and \ref{lemma:36-10-orbits}. This completes the~proof of Theorem~\ref{theorem:36-9}.
\end{proof}

\appendix

\section{Magma codes}
\label{section:Magma}
In this appendix, we present  Magma codes used in the proof of Theorem~\ref{theorem:main} (Main Theorem).

\subsection{Plane free groups acting on Fermat cubic}
\label{subsection:plane-free-groups}
Let $X\subset \mathbb{P}^4$ be the~Fermat cubic \eqref{equation:Fermat}.
Then 
$$
\mathrm{Aut}(X)\simeq C_4^3\rtimes\mathfrak{S}_5,
$$
and the~action of this group on $X$ lifts to $\mathbb{P}^4$, so we consider $\mathrm{Aut}(X)$ as a subgroup in $\mathrm{PGL}_5(\mathbb{C})$.
The~following Magma code goes through all subgroups of the~group $\mathrm{Aut}(X)$, and construct a graph whose vertices are conjugacy classes of subgroups $G\subset\mathrm{Aut}(X)$ such that $\mathbb{P}^4$ contains no $G$-invariant planes (plane free groups), and edges correspond to inclusion of subgroups.
\smallskip

\begin{verbatim}
// Graph of plane-free subgroup classes
// Vertex i = i-th plane-free subgroup conjugacy class representative H_i.
// Directed edge i -> j iff H_i contains a subgroup conjugate to H_j.

K<zeta> := CyclotomicField(60);
w := RootOfUnity(3, K);
G5 := GL(5, K);
S5 := Sym(5);
p12    := PermutationMatrix(K, S5!(1,2));
p12345 := PermutationMatrix(K, S5!(1,2,3,4,5));

// Constructing Aut(X) - group of order 9720
d1 := DiagonalMatrix(K, [w, w^-1, 1, 1, 1]);
d2 := DiagonalMatrix(K, [w, 1, w^-1, 1, 1]);
d3 := DiagonalMatrix(K, [w, 1, 1, w^-1, 1]);
d4 := DiagonalMatrix(K, [w, 1, 1, 1, w^-1]);
AutX := sub< G5 | d1, d2, d3, d4, p12, p12345 >;

// Plane test 
function HasInvariantPlane(H)
    M := GModule(H);
    D := DirectSumDecomposition(M);
    dims := [ Dimension(U) : U in D ];
    sums := {0};
    for d in dims do
        sums := sums join { s + d : s in sums };
    end for;
    return (3 in sums);
end function;

// Plane-free classes 
S := Subgroups(AutX);
PlaneFreeRecs := [* *];
for rec in S do
    if not HasInvariantPlane(rec`subgroup) then
        Append(~PlaneFreeRecs, rec);
    end if;
end for;
n := #PlaneFreeRecs;
PlaneFreeSubs := [ PlaneFreeRecs[i]`subgroup : i in [1..n] ];
PlaneFreeOrders := [ PlaneFreeRecs[i]`order : i in [1..n] ];
print "Plane-free vertex count =", n;

//Locate to which vertex class the subgroup belongs to.
OrderToIdx := AssociativeArray();
for i in [1..n] do
    o := PlaneFreeOrders[i];
    if IsDefined(OrderToIdx, o) then
        Append(~OrderToIdx[o], i);
    else
        OrderToIdx[o] := [i];
    end if;
end for;

function VertexIndexOfSubgroup(K)
    o := #K;
    if not IsDefined(OrderToIdx, o) then
        return 0;
    end if;
    for j in OrderToIdx[o] do
        flag, _ := IsConjugate(AutX, K, PlaneFreeSubs[j]);
        if flag then
            return j;
        end if;
    end for;
    return 0; // should not happen if K is plane-free 
end function;

// Build adjacency sets 
Adj := [* {} : i in [1..n] *];
for i in [1..n] do
    Hi := PlaneFreeSubs[i];
    oi := PlaneFreeOrders[i];
    SHi := Subgroups(Hi); // all subgroup classes inside Hi
        for srec in SHi do
            if srec`order ge oi then 
                continue;
            end if;
            K := srec`subgroup; 
            if HasInvariantPlane(K) then // only consider plane-free K
                continue;
            end if;
        j := VertexIndexOfSubgroup(K);
        if j ne 0 then
            Include(~Adj[i], j);
        end if;
    end for;
end for;

// Create digraph object and print adjacency 
Edges := &cat[ [ [i, j] : j in Adj[i] ] : i in [1..n] ];
D := Digraph< n | Edges >;
print "";
print "Adjacency list (edges i -> j):";
for i in [1..n] do
    printf "%o -> %o\n", i, Sort(SetToSequence(Adj[i]));
end for;

// Print output 
print "";
print "========================================";
print "digraph PlaneFree {";
for i in [1..n] do
    o := PlaneFreeOrders[i];
    lab := Sprintf("%o: |G|=%o", i, o);
    if CanIdentifyGroup(o) then
        lab := lab cat Sprintf("\\nIdGroup=%o", IdentifyGroup(PlaneFreeSubs[i]));
    end if;
    printf " %o [label=\"%o\"];\n", i, lab;
end for;
for e in Edges do
    printf " %o -> %o;\n", e[1], e[2];
end for;
print "}";
print "========================================";
\end{verbatim}

\subsection{Qualifying subgroups of the~group $C_4^3\rtimes\mathfrak{S}_5$}
\label{subsection:qualifying-subgroups}
Let $X\subset \mathbb{P}^4$ be the~Fermat cubic \eqref{equation:Fermat}.
Then 
$$
\mathrm{Aut}(X)\simeq C_4^3\rtimes\mathfrak{S}_5.
$$
The~following Magma code goes through all subgroups of the~group $\mathrm{Aut}(X)$, and construct a graph whose vertices are conjugacy classes of subgroups $G\subset\mathrm{Aut}(X)$ such that $G$ has an irreducible representation of dimension $4$ or $5$ (qualifying subgroups), and the~edges correspond to inclusion.
\smallskip

\begin{verbatim}
// Directed graph of containment among the qualifying subgroup classes
// Edge v_i -> v_j means G_i contains a proper subgroup conjugate in Aut(X) to G_j.

RF := recformat<lattice_elt,class_index,order,class_length,gap_id,deg45,Hperm,Hmat>;

function GAPIdString(H)
    o := Order(H);
    if CanIdentifyGroup(o) then
        id := IdentifyGroup(H);
        return Sprintf("(%o,%o)", id[1], id[2]);
    else
        return "N/A";
    end if;
end function;

function RelevantDegrees45(H)
    return [ t[1] : t in CharacterDegrees(H) | t[1] in {4,5} and t[2] gt 0 ];
end function;

function BuildAutModels()
    K<zeta60> := CyclotomicField(60);
    z3 := zeta60^20;   // primitive cube root of unity
    P4<x0,x1,x2,x3,x4> := ProjectiveSpace(K, 4);
    f := x0^3 + x1^3 + x2^3 + x3^3 + x4^3;
    X := Scheme(P4, [f]);
    S5 := Sym(5);
    d1 := GL(5,K)!DiagonalMatrix(K, [z3, 1, 1, 1, z3^-1]);
    d2 := GL(5,K)!DiagonalMatrix(K, [1, z3, 1, 1, z3^-1]);
    d3 := GL(5,K)!DiagonalMatrix(K, [1, 1, z3, 1, z3^-1]);
    d4 := GL(5,K)!DiagonalMatrix(K, [1, 1, 1, z3, z3^-1]);
    sM := GL(5,K)!PermutationMatrix(K, S5!(1,2));
    tM := GL(5,K)!PermutationMatrix(K, S5!(1,2,3,4,5));
    AutXmat := sub< GL(5,K) | d1, d2, d3, d4, sM, tM >;
    assert Order(AutXmat) eq 9720;
    assert IsInvariant(f, AutXmat);
    S15 := Sym(15);
    c1 := S15!(1,2,3);
    c2 := S15!(4,5,6);
    c3 := S15!(7,8,9);
    c4 := S15!(10,11,12);
    c5 := S15!(13,14,15);
    a1 := c1*c5^-1;
    a2 := c2*c5^-1;
    a3 := c3*c5^-1;
    a4 := c4*c5^-1;
    sP := S15!(1,4)(2,5)(3,6);
    tP := S15!(1,4,7,10,13)(2,5,8,11,14)(3,6,9,12,15);
    AutXperm := sub< S15 | a1, a2, a3, a4, sP, tP >;
    assert Order(AutXperm) eq 9720;
    perm_to_mat := hom< AutXperm -> AutXmat | d1, d2, d3, d4, sM, tM >;
    return X, AutXmat, AutXperm, perm_to_mat;
end function;

function BuildQualifiedData()
    X, AutXmat, AutXperm, perm_to_mat := BuildAutModels();
    classes := SubgroupClasses(AutXperm);
    L := SubgroupLattice(AutXperm);
    good := [* *];
    for i in [1..#classes] do
        Hperm := classes[i]`subgroup;
        deg45 := RelevantDegrees45(Hperm);
        if #deg45 gt 0 then
            Hmat := sub< AutXmat | [ perm_to_mat(g) : g in Generators(Hperm) ] >;
            Append(~good, rec< RF |
                lattice_elt := L!Hperm,
                class_index := i,
                order := classes[i]`order,
                class_length := classes[i]`length,
                gap_id := GAPIdString(Hperm),
                deg45 := deg45,
                Hperm := Hperm,
                Hmat := Hmat>);
        end if;
    end for;
    return X, AutXmat, AutXperm, L, good;
end function;

// Driver 
X, AutXmat, AutXperm, L, good := BuildQualifiedData();
labels := [ Sprintf("v%o", i) : i in [1..#good] ];
S := {@ labels[i] : i in [1..#labels] @};
// Full containment digraph on the selected vertices
edges := [ [ S[i], S[j] ] :
    i, j in [1..#good] |
    i ne j and good[i]`lattice_elt ge good[j]`lattice_elt ];

D := Digraph< S | edges : SparseRep := true >;

printf "Number of vertices: %o\n", #good;
printf "Number of directed edges: %o\n\n", #edges;
print "Vertex legend:";
for i in [1..#good] do
    r := good[i];
    printf "%o : class index = %o, order = %o, class length = %o, 
            GAP ID = %o, irreducible degrees in {4,5} = %o\n",
        labels[i], r`class_index, r`order, r`class_length, r`gap_id, r`deg45;
    print "    generators:";
    for g in Generators(r`Hmat) do
        print g;
    end for;
    print "";
end for;
print "Containment digraph:";
print D;
\end{verbatim}

\subsection{Qualifying subgroups of the~group $((C_3^2\rtimes C_3)\rtimes C_4)\times\mathfrak{S}_3$}
\label{subsection:weird-threefold-qualifying-subgroups}

The following Magma code goes through all subgroups of the~group $((C_3^2\rtimes C_3)\rtimes C_4)\times\mathfrak{S}_3$ with GAP ID [648,541], and constructs a graph whose vertices are conjugacy classes of subgroups $G$ such that $G$ has an irreducible representation of dimension $4$ or $5$ (qualifying subgroups), and the~edges of this graph correspond to inclusion of subgroups.
\smallskip

\begin{verbatim}
// Directed graph of containment among the qualifying subgroup classes
// Edge v_i -> v_j means G_i contains a proper subgroup conjugate in Aut(X) to G_j.

RF := recformat<lattice_elt,class_index,order,class_length,gap_id,deg45,Hperm,Hmat>;

function GAPIdString(H)
    o := Order(H);
    if CanIdentifyGroup(o) then
        id := IdentifyGroup(H);
        return Sprintf("(%o,%o)", id[1], id[2]);
    else
        return "N/A";
    end if;
end function;

function RelevantDegrees45(H)
    return [ t[1] : t in CharacterDegrees(H) | t[1] in {4,5} and t[2] gt 0 ];
end function;

function BuildAutModels()
    K<zeta60> := CyclotomicField(60);
    z3 := zeta60^20;   // primitive cube root of unity
    P4<x0,x1,x2,x3,x4> := ProjectiveSpace(K, 4);
    sq3:=SquareRoot(K!3);
    f := x0^3 + x1^3 + x2^3 + x3^3 + x4^3 + 3*(sq3-1)*x0*x1*x2;
    X := Scheme(P4, [f]);
    S5 := Sym(5);
    A1 := GL(5,K)!PermutationMatrix(K,S5!(1,2,3));
    A2 := GL(5,K)!DiagonalMatrix(K, [1, z3, z3^2, 1, 1]);
    A3 := GL(5,K)!Matrix(K,5,5,[1,1,1,0,0, 
                                1,z3,z3^2,0,0, 
                                1,z3^2,z3,0,0, 
                                0,0,0,sq3,0, 
                                0,0,0,0,sq3])*ScalarMatrix(5,K!(1/sq3));
    A4 := GL(5,K)!PermutationMatrix(K,S5!(4,5));
    A5 := GL(5,K)!DiagonalMatrix(K, [1, 1, 1, z3, z3^2]);
    AutXmat := sub< GL(5,K) | A1,A2,A3,A4,A5 >;
    assert Order(AutXmat) eq 648;
    S32 := Sym(32);
    a1 :=  S32!(9,10,11)(15,16,17)(18,21,23)(19,22,20)(24,29,27)(25,30,28);
    a2 :=  S32!(12,13,14)(15,18,20)(16,19,21)(24,25,26)(27,30,32)(28,31,29);
    a3 := S32!(9,12)(10,13,11,14)(15,24,16,25)(17,26)
              (18,27,21,30)(19,28,20,29)(22,31)(23,32);
    a4 :=  S32!(1,2)(3,4)(5,8)(6,7);
    a5 := S32!(3,5,7)(4,6,8);
    AutXperm := sub< S32 | a1, a2, a3, a4, a5 >;
    assert Order(AutXperm) eq 648;
    perm_to_mat := hom< AutXperm -> AutXmat | A1,A2,A3,A4,A5 >;
    return X, AutXmat, AutXperm, perm_to_mat;
end function;

function BuildQualifiedData()
    X, AutXmat, AutXperm, perm_to_mat := BuildAutModels();
    classes := SubgroupClasses(AutXperm);
    L := SubgroupLattice(AutXperm);
    good := [* *];
    for i in [1..#classes] do
        Hperm := classes[i]`subgroup;
        deg45 := RelevantDegrees45(Hperm);
        if #deg45 gt 0 then
            Hmat := sub< AutXmat | [ perm_to_mat(g) : g in Generators(Hperm) ] >;
            Append(~good, rec< RF |
                lattice_elt := L!Hperm,
                class_index := i,
                order := classes[i]`order,
                class_length := classes[i]`length,
                gap_id := GAPIdString(Hperm),
                deg45 := deg45,
                Hperm := Hperm,
                Hmat := Hmat>);
        end if;
    end for;
    return X, AutXmat, AutXperm, L, good;
end function;

// Driver
X, AutXmat, AutXperm, L, good := BuildQualifiedData();

labels := [ Sprintf("v%o", i) : i in [1..#good] ];
S := {@ labels[i] : i in [1..#labels] @};

// Full containment digraph on the elected vertices
edges := [ [ S[i], S[j] ] :
    i, j in [1..#good] |
    i ne j and good[i]`lattice_elt ge good[j]`lattice_elt ];

D := Digraph< S | edges : SparseRep := true >;

printf "Number of vertices: %o\n", #good;
printf "Number of directed edges: %o\n\n", #edges;
print "Vertex legend:";
for i in [1..#good] do
    r := good[i];
    printf "%o : class index = %o, order = %o, class length = %o, 
            GAP ID = %o, irreducible degrees in {4,5} = %o\n",
        labels[i], r`class_index, r`order, r`class_length, r`gap_id, r`deg45;
    print "    generators:";
    for g in Generators(r`Hmat) do
        print g;
    end for;
    print "";
end for;
print "Containment digraph:";
print D;
\end{verbatim}

\bibliographystyle{alpha}
\bibliography{cubics}

\end{document}